\documentclass[11pt]{amsart}
\usepackage{graphicx}
\usepackage[pagebackref, colorlinks=true, linkcolor=black, citecolor=black, urlcolor=black]{hyperref}

\usepackage{xargs}                      
\usepackage[colorinlistoftodos,prependcaption 
]{todonotes}                                                                                                                 

\newcommandx{\todoN}[2][1=]{\todo[linecolor=red,backgroundcolor=white,bordercolor=red,#1]{#2}}			

\usepackage{comment}
\usepackage{calc}

\usepackage{cite}
\makeatletter
\newcommand{\citecomment}[2][]{\citen{#2}#1\citevar}
\newcommand{\citeone}[1]{\citecomment{#1}}
\newcommand{\citetwo}[2][]{\citecomment[,~#1]{#2}}
\newcommand{\citevar}{\@ifnextchar\bgroup{;~\citeone}{\@ifnextchar[{;~\citetwo}{]}}}
\newcommand{\citefirst}{\@ifnextchar\bgroup{\citeone}{\@ifnextchar[{\citetwo}{]}}}
\newcommand{\cites}{[\citefirst}
\makeatother

\usepackage{amsmath, amsthm, amsfonts, amssymb}
\usepackage{enumerate}
\usepackage{bbm}
\usepackage{mathrsfs}
\usepackage{amssymb}
\usepackage{mathtools}
\usepackage[all]{xy}
\usepackage{tikz}
\usepackage{tikz-cd}
\usetikzlibrary{matrix,arrows,decorations.pathmorphing}
\usepackage{xcolor}
\usepackage{bm}
	\usetikzlibrary{calc}
	\usetikzlibrary{trees}
	\tikzstyle{every picture}=[scale=.35,inner sep=0]
	\usepackage{color}

\newtheorem{thm}{Theorem}

\newtheorem{cor}{Corollary}
\theoremstyle{definition}
\newtheorem{defi}{Definition}

\theoremstyle{theorem}
\newtheorem{theorem}{Theorem}[section]
\newtheorem{lemma}[theorem]{Lemma}
\newtheorem{proposition}[theorem]{Proposition}
\newtheorem{corollary}[theorem]{Corollary}

\theoremstyle{definition}
\newtheorem{definition}[theorem]{Definition}

\theoremstyle{remark}
\newtheorem{remark}[theorem]{Remark}
\newtheorem{claim}[theorem]{Claim}

\numberwithin{equation}{section}

\newcommand{\M}{\mathcal{M}}

\newcommand{\ra}{\rightarrow}
\renewcommand{\P}{\mathbb{P}}

\newcommand{\E}{\mathbb{E}}

\renewcommand{\H}{\mathbb{H}}

\newcommand{\ov}{\overline}

\makeatletter
\@namedef{subjclassname@2020}{\textup{MSC2020}}
\makeatother

\begin{document}

\title[Incidence varieties in the projectivized $k$-th Hodge bundle]{Incidence varieties\\ in the projectivized $k$-th Hodge bundle\\ over curves with rational tails}

\author[I.~Gheorghita]{Iulia Gheorghita}
\address{Iulia Gheorghita  
\newline \indent Department of Mathematics
\newline \indent Boston College, Chestnut Hill, MA 02467}
\email{iuliaisaia@gmail.com}

\author[N.~Tarasca]{Nicola Tarasca}
\address{Nicola Tarasca 
\newline \indent Department of Mathematics \& Applied Mathematics
\newline \indent Virginia Commonwealth University, Richmond, VA 23284}
\email{tarascan@vcu.edu}


\begin{abstract}
Over the moduli space of pointed smooth algebraic curves, the projectivized $k$-th Hodge bundle is the space of $k$-canonical divisors.
The incidence loci are defined by requiring the $k$-canonical divisors to have prescribed multiplicities at the marked points.
We compute the classes of the closure of the incidence loci in the projectivized $k$-th Hodge bundle over the moduli space of curves with rational tails.
The classes are expressed as a linear combination of tautological classes indexed by decorated stable graphs with coefficients enumerating appropriate weightings. As a consequence, we obtain an explicit expression for some relations in tautological rings of moduli  of curves with rational tails.
\end{abstract}

\vspace*{-2pc}

\maketitle



\vspace{-2pc}

\section{Introduction}

This study revolves around  families of algebraic curves with differentials. For an algebraic curve $C$ with at worst simple nodal singularities,
the \textit{stable (holomorphic) abelian differentials} are the global sections of the dualizing sheaf $\omega_C$. Similarly, for $k\geq 1$, the \textit{stable (holomorphic) $k$-differentials} are the global sections of $\omega_C^{\otimes k}$. 
The \textit{\mbox{$k$-th} Hodge bundle} $\mathbb{E}^k_{g,n}$ is the vector bundle of stable $k$-differentials over the moduli space $\ov{\M}_{g,n}$ of stable curves of genus $g$ with $n$ marked points. Its projectivization is here denoted $\mathbb{PE}^k_{g,n}$. 
A point of $\mathbb{PE}^k_{g,n}$ consists of a stable $n$-pointed genus $g$ curve together with the class of a nonzero stable $k$-differential modulo scaling by units.
The fibers of $\mathbb{PE}^k_{g,n}$ over the locus of smooth curves  $\M_{g,n}$  consist of \textit{$k$-canonical divisors}.

The closure of loci of $k$-canonical divisors has been the focus of several recent studies. 
Pointed curves admitting $k$-canonical divisors with prescribed multiplicities at the marked points define natural loci in $\M_{g,n}$.
Similar loci can be obtained by requiring the existence of  a \textit{meromorphic} $k$-differential with prescribed zero and pole orders at the marked points.
The closure of such loci has been studied by Farkas and Pandharipande \cite{farkastwisted} by means of their moduli space of {twisted canonical divisors}. This is a proper  space inside $\ov{\M}_{g,n}$ and contains the space of canonical divisors together with some additional components supported on singular curves. 
A conjecture for the weighted sum of the fundamental classes of all such components in the meromorphic case was proposed by Janda, Pandharipande, Pixton,  Zvonkine \cite[Appendix]{farkastwisted}, together with an algorithm to determine the class of the closure of loci of canonical divisors in $\ov{\M}_{g,n}$ in both the holomorphic and meromorphic case. Extended to $k\geq 1$ by Schmitt \cite{schmitttwisted}, this conjecture has  been proved by Bae, Holmes, Pandharipande, Schmitt,  Schwarz \cite{baepixton}.

It follows that the class of the closure of loci of $k$-canonical divisors in $\ov{\M}_{g,n}$ is determined by the algorithm given in \cite[Appendix]{farkastwisted} and \cite{schmitttwisted}. 
The problem of computing a \textit{closed formula} for such classes remains open.
We make  progress on this problem by computing a closed formula for the classes of the closure of loci of $k$-canonical divisors in the moduli space ${\M}_{g,n}^{\rm rt}$ of curves with rational tails. The space ${\M}_{g,n}^{\rm rt}$ is a partial compactification of ${\M}_{g,n}$ and consists of stable pointed curves with a component of geometric genus $g$ and possibly additional rational components.

More generally, we  obtain a  stronger result by computing a closed formula for the class of the closure of incidence varieties inside the restriction of $\mathbb{PE}^k_{g,n}$ over ${\M}_{g,n}^{\rm rt}$. Specifically, define the \textit{incidence locus}
\begin{equation}
\label{eq:Hkgndef}
{\H}^k_{g,n} \subset \mathbb{PE}^k_{g,n}
\end{equation}
as the locus consisting of  smooth $n$-pointed genus $g$ curves together with the class of a stable $k$-differential having a zero at each marked point. This locus has codimension $n$ in $\mathbb{PE}^k_{g,n}$.
Similarly, we consider the incidence loci obtained by imposing the vanishing of the stable $k$-differential at the marked points with \textit{prescribed multiplicities}, see \eqref{eq:Hkgm}. 

Starting from  a recursive description  of the closure of the incidence loci in $\mathbb{PE}^k_{g,n}$ established by Sauvaget \cite{sauvagetcohomology, sauvaget2020volumes},
 we solve the recursion by finding a closed formula (Definition \ref{def:graphfor}) in the Chow ring of $\mathbb{PE}^k_{g,n}$ which expresses the class of such loci in the restriction of $\mathbb{PE}^k_{g,n}$ over ${\M}_{g,n}^{\rm rt}$ in all cases when $k=1$ and with few exceptions when $k\geq 2$ (Theorems \ref{thm:graphFor} and~\ref{thm:graphFor-m}).

Via pull-back, $A^*\left( \ov{\M}_{g,n} \right)$ injects in the Chow ring of $\mathbb{PE}^k_{g,n}$, and one has
\begin{equation}
\label{eq:APE}
A^*\left( \mathbb{PE}^k_{g,n} \right) \cong A^*\left( \ov{\M}_{g,n} \right) \left[\eta \right] \left/ \left(\sum_{i=0}^r (-\eta)^i c_{r-i}\left( \mathbb{E}^k_{g,n} \right) \right) \right.
\end{equation}
where $r=\mathrm{rank}\left(\mathbb{E}^k_{g,n}\right)$ and
$\eta:=c_1\left( \mathscr{O}_{\mathbb{PE}^k_{g,n}}(-1)\right)$.

Our formula (Definitions \ref{def:graphfor} and \ref{def:Fm}) lies in the \textit{tautological ring} $R^*\left(\mathbb{PE}^k_{g,n}\right)$ which is the subring of \eqref{eq:APE} generated by the tautological ring $R^*\left( \ov{\M}_{g,n} \right)$ and $\eta$. The formula is expressed as a linear combination of tautological classes indexed by \textit{decorated stable graphs}, with coefficients enumerating appropriate \textit{weightings}, which we describe next.

\subsection{Dual graphs of pointed curves with rational tails}
For $g\geq 0$,
let $\mathsf{G}^\mathsf{rt}_{g,n}$ be the set of trees that are dual to the locally closed strata in  $\M_{g,n}^\mathsf{rt}$ consisting of  pointed nodal curves with the same topological type. In particular, an element in $\mathsf{G}^\mathsf{rt}_{g,n}$ is a tree consisting of a \textit{genus $g$ vertex} and possibly additional \textit{rational} (i.e., \textit{genus~$0$}) \textit{vertices} dual to curve components, edges dual to nodes, and $n$ legs dual to the $n$ marked points.  

For a tree $\Gamma$ in $\mathsf{G}^\mathsf{rt}_{g,n}$, let $\mathrm{V}(\Gamma)$ be the set of vertices, $\mathrm{L}(\Gamma)$ be the set of legs, and $\mathrm{E}(\Gamma)$ be the set of edges of $\Gamma$. The set $\mathrm{L}(\Gamma)$ is often identified with a set of labels, e.g., $\mathrm{L}(\Gamma) \cong \{1,\dots, n\}$.

\subsection{Rooted trees}
Let $g>0$ and $\Gamma\in \mathsf{G}^\mathsf{rt}_{g,n}$. By identifying the genus $g$ vertex as \textit{root}, the graph $\Gamma$ will be considered as a \textit{rooted tree} with all edges oriented away from the root.
An edge $e\in \mathrm{E}(\Gamma)$ consists of a pair of two half-edges $e=(h^+,h^-)$: we denote by $h^+$ (or \textit{head}) the half-edge which is further away from the genus $g$ vertex, and by $h^-$ (or \textit{tail}) the half-edge which is closer to the genus $g$ vertex. 

The set $\mathrm{H}(\Gamma)$ of half-edges of $\Gamma$ is partitioned as
\[
\mathrm{H}(\Gamma)=\mathrm{H}^+(\Gamma) \sqcup \mathrm{H}^-(\Gamma)
\]
where $\mathrm{H}^+(\Gamma)$ (respectively, $\mathrm{H}^-(\Gamma)$) is the set of head (resp., tail) half-edges.

\subsection{Decorations and weightings}
\label{sec:fordec}
Select $g>0$ and $\Gamma\in \mathsf{G}^\mathsf{rt}_{g,n}$. 
Let $v_g$ be the genus $g$ vertex of $\Gamma$, and $\mathrm{V}_0(\Gamma):= \mathrm{V}(\Gamma)\setminus \{v_g\}$ be the set of rational vertices of $\Gamma$.
Consider the moduli space 
\begin{equation}
\label{eq:MGamma}
\ov{\M}_\Gamma := \ov{\M}_{g,n(v_g)}  \times \prod_{v\in \mathrm{V}_0(\Gamma)} \ov{\M}_{0, n(v)} 
\end{equation}
where $n(v)$ is the valence of a vertex $v$, i.e., the number of legs and half-edges incident to $v$.

The \textit{set of decorations} $\Psi(\Gamma)$ of the graph $\Gamma$ is  defined here as
\[
\Psi(\Gamma):=\left\{\bm{\psi}=\prod_{h\in \mathrm{H}(\Gamma)} \psi_h^{d_h} \,: \, d_h\geq 0 \mbox{ for all }h
\right\} \subset A^*\left( \ov{\M}_\Gamma \right).
\]
For each $h$, the class $\psi_h$ is the first Chern class of the cotangent line bundle at the marked point corresponding to $h$ \cite{graber2003constructions}.

Select $\bm{\psi}=\prod_{h\in \mathrm{H}(\Gamma)} \psi_h^{d_h}$ in $\Psi(\Gamma)$, and
define the following set of pairs: 
\begin{eqnarray}
\label{eq:EHTpsi}
\begin{split}
\mathrm{H}(\Gamma, \bm{\psi}) &:=&& \{ (h, e) \, | \, h\in \mathrm{H}^+(\Gamma), \, 0\leq e \leq d_h  \}\\
&&& \sqcup \{ (h, e) \, | \, h\in \mathrm{H}^-(\Gamma), \, 1\leq e \leq d_h  \}.
\end{split}
\end{eqnarray}
One has 
\[
|\mathrm{H}(\Gamma, \bm{\psi})|=|\mathrm{E}(\Gamma)|+\deg \bm{\psi}.
\]

For a half-edge $h$ of $\Gamma$, let $\iota(h)$ be the half-edge such that $(h, \iota(h))\in \mathrm{E}(\Gamma)$.
Define the \textit{capacity} of a head half-edge $h^+$ in $\mathrm{H}^+(\Gamma)$ as
\begin{equation}
\label{eq:lenheads}
\ell_{h^+} := 
\left\{
\begin{array}{l}
\mbox{number of legs of $\Gamma$ in the maximal}\\
\mbox{connected subtree containing $h^+$, but not $\iota(h^+)$.}
\end{array}
\right.
\end{equation}
One has $\ell_{h^+}\geq 2$ for all $h^+$.

A \textit{weighting} of the decorated stable  graph  $(\Gamma, \bm{\psi})$ is a function
\[
w \colon  \mathrm{H}(\Gamma, \bm{\psi}) \longrightarrow \mathbb{N}
\]
satisfying the following conditions:
\begin{align}
\label{eq:conditionsiweight}
\left\{
\begin{array}{l}
\ell_{h^+} -1 \geq  w(h^+,0) \geq w(h^+,1) \geq \cdots \geq w(h^+,d_{h^+})   \\ [.2cm]
 \qquad\mbox{for all heads $h^+$,}\\ [.2cm]
1 \leq w(h^-,1) < \cdots <  w(h^-,d_{h^-}) < w(\iota(h^-),0) \\ [.2cm]
 \qquad\mbox{for all tails $h^-$.}
\end{array}
\right.
\end{align}
The set of {weightings} of the decorated stable  graph  $(\Gamma, \bm{\psi})$ 
\begin{equation}
\label{eq:Witipsi}
\mathsf{W}_{\Gamma, \bm{\psi}} := \{ \mbox{{weightings} of $(\Gamma, \bm{\psi})$}\}
\end{equation}
 is a finite set.
For  $w\in \mathsf{W}_{\Gamma, \bm{\psi}}$,  the corresponding \textit{weight} of $(\Gamma, \bm{\psi})$ is
\begin{equation}
\label{eq:iweighttipsi}
w(\Gamma, \bm{\psi}) := \prod_{(h,e)\in \mathrm{H}(\Gamma, \bm{\psi})} w(h,e).
\end{equation}
Define 
\begin{equation}
\label{eq:citipsi}
{c}_{\Gamma, \bm{\psi}} := \sum_{w\in \mathsf{W}_{\Gamma, \bm{\psi}}} w(\Gamma, \bm{\psi}).
\end{equation}
For instance, one has
\begin{align*}
(\Gamma, \bm{\psi}) &=
\begin{tikzpicture}[baseline={([yshift=-.5ex]current bounding box.center)}]
      \path(0,0) ellipse (2 and 1.7);
      \tikzstyle{level 1}=[counterclockwise from=-90,level distance=15mm,sibling angle=120]
      \node [draw,circle,inner sep=2.5] (A0) at (180:2) {$\scriptstyle{g}$};
      \tikzstyle{level 1}=[counterclockwise from=-60,level distance=12mm,sibling angle=60]
      \node [draw,circle,fill] (A1) at (0:2) {}
            child {node [label=-60: {}]{}}
	    child {node [label=0: {}]{}}
    	    child {node [label=60: {}]{}};
      \path (A0) edge []  node[auto, above=0.5, label={180:}]{}
      				  node[near end, below=0.1, label={-90:$\psi$}]{} (A1);
    \end{tikzpicture}
& \mapsto && {c}_{\Gamma, \bm{\psi}} &=7,\\
(\Gamma, \bm{\psi}) &=
\begin{tikzpicture}[baseline={([yshift=-.5ex]current bounding box.center)}]
      \path(0,0) ellipse (2 and 1.7);
      \tikzstyle{level 1}=[counterclockwise from=-90,level distance=15mm,sibling angle=120]
      \node [draw,circle,inner sep=2.5] (A0) at (180:2) {$\scriptstyle{g}$};
      \tikzstyle{level 1}=[counterclockwise from=-60,level distance=12mm,sibling angle=60]
      \node [draw,circle,fill] (A1) at (0:2) {}
            child {node [label=-60: {}]{}}
	    child {node [label=0: {}]{}}
    	    child {node [label=60: {}]{}};
      \path (A0) edge []  node[near start, above=0.1, label={90:$\psi$}]{}
      				  node[near end, below=0.1, label={-90:$\psi$}]{} (A1);
    \end{tikzpicture}
& \mapsto && {c}_{\Gamma, \bm{\psi}} &=6,\\
(\Gamma, \bm{\psi}) &=
\begin{tikzpicture}[baseline={([yshift=-.5ex]current bounding box.center)}]
      \path(0,0) ellipse (2 and 1.7);
      \tikzstyle{level 1}=[counterclockwise from=-60, level distance=12mm, sibling angle=60]
      \node [draw,circle,fill] (A0) at (0:2.5) {}
	    child  {node [label=-60: {$\scriptstyle{}$}]{}}
	    child  {node [label=0: {$\scriptstyle{}$}]{}}
    	    child  {node [label=60: {$\scriptstyle{}$}]{}};
       \tikzstyle{level 1}=[counterclockwise from=-90,level distance=12mm,sibling angle=60]
        \node [draw,circle,fill] (A1) at (0:0) {}
    	    child {node [label=-90: {$\scriptstyle{}$}]{}};
       \tikzstyle{level 1}=[counterclockwise from=120,level distance=15mm,sibling angle=120]
             \node [draw,circle,inner sep=2.5] (A2) at (180:3) {$\scriptstyle{g}$};
        \path (A0) edge [] node[near start, above=0.1, label={90:$\psi$}]{} (A1);
        \path (A1) edge [] node[]{} (A2);        
    \end{tikzpicture}
    & \mapsto && {c}_{\Gamma, \bm{\psi}} &= 42.
\end{align*}
Here and throughout, rational vertices are contracted into points for convenience, and  $\psi$-classes decorate their corresponding half-edges (later in the paper, $\psi$-classes will decorate legs too).

\subsection{The graph formula}
For $g>0$, a graph $\Gamma\in \mathsf{G}^\mathsf{rt}_{g,n}$ defines the moduli space $\ov{\M}_\Gamma$ from \eqref{eq:MGamma} together with a gluing map of degree one
$\ov{\M}_\Gamma \rightarrow \ov{\M}_{g,n}$.
Similarly, define $\mathbb{P}\mathbb{E}^k_\Gamma$ as the fiber product
\begin{equation}
\label{eq:xiGamma}
\begin{tikzcd}
\mathbb{P}\mathbb{E}^k_\Gamma \arrow[rightarrow]{r}{\xi_\Gamma} \arrow[rightarrow]{d}& \mathbb{P}\mathbb{E}^k_{g,n} \arrow[rightarrow]{d}\\
\ov{\M}_\Gamma \arrow[rightarrow]{r} & \ov{\M}_{g,n}.
\end{tikzcd}
\end{equation}
The image of $\xi_\Gamma$ has codimension equal to $|\mathrm{E}(\Gamma)|$ in $\mathbb{P}\mathbb{E}^k_{g,n}$. 
We set $\eta:=\xi_\Gamma^*(\eta)$ in $A^*\left(\mathbb{P}\mathbb{E}^k_\Gamma \right)$, so that for any cycle $\alpha$, one has \mbox{$\xi_{\Gamma*}(\alpha\cdot \eta)=\xi_{\Gamma*}(\alpha)\cdot \eta$}  by the projection formula. We will use few more classes, introduced below.

For a marked point $\ell$, define $\omega_\ell:= \rho_\ell^*(\psi)$ in $R^*\left(\ov{\M}_{g,n}\right)$, where $\psi$ is the $\psi$-class in~$R^*\left(\ov{\M}_{g,1}\right)$ and $\rho_\ell\colon \ov{\M}_{g,n}\rightarrow \ov{\M}_{g,1}$ is the map obtained by forgetting all but the \mbox{$\ell$-th} marked point. Define $\omega_\ell:=\xi_\Gamma^*(\omega_\ell)$ in $A^*\left(\ov{\M}_\Gamma\right)$.

For $h\in \mathrm{H}(\Gamma)$, define $\omega_h:=\xi_\Gamma^*(\omega_\ell)$ in~$A^*\left(\ov{\M}_\Gamma\right)$  for a leg $\ell$ which is further away from the root than $h$. This definition is independent of the choice of such a leg $\ell$ (see  \cite[\S 1.2]{CaT}).

Finally, for $\bm{\psi}=\prod_{h\in \mathrm{H}(\Gamma)} \psi_h^{d_h}$ in $\Psi(\Gamma)$, let
\[
\bm{\beta}_{\Gamma, \bm{\psi}}:=
\prod_{h\in \mathrm{H}^+(\Gamma)} (k \,\omega_{h}-\eta)^{1+d_{h}}
\prod_{h\in \mathrm{H}^-(\Gamma)} (k \,\omega_{h}-\eta)^{d_{h}}\in A^*\left(\mathbb{P}\mathbb{E}^k_\Gamma \right).
\]
One has $\deg \bm{\beta}_{\Gamma, \bm{\psi}} = |\mathrm{E}(\Gamma)|+\deg \bm{\psi}$. 
We are now ready for the graph formula:

\begin{defi}
\label{def:graphfor}
Define $\mathsf{F}^k_{g,n}\in R^n\left(\mathbb{P}\mathbb{E}^k_{g,n} \right)$ as
\[
\mathsf{F}^k_{g,n}:=
\sum_{\Gamma\in \mathsf{G}^\mathsf{rt}_{g,n}} \, (-1)^{|\mathrm{E}(\Gamma)|}  \, {\xi_{\Gamma}}_* \left[ \prod_{\ell\in\mathrm{L}(\Gamma)}  (k \,\omega_\ell-\eta) 
\sum_{\bm{\psi}\in \Psi(\Gamma)} {c}_{\Gamma, \bm{\psi}}\, \bm{\psi} \, \bm{\beta}_{\Gamma, \bm{\psi}}^{-1}
\right].
\]
\end{defi}

To verify that  $\mathsf{F}^k_{g,n}$ has degree $n$, observe that the product of the divisor classes $k \,\omega_\ell-\eta$ for $\ell=1,\dots,n$ has degree $n$, and for each $(\Gamma, \bm{\psi})$, the factor $\bm{\beta}_{\Gamma, \bm{\psi}}^{-1}$ cancels out some of the factors $k \,\omega_\ell-\eta$
to offset the codimension  of the image of $\xi_\Gamma$ and the degree of $\bm{\psi}$.
Remarkably, the sum of the contributions given by $(\Gamma, \bm{\psi})$ such that $\deg \, \bm{\beta}_{\Gamma, \bm{\psi}} > n$ vanishes. This is nontrivial, and follows from Theorem \ref{thm:van}.
For example, when $n\in \{1,2\}$, one has
\begin{align*}
\mathsf{F}^k_{g,1} & =
\mbox{
\begin{tikzpicture}[baseline={([yshift=-.5ex]current bounding box.center)}]
      \path(0,0) ellipse (2 and 2);
      \tikzstyle{level 1}=[counterclockwise from=30,level distance=15mm,sibling angle=120]
      \node [draw,circle,inner sep=2.5] (A0) at (0:0) {$\scriptstyle{g}$}
	    child {node [label=30: {$k\,\omega_1 - \eta$}]{}};
    \end{tikzpicture}
} && \in R^1\left(\P\E^k_{g,1} \right), \\
\mathsf{F}^k_{g,2} & = 
\mbox{
\begin{tikzpicture}[baseline={([yshift=-.5ex]current bounding box.center)}]
      \path(0,0) ellipse (2 and 2);
      \tikzstyle{level 1}=[counterclockwise from=30,level distance=15mm,sibling angle=120]
      \node [draw,circle,inner sep=2.5] (A0) at (0:0) {$\scriptstyle{g}$}
            child {node [label=30: {$k\,\omega_2 - \eta$}]{}}
	    child {node [label=150: {$k\,\omega_1 - \eta$}]{}};
    \end{tikzpicture}
}
-
\mbox{
\begin{tikzpicture}[baseline={([yshift=-.5ex]current bounding box.center)}]
      \path(0,0) ellipse (2 and 2);
      \tikzstyle{level 1}=[counterclockwise from=-60,level distance=12mm,sibling angle=120]
      \node [draw,circle,inner sep=2.5] (A0) at (180:2) {$\scriptstyle{g}$};
      \tikzstyle{level 1}=[counterclockwise from=-60,level distance=12mm,sibling angle=120]
      \node [draw,circle,fill] (A1) at (0:2) {$\scriptstyle{}$}
            child {node [label=-60: {$\scriptstyle{2}$}]{}}
	    child {node [label=60: {$\scriptstyle{1}$}]{}};
      \path (A0) edge [] node[auto, above=0.5, label={180:${k\,\omega-\eta}$}]{} (A1);
    \end{tikzpicture}
} && \in R^2\left(\P\E^k_{g,2} \right). 
\end{align*}

Already for $n=3$, the expansion of the cycle $\mathsf{F}^k_{g,n}$ is lengthy. We include this and more examples in \S\ref{sec:ex}. 
Our first result about abelian differentials is:

\begin{thm}
\label{thm:graphFork1}
For $k=1$ and $g\geq 2$, one has
\[
\left[\ov{\H}^1_{g,n} \right] = \mathsf{F}^1_{g,n} \qquad \in A^{n}\left( \P\E^1_{g,n}\big|_{\M_{g,n}^\mathsf{rt}}  \right).
\]
\end{thm}

The statement holds for all $n$. Since abelian differentials of genus $g$ curves have degree $2g-2$, the loci $\ov{\H}^1_{g,n}$ are empty for $n>2g-2$. In this case, Theorem \ref{thm:graphFork1} gives relations in the tautological ring, as discussed in the next Corollary \ref{cor:tautrel}.
More generally, for $k\geq 1$, we prove:

\begin{thm}
\label{thm:graphFor}
For $g\geq 2$, one has
$
\left[\ov{\H}^k_{g,n} \right] = \mathsf{F}^k_{g,n}
$
in: 
\begin{enumerate}[(i)]
\item $A^{n}\left( \P\E^k_{g,n}\big|_{\M_{g,n}^\mathsf{rt}}  \right)$ when
\[
\left\{
\begin{array}{l}
\mbox{$k=1$; or}\\[.1cm]
\mbox{$k\geq 2$ and $n \leq k(2g-2)-1$;}
\end{array}
\right.
\]
\item $A^{n}\left( \P\E^k_{g,n} \right)$ when $n\leq k$.
\end{enumerate}
\end{thm}

In the case $k\geq 2$ and $n=k(2g-2)$ left out by the statement, the cycle $\mathsf{F}^k_{g,n}$ decomposes as the sum of the classes of $\ov{\H}^k_{g,n}$ and an additional locus. We discuss this case in Remarks \ref{rmk:1} and \ref{rmk:2}.

Our most general result is described in \S\ref{sec:proofgraphFor} and concerns the study of incidence loci with prescribed multiplicities.
Namely, for an $n$-tuple $\bm{m}=(m_1,\dots,m_n)$ of positive integers, let ${\H}^k_{g,\bm{m}} \subset \mathbb{PE}^k_{g,n}$ be the incidence locus consisting of  smooth $n$-pointed genus $g$ curves  together with the class of a  $k$-differential having a zero of order at least $m_\ell$ at the $\ell$-th marked point for each $\ell=1,\dots, n$. 
Theorem \ref{thm:graphFor-m} shows that the class of $\ov{\H}^k_{g,\bm{m}}$ in the appropriate Chow ring determined by $k$, $g$,  $\bm{m}$ is given by the cycle $\mathsf{F}^k_{g,\bm{m}}$ from Definition \ref{def:Fm} with few exceptions. Theorem~\ref{thm:graphFor} is the specialization of Theorem \ref{thm:graphFor-m} at $\bm{m}=(1,\dots,1)$.

\subsection{Tautological relations}
For $k=1$ and $n>2g-2$, Theorem \ref{thm:graphFor} gives a  vanishing statement:

\begin{cor}
\label{cor:tautrel}
For $g\geq 2$ and  $n>2g-2$, one has
\[
\mathsf{F}^1_{g,n} =0 \qquad \mbox{in } R^n\left( \P\E^1_{g,n}\big|_{\M_{g,n}^\mathsf{rt}} \right).
\]
\end{cor}

This statement provides explicit expressions for tautological  relations.
By construction, these are  restrictions of  relations in $R^n\left( \P\E^1_{g,n} \right)$.
Moreover, Corollary \ref{cor:tautrel} induces relations in  $R^*\left( {\M_{g,n}^\mathsf{rt}} \right)$. 
E.g., for $g=2$ and $n=3$, one recovers  the restriction to $R^2\left( {\M}_{2,3}^\mathsf{rt} \right)$ of the relation  from~\cite{belorousski2000descendent}, see \S\ref{sec:BP}.

Determining the complete space of relations in $R^*\left( \ov{\M}_{g,n} \right)$ is an  open problem \cite{MR1722541, pixton2012conjectural, MR3264769, janda2015relations, pandharipande2015calculus}. 
In the  rational tails case, the tautological ring has been studied in \cite{zbMATH06493559, MR2120989, tavakol2014tautological}.
It is nontrivial to verify whether  all the relations from Corollary \ref{cor:tautrel} are  in the space of known relations, or 
whether Corollary \ref{cor:tautrel}  contributes any new relations. 
For this, it would be desirable to undertake a numerical study of the relations from Corollary \ref{cor:tautrel}.

\subsection{Structure of the paper}
We start with  a recursive identity (Theorem \ref{thm:pinZnrhonZ1plus}) for the class of the closure of the incidence locus \eqref{eq:Hkgndef}  over ${\M}_{g,n}^{\rm rt}$ in \S\ref{sec:recidincinit}  
following \cite{sauvagetcohomology, sauvaget2020volumes} (see also \S\S\ref{sec:Hkg2}--\ref{sec:recidinc}).
Next, we show how the cycle $\mathsf{F}^k_{g,n}$ from Definition \ref{def:graphfor} solves the recursion.
For this, we first analyze certain cycles on moduli spaces of pointed rational curves related to $\mathsf{F}^k_{g,n}$
 in \S\S\ref{sec:cycleswrt}--\ref{sec:recid}, and show how they satisfy a recursive identity (Theorem \ref{thm:Decrec}). This will be used to conclude that the cycles $\mathsf{F}^k_{g,n}$ satisfy the same recursion as the classes of the closure of the incidence loci  over ${\M}_{g,n}^{\rm rt}$ (Theorem \ref{thm:Frec}). 
The proof of Theorem \ref{thm:graphFor} is given in \S\ref{sec:proofgraphFor} together with the proof of the more general statement about the incidence loci $\ov{\H}^k_{g,\bm{m}}$ with prescribed multiplicities.  We conclude with some examples and checks~in~\S \ref{sec:ex}.

\subsection*{Acknowledgements}
Conversations with Dawei Chen and Renzo Cavalieri and their collaborations on related projects \cite{CT, CT2, CaT} provided the inspiration and the groundwork that made our results possible. Also, we would like to thank Rahul Pandharipande, Georg Oberdieck, and Adrien Sauvaget for comments and questions on a preliminary announcement of this work, and Johannes Schmitt for providing some numerical verification of our formula.


\section{A recursive identity for incidence loci} 
\label{sec:recidincinit}
We present here a recursive identity for the class of the incidence locus $\ov{\H}_{g, n}^k$ over curves with rational tails following \cite{sauvagetcohomology, sauvaget2020volumes}. 

First we define the extra loci $E_{\bm{I}}$. 
For $m\leq n-1$, consider the locus
\begin{equation}
\label{eq:H1heavypt}
{\H}^k_{g,\bm{m}} \subset \mathbb{PE}^k_{g,n-m} \qquad \mbox{where} \qquad \bm{m}=\left(m, 1^{n-m-1}\right)
\end{equation}
consisting of  smooth genus $g$ curves with $n-m$ marked points together with the class of a stable $k$-differential having a zero of order at least $m$ at the first marked point and  zeros at the other marked points.
This locus has codimension $n-1$ in $\mathbb{PE}^k_{g,n-m}$.

For a non-empty $\bm{I}\subseteq \{1,\dots, n-1\}$ of size $|\bm{I}|=m$, let
\begin{equation}
\label{eq:gamma0}
\gamma\colon \mathbb{PE}^k_{g,n-m} \times \ov{\M}_{0,\bm{I}\sqcup\{n,h_n\}} \longrightarrow \mathbb{PE}^k_{g,n}
\end{equation}
be the gluing map of degree one obtained by  identifying the first marked point on elements of $\mathbb{PE}^k_{g,n-m}$ with the marked point $h_n$ of elements in $\ov{\M}_{0,\bm{I}\sqcup\{n,h_n\}}$ (and relabeling the other marked points of elements in $\mathbb{PE}^k_{g,n-m}$ by elements of $\{1,\dots, n-1\}\setminus \bm{I}$ in case $n-m>1$, i.e., $|\bm{I}|<n-1$) and extending stable $k$-differentials by zero on the attached rational tail.

The \textit{extra locus} $E_{\bm{I}}$ is defined as
\begin{equation}
\label{eq:EI}
E_{\bm{I}}:= \gamma_* \, \ov{\H}^k_{g,\bm{m}}  \subset \P\E^k_{g,n}.
\end{equation}
A general element of $E_{\bm{I}}$ consists of a stable pointed curve with a rational tail containing the marked points with labels in $\bm{I}\sqcup\{n\}$ and a genus $g$ component containing the remaining marked points, together with a stable $k$-differential vanishing with order $m$ at the preimage of the node in the genus $g$ component and vanishing at all marked points in the genus $g$ component.

\smallskip

Let $\pi_n\colon \P\E^k_{g,n} \rightarrow \P\E^k_{g,n-1}$ be the map obtained by forgetting the last marked point, and let $\rho_n\colon \P\E^k_{g,n} \rightarrow\P\E^k_{g,1}$ be the map obtained by forgetting all but the last marked point (and relabeling it as $P_1$).

\begin{theorem}[{\cite{sauvagetcohomology, sauvaget2020volumes}}]
\label{thm:pinZnrhonZ1plus}
For $g\geq 2$ and $n\geq 2$, the identity
\[
\pi_n^* \left[\ov{\H}_{g, n-1}^k \right] \cdot \rho_n^* \left[\ov{\H}_{g, 1}^k \right] = \left[\ov{\H}_{g, n}^k \right] + \sum_{\bm{I}} |\bm{I}|[E_{\bm{I}}]
\] 
where the sum is over all non-empty $\bm{I}\subseteq \{1,\dots, n-1\}$, holds in:
\begin{enumerate}[(i)]
\item $A^{n}\left( \P\E^k_{g,n}\big|_{\M_{g,n}^\mathsf{rt}}  \right)$ when
\begin{equation*}
\left\{
\begin{array}{l}
\mbox{$k=1$; or}\\[.1cm]
\mbox{$k\geq 2$ and $n \leq k(2g-2)-1$;}
\end{array}
\right.
\end{equation*}
\item $A^{n}\left( \P\E^k_{g,n} \right)$ when $n\leq k$.
\end{enumerate}
\end{theorem}

The loci ${\H}^k_{g,\bm{m}} \subset \mathbb{PE}^k_{g,n-m}$ from \eqref{eq:H1heavypt} with $m\leq n-1$ are of pure codimension $n-1$ and are obtained from ${\H}^k_{g,n-1} \subset \mathbb{PE}^k_{g,n-1}$ by colliding the first $m$ marked points. Hence, from the class of ${\H}^k_{g,n-1}$, one computes 
 the classes of the loci from \eqref{eq:H1heavypt} by colliding points, and the classes of the extra loci $E_{\bm{I}}$ from \eqref{eq:EI} by push-forward via the corresponding map~$\gamma$. 
It follows that Theorem \ref{thm:pinZnrhonZ1plus} expresses the class of the incidence locus $\ov{\H}_{g, n}^k$ recursively in $n$. The base case $n=1$ is computed in \cites[\S 1.6]{sauvagetcohomology}[\S4]{korotkin2019tau}:
For $k\geq 1$ and $g\geq 2$, one has 
\begin{equation}
\label{eq:Z1class}
\ov{\H}_{g, 1}^k \equiv k \,\omega -\eta \quad \in \mathrm{Pic}\left(\P\E^k_{g,1} \right).
\end{equation}

\begin{remark}
\label{rmk:1}
In the case $k\geq 2$ and $n=k(2g-2)$ left out by Theorem \ref{thm:pinZnrhonZ1plus}, one has
\begin{equation}
\label{eq:k2plusnmax}
\pi_{n}^* \left[\ov{\H}_{g, n-1}^k \right] \cdot \rho_{n}^* \left[\ov{\H}_{g, 1}^k \right] = \left[\ov{\H}_{g, n}^k \right] + \left[ E^{\textrm{ab}} \right] + \sum_{\bm{I}} |\bm{I}|[E_{\bm{I}}]  
\end{equation}
in $A^{n}\left( \P\E^k_{g,n}\big|_{\M_{g,n}^\mathsf{rt}}  \right)$,
where $E^{\textrm{ab}}\subset \P\E^k_{g,n}$ is the locus whose
  general element   has a genus $g$ component, a rational tail containing all marked points, and the $k$-th power of an abelian differential vanishing with order $k(2g-2)$ at the preimage of the node on the genus $g$ component.
\end{remark}

For $k=1$, Theorem \ref{thm:pinZnrhonZ1plus} and Remark \ref{rmk:1} were obtained by Sauvaget \cite{sauvagetcohomology} over $\ov{\M}_{g,n}$ (but no graph formula was given there). Sauvaget informed us that the techniques of \cite{sauvaget2020volumes} could  be used to extend the argument for $k\geq 1$. 
We provide an alternative proof of Theorem \ref{thm:pinZnrhonZ1plus} and Remark \ref{rmk:1} in \S\ref{sec:recidinc} by  applying the incidence variety compactification from \cite{bcggm1, bcggm}.


\section{Cycles from weighted rational trees}
\label{sec:cycleswrt}

Theorem \ref{thm:pinZnrhonZ1plus} gives a recursion satisfied by the classes of the loci $\ov{\H}_{g, n}^k$.
As a preliminary study towards the solution of the recursion, we consider here cycles on moduli spaces of stable pointed rational  curves.  
After introducing the set of  rational rooted trees, we define the cycles $\mathsf{Z}^m(n,i,j)$ and $\mathsf{E}_{\bm{I}}(i,j)$. These will be shown to satisfy a recursive identity in \S\ref{sec:recid}.

\subsection{Rational rooted trees}
\label{sec:RRT}
Let $\mathsf{G}_{0,n+1}= \mathsf{G}^\mathsf{rt}_{0,n+1}$ be the set of trees dual to strata in $\ov{\M}_{0,n+1}$. Relabel the $(n+1)$-th leg of each tree as $h_0$. 

Let $\mathrm{T}\in \mathsf{G}_{0,n+1}$. By identifying the vertex incident to $h_0$ as \textit{root}, the tree $\mathrm{T}$ will be considered as a rooted tree.
An edge $e\in \mathrm{E}(\mathrm{T})$ consists of a pair of two half-edges $e=(h^+, h^-)$, where $h^+$ (or \textit{head}) is the half-edge which is further away from the leg $h_0$, and  $h^-$ (or \textit{tail}) is the half-edge which is closer to the leg $h_0$. Let 
\begin{align*}
\mathrm{H}^+(\mathrm{T}) &:= \{ \mbox{head half-edges $h^+$ of $\mathrm{T}$}\} \sqcup \{h_0\},\\
\mathrm{H}^-(\mathrm{T}) &:= \{ \mbox{tail half-edges $h^-$ of $\mathrm{T}$}\},\\
\mathrm{H}(\mathrm{T}) &:= \mathrm{H}^+(\mathrm{T}) \sqcup \mathrm{H}^-(\mathrm{T}).
\end{align*}

A vertex $v$ of $\mathrm{T}$ is said to be \textit{external}  if there is no tail half-edge of $\mathrm{T}$  incident to $v$.

These definitions are compatible with the ones  in the introduction when $\mathrm{T}$ arises as a rational subtree 
of a graph $\Gamma\in \mathsf{G}^\mathsf{rt}_{g,n^+}$ with $g>0$ and $n^+\geq n$.

\subsection{Trees decorations}
Next, we extend the decorations from \S\ref{sec:fordec} to rational rooted trees. The decorations will account for the special role played by $h_0$. Also, we will allow a heavier weight on one leg. 

Let $\mathrm{T}\in \mathsf{G}_{0,n+1}$.
Given $m\geq 1$,  assign weight $m$ to the first leg of $\mathrm{T}$, and weight $1$ to all other legs. 

For a half-edge $h$ of $\mathrm{T}$, let $\iota(h)$ be the half-edge such that $(h, \iota(h))\in \mathrm{E}(\mathrm{T})$.

For $h\in \mathrm{H}^+(\mathrm{T})$, define the \textit{weighted capacity} of  $h$ as
\[
\ell^m_{h} := \left\{
\begin{array}{ll}
\begin{array}{l}
\mbox{sum of the weights of the legs in the maximal}\\
\mbox{ connected subtree containing $h$, but not $\iota(h)$}
\end{array}
& \mbox{if $h\not= h_0$,}\\
\,\,\, n-1+m & \mbox{if $h= h_0$.}
\end{array}
\right.
\]

Consider the moduli space 
\[
\ov{\M}_ \mathrm{T}  := \prod_{v\in \mathrm{V}(\mathrm{T})} \ov{\M}_{0, n(v)}, 
\]
and define the set of decorations $\Psi(\mathrm{T})$ of the graph $\mathrm{T}$ as
\[
\Psi(\mathrm{T}):=\left\{\bm{\psi}= \psi_1^{d_1} \prod_{h\in \mathrm{H}(\mathrm{T})} \psi_h^{d_h} \,\Bigg| \, 
\begin{array}{l}
d_h\geq 0 \mbox{ for all }h,\\[.2cm]
d_1 \geq 0
\end{array}
\right\} \subset A^*\left(\ov{\M}_ \mathrm{T} \right).
\]
Let $\bm{\psi}=\psi_1^{d_1}\prod_{h\in \mathrm{H}(\mathrm{T})} \psi_h^{d_h}$ in $\Psi(\mathrm{T})$, and define accordingly the set 
\begin{eqnarray*}
\begin{split}
\mathrm{H}(\mathrm{T},\bm{\psi}) &:= && \{ (h, e) \, | \, h\in \mathrm{H}^+(\mathrm{T}), \, 0\leq e \leq d_h  \}\\
&&& \sqcup \{ (h, e) \, | \, h\in \mathrm{H}^-(\mathrm{T}), \, 1\leq e \leq d_h  \}
\sqcup \{ (1, e) \, | \,  1\leq e \leq d_1  \}.
\end{split}
\end{eqnarray*}
One has 
\begin{align}
\label{eq:Htipsisize}
\begin{split}
|\mathrm{H}(\mathrm{T},\bm{\psi})| &= |\mathrm{H}^+(\mathrm{T})| + \deg \bm{\psi} = 1+ |\mathrm{E}(\mathrm{T})| + \deg \bm{\psi}. 
\end{split}
\end{align}
This uses that $\mathrm{H}^+(\mathrm{T})$ contains the head half-edges together with $h_0$,
hence $|\mathrm{H}^+(\mathrm{T})| = 1+ |\mathrm{H}^-(\mathrm{T})| = 1 + |\mathrm{E}(\mathrm{T})|$.

The set of \textit{$i$-weightings of $(\mathrm{T},\bm{\psi})$ compatible with the legs weight} is
\begin{align}
\label{eq:Wimtipsi}
\mathsf{W}^{i,m}_{\mathrm{T},\bm{\psi}} := \left\{ 
\begin{array}{l}
w \colon \mathrm{H}(\mathrm{T},\bm{\psi}) \rightarrow \mathbb{N}  \,: \, \\[.2cm]
\quad w (h_0,0) = i, \\[.2cm]
\quad \ell^m_{h} -1 \geq w(h,0) \geq w(h,1)  \geq \cdots \geq w(h,d_{h})  \\ [.2cm]
\qquad\qquad\mbox{for all heads $h$ or $h=h_0$,}\\ [.2cm]
\quad 1 \leq w(h^-,1) < \cdots <  w(h^-,d_{h^-}) < w(\iota(h^-),0) \\ [.2cm]
\qquad\qquad\mbox{for all tails $h^-$,}\\ [.2cm]
\quad  1\leq w(1,1) < \cdots <  w(1,d_1) < m
\end{array}
\right\}.
\end{align}
For $m=1$, we set
\begin{equation}
\label{eq:Wi1tipsi}
\mathsf{W}^{i}_{\mathrm{T},\bm{\psi}} := \mathsf{W}^{i,1}_{\mathrm{T},\bm{\psi}}.
\end{equation}
One has
\begin{equation}
\label{eq:Wimnonempty}
\mbox{if} \quad \mathsf{W}^{i,m}_{\mathrm{T},\bm{\psi}} \neq \varnothing, \quad \mbox{then} \quad
 i<n-1+m \quad \mbox{and} \quad
 d_1 <m.
\end{equation}
The inequality $ i<n-1+m$ follows from the conditions $i = w (h_0,0)$ and $w (h_0,0) \leq \ell^m_{h_0} -1$. In particular, there are only finitely many non-empty $\mathsf{W}^{i,m}_{\mathrm{T},\bm{\psi}}$ for given $m$ and $\mathrm{T}\in \mathsf{G}_{0,n+1}$, $\bm{\psi} \in \Psi(\mathrm{T})$. The inequality $d_1< m$ follows from the conditions on $w(1,e)$ for $e=1, \dots, d_1$.

For $w\in \mathsf{W}^{i,m}_{\mathrm{T},\bm{\psi}}$, the corresponding \textit{$i$-weight} of $(\mathrm{T},\bm{\psi})$ is 
\begin{equation}
\label{eq:imweighttipsi}
w(\mathrm{T}, \bm{\psi}) := \prod_{(h,e)\in \mathrm{H}(\mathrm{T}, \bm{\psi})} w(h,e).
\end{equation}
Define 
\begin{equation}
\label{eq:cimtipsi}
{c}^{i,m}_{\mathrm{T}, \bm{\psi}} := \sum_{w\in \mathsf{W}^{i,m}_{\mathrm{T}, \bm{\psi}}} w(\mathrm{T}, \bm{\psi}).
\end{equation}
Following \eqref{eq:Wimnonempty}, if this is non-zero, then $i<n-1+m$ and $d_1 <m$.

\subsection{The cycles $\mathsf{Dec}_n^{i,m}(D)$ and $\mathsf{Z}^m(n,i,j)$}
Let $\mathrm{T}\in \mathsf{G}_{0,n+1}$.
Consider the gluing map of degree one defined by $\mathrm{T}$:
\begin{equation}
\label{eq:xiT}
\xi_{\mathrm{T}}\colon  \ov{\M}_ \mathrm{T} \longrightarrow \ov{\M}_{0,n+1}.
\end{equation}
For each $\bm{\psi}\in \Psi(\mathrm{T})$, one has
\begin{equation}
\label{eq:mu1minus1}
\begin{split}
\deg \xi_{\mathrm{T} *} (\bm{\psi})
&= |\mathrm{E}(\mathrm{T})| + \deg \bm{\psi}
= -1+|\mathrm{H}^+(\mathrm{T})| + \deg \bm{\psi}
= -1+ |\mathrm{H}(\mathrm{T}, \bm{\psi})|.
\end{split}
\end{equation}
The last two equalities follow from \eqref{eq:Htipsisize}.

For a formal variable $D$, the cycle $\mathsf{Dec}_n^{i,m}(D)$ in $A^*\left(\ov{\M}_{0,n+1}\right) \left[D^{-1}\right]$ is defined as
\begin{equation}
\label{eq:Decimdef}
\mathsf{Dec}_n^{i,m}(D):=\sum_{{\mathrm{T}}\in \mathsf{G}_{0,n+1}} \sum_{\bm{\psi}\in \Psi(\mathrm{T})} (-1)^{|\mathrm{H}^+(\mathrm{T})|} \,\,
{c}^{i,m}_{\mathrm{T}, \bm{\psi}}\,\, \xi_{\mathrm{T} *} \left( \bm{\psi}  \right)\, D^{-|\mathrm{H}(\mathrm{T}, \bm{\psi})|}.
\end{equation}
This is a  polynomial in $D^{-1}$ with coefficients in $A^*\left(\ov{\M}_{0,n+1}\right)$. It will be convenient to name also its coefficients. For this, define 
\begin{equation}
\label{eq:Znijdef-m}
\mathsf{Z}^m(n,i,j) := \left[ \mathsf{Dec}_n^{i,m}(D)\right]_{1-n-m-j+i} \qquad \mbox{in $A^{*} \left(\ov{\M}_{0,n+1}\right)$}
\end{equation}
where $[X]_{t}$ denotes the coefficient of $D^t$ in $X$.
One has
\begin{equation}
\label{eq:Znij-m}
\mathsf{Z}^m(n,i,j) = \sum_{(\mathrm{T}, \bm{\psi})}  (-1)^{|\mathrm{H}^+(\mathrm{T})|} \,\, c^{i,m}_{\mathrm{T}, \bm{\psi}} \,\, \xi_{\mathrm{T} *}\left( \bm{\psi} \right) \qquad \mbox{in $A^{n+m-2+j-i} \left(\ov{\M}_{0,n+1}\right)$}
\end{equation}
where the sum is over decorated trees $(\mathrm{T}, \bm{\psi})$ 
such that $\xi_{\mathrm{T} *}\left( \bm{\psi} \right)$
has  degree $n+m-2+j-i$ in $A^* \left(\ov{\M}_{0,n+1}\right)$.
When $m=1$, we simply write 
\begin{equation*}
\mathsf{Dec}_n^{i}(D):=\mathsf{Dec}_n^{i,1}(D) \qquad \mbox{and} \qquad \mathsf{Z}(n,i,j):= \mathsf{Z}^1(n,i,j).
\end{equation*}
In Theorem \ref{thm:van2}, we will show that the cycles $\mathsf{Z}(n,i,j)$ vanish for $i,j\geq 1$.

\subsection{Colliding points on rational trees}
Here we show  that the cycles $\mathsf{Z}^m(n,i,j)$ for arbitrary $m$ (and arbitrary $i,j$) are obtained from the case $m=1$ by colliding points. This is a prelude to Theorem \ref{thm:collidinggenusg}.

\begin{theorem}
\label{thm:collidinggenus0}
For $1\leq m< n$, let $n^-:=n-m+1$.
The cycle 
\[
\mathsf{Z}^m(n^-,i,j)\in A^{n-1+j-i} \left(\ov{\M}_{0,n^- +1}\right)
\]
 is obtained  from $\mathsf{Z}(n,i,j)$ in $A^{n-1+j-i} \left(\ov{\M}_{0,n+1}\right)$ by colliding the first $m$ marked points into the marked point $P_1$ --- and relabeling the other marked points as $P_2,\dots,P_{n^-}$.
\end{theorem}

\begin{proof}
The case $m=1$ is trivial. Assume the statement holds for $m-1$ such that $1\leq m-1 <n-1$. It is enough to show that the cycle 
$\mathsf{Z}^{m}(n^-,i,j)$ in $A^{n-1+j-i} \left(\ov{\M}_{0,n^- +1}\right)$
 is obtained  from 
\begin{equation}
\label{eq:Zmminus1}
 \mathsf{Z}^{m-1}(n^- +1,i,j) \in A^{n-1+j-i} \left(\ov{\M}_{0,n^- +2}\right)
\end{equation}
 by colliding the first two marked points into the marked point $P_1$  --- and relabeling the other marked points as $P_2,\dots,P_{n^-}$.

To collide the points, one proceeds as follows. Let $\delta_{1,2}$ be the class of the divisor of stable pointed rational curves whose general element has two rational components with one of them containing only two marked points, namely $P_1$ and $P_2$. The cycle obtained from \eqref{eq:Zmminus1} by colliding the first two marked points is
\begin{equation}
\label{eq:collidingfrommminus1}
\pi_{2 *}\left(  \delta_{1,2} \cdot \mathsf{Z}^{m-1}(n^- +1,i,j) \right) \in A^{n-1+j-i} \left(\ov{\M}_{0,n^- +1}\right)
\end{equation}
where $\pi_2$ is the map obtained by forgetting the second marked point. We need to show that 
\begin{equation}
\label{eq:diffcoll}
\pi_{2 *}\left(  \delta_{1,2} \cdot \mathsf{Z}^{m-1}(n^- +1,i,j) \right) = \mathsf{Z}^m(n^-,i,j)  \in A^{n-1+j-i} \left(\ov{\M}_{0,n^- +1}\right).
\end{equation}
For this, we analyze the various terms arising in the expansion of \eqref{eq:collidingfrommminus1} and show that they match the terms in $\mathsf{Z}^m(n^-,i,j)$.
The non-zero terms arising in the expansion of \eqref{eq:collidingfrommminus1} are obtained by colliding the first two marked points in the terms of \eqref{eq:Zmminus1}
contributed by decorated trees where the first two legs are incident to the same vertex, say $v_1$. We distinguish two cases depending on the valence of $v_1$.

\smallskip

\noindent{\textit{Case 1.}}
Let  $(\widehat{\mathrm{T}}, \bm{\psi})$ be a decorated tree contributing to \eqref{eq:Zmminus1} such that the first two  legs  are incident to a \textit{trivalent} vertex $v_1$. We show that the contribution of $(\widehat{\mathrm{T}}, \bm{\psi})$ to \eqref{eq:collidingfrommminus1} matches the contribution of $(\mathrm{T}, \bm{\psi}\cdot \psi_1)$ to $\mathsf{Z}^m(n^-,i,j)$, where $\mathrm{T}$ is the tree obtained from $\widehat{\mathrm{T}}$ by removing the second leg and contracting the edge incident to the  vertex $v_1$.

Using \eqref{eq:Znij-m}, 
 the contribution of $(\widehat{\mathrm{T}}, \bm{\psi})$ to \eqref{eq:Zmminus1} is given by
\[
(-1)^{|\mathrm{H}^+(\widehat{\mathrm{T}})|} \,\, c^{i,m-1}_{\widehat{\mathrm{T}}, \bm{\psi}} \,\, \xi_{\widehat{\mathrm{T}} *}\left( \bm{\psi} \right).
\]
By  \eqref{eq:cimtipsi}, the coefficient $c^{i,m-1}_{\widehat{\mathrm{T}}, \bm{\psi}}$ is the sum of $\widehat{w}(\widehat{\mathrm{T}}, \bm{\psi})$ for $\widehat{w}$ in $\mathsf{W}^{i,m-1}_{\widehat{\mathrm{T}}, \bm{\psi}}$. 
By linearity, for $\widehat{w}\in \mathsf{W}^{i,m-1}_{\widehat{\mathrm{T}}, \bm{\psi}}$, the contribution of $(\widehat{\mathrm{T}}, \bm{\psi}, \widehat{w})$ to \eqref{eq:Zmminus1} is
\[
(-1)^{|\mathrm{H}^+(\widehat{\mathrm{T}})|} \,\, \widehat{w}(\widehat{\mathrm{T}}, \bm{\psi}) \,\, \xi_{\widehat{\mathrm{T}} *}\left( \bm{\psi} \right).
\]
Colliding the first two marked points gives rise to the following contribution of $(\widehat{\mathrm{T}}, \bm{\psi}, \widehat{w})$ to \eqref{eq:collidingfrommminus1}:
\begin{equation}
\label{eq:contdiffcolltipsipsi1}
(-1)^{|\mathrm{H}^+(\mathrm{T})|} \,\, \widehat{w}(\widehat{\mathrm{T}}, \bm{\psi}) \,\, \xi_{\mathrm{T} *}\left( \bm{\psi}\cdot \psi_1 \right).
\end{equation}
The factor $-\psi_1$ decorating the first leg arises from the intersection with $\delta_{1,2}$ in \eqref{eq:collidingfrommminus1}. 
Next, we show that $\widehat{w}(\widehat{\mathrm{T}}, \bm{\psi}) = w(\mathrm{T}, \bm{\psi}\cdot \psi_1)$ for an appropriate $w\in \mathsf{W}^{i,m}_{\mathrm{T}, \bm{\psi}\cdot \psi_1}$.

Let $h_1^+$ be the head half-edge in $\widehat{\mathrm{T}}$ incident to the  trivalent vertex $v_1$, and let $h_1^-$ be the tail half-edge such that $(h_1^+, h_1^-)\in \mathrm{E}(\widehat{\mathrm{T}})$. A factor $\psi_h$ with $h=h_1^-$ in the decoration $\bm{\psi}$  of $\widehat{\mathrm{T}}$ induces a factor $\psi_1$ decorating  $\mathrm{T}$.
Let $d_1$ be the resulting degree of $\psi_1$ in $\bm{\psi}\cdot \psi_1$ decorating  $\mathrm{T}$. Then the degree of $\psi_h$ with $h=h_1^-$ in the decoration $\bm{\psi}$ of $\widehat{\mathrm{T}}$ is $d_1-1$. 
For $\widehat{w}$ in $\mathsf{W}^{i,m-1}_{\widehat{\mathrm{T}}, \bm{\psi}}$, one has 
\[
\widehat{w}(h_1^+,0) \leq \ell^{m-1}_{h_1^+} -1=m-1.
\]
The conditions on $\widehat{w}$ for tails give
\[
1\leq \widehat{w}(h_1^-,1) < \cdots < \widehat{w}(h_1^-,d_1-1) < \widehat{w}(h_1^+,0).
\]
It follows that $d_1<m$. 
Consider the bijection
\begin{align*}
\mathrm{H}(\widehat{\mathrm{T}}, \bm{\psi}) &\xrightarrow{\cong} \mathrm{H}(\mathrm{T}, \bm{\psi}\cdot \psi_1)\\
(h_1^-,e)&\mapsto (1,e), && \mbox{for  $1\leq e < d_1$,}\\
(h_1^+,0) &\mapsto (1,d_1), \\
(h,e)& \mapsto (h,e), && \mbox{for $h\not\in\{h_1^-, h_1^+\}$ and all $e$.}
\end{align*}
For all heads $h^+\neq h_1^+$ or $h^+=h_0$ in $\widehat{\mathrm{T}}$,
the weighted capacity $\ell^{m-1}_{h^+}$ of $h^+$ in $\widehat{\mathrm{T}}$ is equal to the weighted capacity $\ell^{m}_{h^+}$ of the corresponding $h^+$ in $\mathrm{T}$.
Consider the map
\[
\mathsf{W}^{i,m-1}_{\widehat{\mathrm{T}}, \bm{\psi}} \rightarrow \mathsf{W}^{i,m}_{\mathrm{T}, \bm{\psi}\cdot \psi_1}, \,\, \widehat{w} \mapsto w,
\,\, \mbox{with} \,\,
\left\{
\begin{array}{rll}
w(1,e) \!\!\!\!&:= \widehat{w}(h_1^-,e), &\mbox{if $1\leq e< d_1$,}\\
w(1,d_1) \!\!\!\!&:= \widehat{w}(h_1^+,0), \\
w(h,e) \!\!\!\!&:= \widehat{w}(h,e), &\mbox{otherwise.}
\end{array}
\!\!\!
\right.
\]
This map is a bijection and preserves the weights, i.e., 
\[
\mbox{if}\quad \widehat{w} \mapsto w, \quad\mbox{then}\quad \widehat{w}(\widehat{\mathrm{T}}, \bm{\psi}) = w(\mathrm{T}, \bm{\psi}\cdot \psi_1).
\]
It follows that for $\widehat{w}\in \mathsf{W}^{i,m-1}_{\widehat{\mathrm{T}}, \bm{\psi}}$, the contribution of $(\widehat{\mathrm{T}}, \bm{\psi}, \widehat{w})$ to \eqref{eq:collidingfrommminus1}
given by \eqref{eq:contdiffcolltipsipsi1} is equal to
\[
(-1)^{|\mathrm{H}^+(\mathrm{T})|} \,\, w(\mathrm{T}, \bm{\psi}\cdot \psi_1) \,\, \xi_{\mathrm{T} *}\left( \bm{\psi}\cdot \psi_1 \right)
\]
where $\widehat{w} \mapsto w$. This matches the contribution of $(\mathrm{T}, \bm{\psi}\cdot \psi_1,w)$ to $\mathsf{Z}^m(n^-,i,j)$. 

Vice versa, a decorated tree where $\psi_1$ has positive degree can arise in the expansion of \eqref{eq:collidingfrommminus1} only by colliding the first two marked points  on
 a term of \eqref{eq:Zmminus1} contributed by a decorated tree where the first two legs are incident to a trivalent vertex.

It follows that the contributions to \eqref{eq:collidingfrommminus1} obtained by colliding the first two marked points on decorated trees where the first two legs are incident to a trivalent vertex match the contributions to $\mathsf{Z}^m(n^-,i,j)$ from decorated trees where $\psi_1$ has positive degree.

\smallskip

\noindent{\textit{Case 2.}}
Finally, consider a decorated tree $(\widehat{\mathrm{T}}, \bm{\psi})$ contributing to \eqref{eq:Zmminus1} such that the first two legs  are incident to a  vertex $v_1$ of valence at least four. We can assume that $\psi_1$ has degree zero in $\bm{\psi}$, since otherwise the contribution of $(\widehat{\mathrm{T}}, \bm{\psi})$ to \eqref{eq:collidingfrommminus1} is zero, due to the vanishing $\psi_1\cdot \delta_{1,2}=0$. We show that the contribution of $(\widehat{\mathrm{T}}, \bm{\psi})$ to \eqref{eq:collidingfrommminus1} matches the contribution of $(\mathrm{T}, \bm{\psi})$ to $\mathsf{Z}^m(n^-,i,j)$, where $\mathrm{T}$ is the tree obtained from $\widehat{\mathrm{T}}$ by merging the first two legs.
In this case, the sets of edges of the trees $\widehat{\mathrm{T}}$ and ${\mathrm{T}}$ are isomorphic, and moreover the weighted capacity $\ell^{m-1}_{h^+}$ of $h^+$ in $\widehat{\mathrm{T}}$ is equal to the weighted capacity $\ell^{m}_{h^+}$ of the corresponding $h^+$ in $\mathrm{T}$, for all heads $h^+$ or $h^+=h_0$.
It follows that 
\[
c^{i,m-1}_{\widehat{\mathrm{T}}, \bm{\psi}} = c^{i,m}_{\mathrm{T}, \bm{\psi}}.
\]
Hence, one has
\[
\pi_{2 *}\left( \delta_{1,2}\cdot (-1)^{|\mathrm{H}^+(\widehat{\mathrm{T}})|} \,\, c^{i,m-1}_{\widehat{\mathrm{T}}, \bm{\psi}} \,\, \xi_{\widehat{\mathrm{T}} *}\left( \bm{\psi} \right) \right)
=
(-1)^{|\mathrm{H}^+(\mathrm{T})|} \,\, c^{i,m}_{\mathrm{T}, \bm{\psi}} \,\, \xi_{\mathrm{T} *}\left( \bm{\psi} \right).
\]
The LHS is the contribution of $(\widehat{\mathrm{T}}, \bm{\psi})$ to \eqref{eq:collidingfrommminus1}, and the RHS is the contribution of $(\mathrm{T}, \bm{\psi})$ to $\mathsf{Z}^m(n^-,i,j)$.

We conclude that the contributions to \eqref{eq:collidingfrommminus1} obtained by colliding the first two marked points on decorated trees where the first two legs are incident to a vertex of valence at least four match the contributions to $\mathsf{Z}^m(n^-,i,j)$ from decorated trees where $\psi_1$ has degree zero.
This ends the proof.
\end{proof}

\subsection{The extra cycles \texorpdfstring{$\mathsf{E}_{\bm{I}}$}{E}}
We define here the cycles $\mathsf{E}_{\bm{I}}^i(D)$ and $ \mathsf{E}_{\bm{I}}(i,j)$. These cycles are obtained as push-forward of  $\mathsf{Dec}_n^{i,m}(D)$ and $\mathsf{Z}^m(n,i,j)$, and will be used to express the recursive identities  in Theorems \ref{thm:mainthmgenuszeroij} and \ref{thm:mainthmgenuszero}.

For $\mathrm{T} \in \mathsf{G}_{0,n+1}$, let $v_n$  be the vertex incident to the leg $n$, and let $h_n^+$ be the element in $\mathrm{H}^+(\mathrm{T})$ incident to $v_n$.
If $v_n$ is external, let $\mathrm{T}\setminus v_n$ be the graph obtained from $\mathrm{T}$ by removing the vertex $v_n$ together with the legs incident to it and the element $h_n^+$. 

For $\varnothing \neq \bm{I}\subseteq \{1,\dots,n-1\}$,
define the    \textit{set of decorated codas} as 
\begin{equation}
\label{eq:DCI}
\mathsf{DC}_{\bm{I}}:=\left\{ 
(\mathrm{T}, \bm{\psi}) \, \left|\, 
\begin{array}{l}
\mathrm{T}\in \mathsf{G}_{0,n+1},\\[4pt]
\mbox{$v_n$ is external with set of incident legs $\bm{I}\sqcup \{n\}$,}\\[4pt]
\mbox{and the degree of $\psi_h$ with $h=h_n^+$ in $\bm{\psi}$ is zero}
\end{array}
\right.
\!
\right\} .
\end{equation}
For $\bm{I}= \{1,\dots,n-1\}$, the set $\mathsf{DC}_{\bm{I}}$ consists of a single decorated tree $(\mathrm{T}, 0)$ where $\mathrm{T}$ has only one vertex --- in this case, $\mathrm{T}\setminus v_n=\varnothing$ and $\bm{\psi}=0$.
For $\bm{I}\subsetneq \{1,\dots,n-1\}$, one has $\mathrm{T}\setminus v_n\neq \varnothing$ for $(\mathrm{T},\bm{\psi})\in \mathsf{DC}_{\bm{I}}$. For such decorated trees, 
by relabeling the tail half-edge of the edge connecting $v_n$ to $\mathrm{T}\setminus v_n$ as leg $1$ of $\mathrm{T}\setminus v_n$, and relabeling the other legs of $\mathrm{T}\setminus v_n$ accordingly, one has $\mathrm{T}\setminus v_n\in \mathsf{G}_{0,n-|\bm{I}|+1}$.

Let $(\mathrm{T},\bm{\psi})\in \mathsf{DC}_{\bm{I}}$ for some $\varnothing \neq \bm{I}\subseteq \{1,\dots,n-1\}$. An element $h^+$ in $\mathrm{H}^+(\mathrm{T})$  is called a \textit{predecessor} of $h_n^+$ if  
$h^+$ is incident to a vertex $v\neq v_n$ closer to $h_0$ than $v_n$. 
Consider the subset of $i$-weightings
\[
\mathsf{W}^{i,\bm{I}}_{\mathrm{T},\bm{\psi}} := \left\{ w \in \mathsf{W}^{i}_{\mathrm{T},\bm{\psi}} \, \Bigg| \,
\begin{array}{l}
w \left( h^+, 0 \right)\leq \ell_{h^+}-2  \\[4pt]
\mbox{for all predecessors $h^+$ of $h_n^+$}\\[4pt]
\mbox{and}\,\, w \left( h_n^+, 0 \right)= |\bm{I}| 
\end{array}
\right\}.
\]
Here $\mathsf{W}^{i}_{\mathrm{T},\bm{\psi}}$ is as in \eqref{eq:Wi1tipsi}.
One has
\begin{equation}
\label{eq:WiIneqempty}
\mbox{if} \quad \mathsf{W}^{i,\bm{I}}_{\mathrm{T},\bm{\psi}} \neq \varnothing, \quad \mbox{then} \quad
\left\{
\begin{array}{l}
\mbox{either $i\leq  n-2$ and $|\bm{I}|\leq n-2$,}\\
\mbox{or $i=|\bm{I}|= n-1$.}
\end{array}
\right.
\end{equation}

For $(\mathrm{T},\bm{\psi})\in \mathsf{DC}_{\bm{I}}$,
define 
\begin{equation}
\label{eq:ditipsi}
d^i_{\mathrm{T},\bm{\psi}} := 
\frac{1}{|\bm{I}|} \sum_{w\in \mathsf{W}^{i,\bm{I}}_{\mathrm{T}, \bm{\psi}}} w(\mathrm{T}, \bm{\psi}).
\end{equation}
Following \eqref{eq:WiIneqempty}, one has
\begin{equation}
\label{eq:ditipsineq0}
\mbox{if}\,\, d^i_{\mathrm{T},\bm{\psi}} \neq 0, \quad \mbox{then} \,\,
\left\{
\begin{array}{l}
\mbox{either $i \leq n-2$ and $|\bm{I}|\leq n-2$,}\\
\mbox{or $i=|\bm{I}|= n-1$.}
\end{array}
\right.
\end{equation}
In the case $i=n-1$ and $\bm{I}=\{1,\,\dots, n-1\}$, one has $d^i_{\mathrm{T},0}=1$.

Consider the cycle
\[
\mathsf{E}_{\bm{I}}^i(D)  := 
\sum_{(\mathrm{T}, \bm{\psi})\in \mathsf{DC}_{\bm{I}}}  (-1)^{|\mathrm{H}^+(\mathrm{T})|-1}  \,\, d^i_{\mathrm{T},\bm{\psi}} \,\, \xi_{\mathrm{T} *}\left( \bm{\psi} \right) \, D^{-|\mathrm{H}(\mathrm{T}, \bm{\psi})| }
\]
{in $A^* \left(\ov{\M}_{0,n+1}\right)[D^{-1}]$.}
Similarly to \eqref{eq:Znijdef-m},  let 
\begin{equation}
\label{eq:Eijdef}
 \mathsf{E}_{\bm{I}}(i,j) := \left[ \mathsf{E}_{\bm{I}}^i(D)  \right]_{-n-j+i} \qquad \mbox{in $A^{n-1+j-i} \left(\ov{\M}_{0,n+1}\right)$}
\end{equation}
be its coefficients. One has
\begin{equation}
\label{eq:EIij}
 \mathsf{E}_{\bm{I}}(i,j) = \sum_{(\mathrm{T}, \bm{\psi})}  (-1)^{|\mathrm{H}^+(\mathrm{T})|-1}  \,\, d^i_{\mathrm{T},\bm{\psi}} \,\, \xi_{\mathrm{T} *}\left( \bm{\psi} \right) 
\end{equation}
in $A^{n-1+j-i}\left( \ov{\M}_{0,n+1}\right)$, where the sum is over decorated trees $(\mathrm{T}, \bm{\psi})$ in $\mathsf{DC}_{\bm{I}}$ such that $\xi_{\mathrm{T} *}\left( \bm{\psi} \right)$ has degree $n-1+j-i$. 

\smallskip

Here we show that when $|\bm{I}|\leq n-2$, the cycles $ \mathsf{E}_{\bm{I}}(i,j)$ are obtained as pushforward of cycles $\mathsf{Z}^{m}(n-m,i,j)$ with $m=|\bm{I}|$. Let
\[
\gamma\colon \ov{\M}_{0,n-m+1} \times \ov{\M}_{0,\bm{I}\sqcup\{n,h_n^+\}} \longrightarrow \ov{\M}_{0,n+1}
\]
be the gluing map of degree one obtained by identifying the marked point $P_1$ on elements of $\ov{\M}_{0,n-m+1}$ with the marked point $h_n^+$ of elements in $\ov{\M}_{0,\bm{I}\sqcup\{n,h_n^+\}}$, and relabeling the other marked points of elements in $\ov{\M}_{0,n-m+1}$ by elements of $\{h_0, 1,\dots, n-1\}\setminus \bm{I}$.

\begin{lemma}
\label{lemma:EIisgammapfwd}
Let $\varnothing \neq \bm{I}\subsetneq \{1,\dots,n-1\}$ and $m:=|\bm{I}|\leq n-2$. 
One has
\[
 \mathsf{E}_{\bm{I}}^i(D) = \gamma_*  \, \mathsf{Dec}_{n-m}^{i,m}(D) \qquad \in A^*\left(\ov{\M}_{0,n+1} \right)\left[ D^{-1}\right].
\]
\end{lemma}

\begin{proof}
We show that the coefficients of the two polynomials $ \mathsf{E}_{\bm{I}}^i(D)$ and $\gamma_*  \, \mathsf{Dec}_{n-m}^{i,m}(D)$ in $D^{-1}$ match. Namely, using \eqref{eq:Znijdef-m} and \eqref{eq:Eijdef}, we show 
\begin{equation}
\label{eq:EIijsarepfwds}
 \mathsf{E}_{\bm{I}}(i,j) = \gamma_*  \, \mathsf{Z}^{m}(n-m,i,j) \qquad \in A^{n-1+j-i}\left(\ov{\M}_{0,n+1} \right).
\end{equation}

For $(\mathrm{T}, \bm{\psi})\in \mathsf{DC}_{\bm{I}}$, the gluing map $\xi_{\mathrm{T}}$ factors as $\xi_{\mathrm{T}}=\gamma \circ \xi_{\mathrm{T}\setminus v_n}$, hence
\[
\xi_{\mathrm{T} *}\left( \bm{\psi} \right) = \gamma_*\left(\xi_{\mathrm{T}\setminus v_n \,*}\left( \bm{\psi} \right)\right).
\]
This uses that the degree of $\psi_h$ with $h=h_n^+$ in $\bm{\psi}$ is zero.
Moreover, one has 
\[
\mathrm{H}^+(\mathrm{T}) = \mathrm{H}^+(\mathrm{T}\setminus v_n)\sqcup \{h_n^+\}, 
\]
hence $|\mathrm{H}^+(\mathrm{T})|-1 = |\mathrm{H}^+(\mathrm{T}\setminus v_n)|$.
Let $h_n^-$ be the tail half-edge such that $(h_n^+, h_n^-)\in \mathrm{E}(\mathrm{T})$.
One has a bijection
\begin{align*}
\mathrm{H}(\mathrm{T}, \bm{\psi})\setminus \{(h_n^+,0)\} &\xrightarrow{\cong} \mathrm{H}(\mathrm{T}\setminus v_n, \bm{\psi})\\
(h_n^-,e)&\mapsto (1,e), && \mbox{for  all $e$,}\\
(h,e)& \mapsto (h,e), && \mbox{for $h\not\in\{h_n^-, h_n^+\}$ and all $e$.}
\end{align*}
For all predecessors $h^+$ of $h_n^+$ in $\mathrm{T}$,
the  number $\ell_{h^+}-1$ for $h^+$ in $\mathrm{T}$ is equal to the weighted capacity $\ell^{m}_{h^+}$ of the corresponding $h^+$ in $\mathrm{T}\setminus v_n$.
For all other heads $h^+$ in $\mathrm{T}$, the  capacity $\ell_{h^+}$ of $h^+$ in $\mathrm{T}$ is equal to the weighted capacity $\ell^m_{h^+}$ of the corresponding $h^+$ in $\mathrm{T}\setminus v_n$. It follows that the constraint on $w\in \mathsf{W}^{i,\bm{I}}_{\mathrm{T},\bm{\psi}}$ that $w(h^+,0)\leq \ell_{h^+}-2$ for all predecessors $h^+$ of $h_n^+$ in $\mathrm{T}$ is equivalent to the constraint on $w\in \mathsf{W}^{i,m}_{\mathrm{T}\setminus v_n,\bm{\psi}}$ that $w(h^+,0)\leq \ell^{m}_{h^+}-1$ for the corresponding $h^+$ in $\mathrm{T}\setminus v_n$.

Consider the map
\[
\mathsf{W}^{i,\bm{I}}_{\mathrm{T},\bm{\psi}} \rightarrow \mathsf{W}^{i,m}_{\mathrm{T}\setminus v_n,\bm{\psi}}, \,\, \widehat{w} \mapsto w,
\quad \mbox{with} \quad
\left\{
\begin{array}{rll}
w(1,e) \!\!\!\!&:= \widehat{w}(h_n^-,e), &\\
w(h,e) \!\!\!\!&:= \widehat{w}(h,e), &\mbox{otherwise.}
\end{array}
\right.
\]
This map is a bijection, and since $\widehat{w}(h_n^+,0)=|\bm{I}|$, one has
\[
\mbox{if}\quad \widehat{w} \mapsto w, \quad \mbox{then} \quad \frac{1}{|\bm{I}|}\widehat{w}({\mathrm{T}}, \bm{\psi}) = w(\mathrm{T}\setminus v_n, \bm{\psi}).
\]
Using \eqref{eq:cimtipsi} and \eqref{eq:ditipsi}, it follows that 
\[
d^i_{\mathrm{T},\bm{\psi}} = c^{i,m}_{\mathrm{T}\setminus v_n, \bm{\psi}}.
\]
Hence \eqref{eq:EIij} can be written as
\[
 \mathsf{E}_{\bm{I}}(i,j) = \gamma_* \left(\sum_{(\mathrm{T}^-, \bm{\psi})}  (-1)^{|\mathrm{H}^+(\mathrm{T}^-)|} \,\, c^{i,m}_{\mathrm{T}^-, \bm{\psi}} \,\, \xi_{\mathrm{T}^- *}\left( \bm{\psi} \right)\right)  \in A^{n-1+j-i} \left(\ov{\M}_{0,n+1}\right)
\]
where the sum is over decorated trees $(\mathrm{T}^-, \bm{\psi})$ such that $(\mathrm{T}^-, \bm{\psi})=(\mathrm{T}\setminus v_n, \bm{\psi})$ for some $(\mathrm{T}, \bm{\psi})$ in $\mathsf{DC}_{\bm{I}}$, and such that $\xi_{\mathrm{T}^- *}\left( \bm{\psi} \right)$ has degree $n-2+j-i$. 
Using \eqref{eq:Znij-m}, the statement follows.
\end{proof}


\section{Vanishing cycles on moduli  of pointed rational curves}
\label{sec:vanprt}
Here we show how certain coefficients of the polynomial $\mathsf{Dec}_n^{i}(D)$ introduced in the previous section encode relations in $A^*\left(\ov{\M}_{0,n+1}\right)$.
This will be used in the proof of the recursive relations in \S \ref{sec:recid}.

\smallskip

Recall the cycle $\mathsf{Dec}_n^i(D)=\mathsf{Dec}_n^{i,1}(D)$  from \eqref{eq:Decimdef} defined as
\[
\mathsf{Dec}_n^i(D):=\sum_{{\mathrm{T}}\in \mathsf{G}_{0,n+1}} \sum_{\bm{\psi}\in \Psi(\mathrm{T})} (-1)^{|\mathrm{H}^+(\mathrm{T})|} \,\,
{c}^i_{\mathrm{T}, \bm{\psi}}\,\, \xi_{\mathrm{T} *} \left( \bm{\psi}  \right)\, D^{-|\mathrm{H}(\mathrm{T}, \bm{\psi})|}
\]
in  $A^*\left(\ov{\M}_{0,n+1}\right) \left[D^{-1}\right]$, where the coefficients $c^i_{\mathrm{T}, \bm{\psi}}=c^{i,1}_{\mathrm{T}, \bm{\psi}}$ are as in \eqref{eq:cimtipsi}.

\begin{theorem}
\label{thm:van}
For $i \geq 1$, one has
$
D^{n-i} \, \mathsf{Dec}_n^{i}(D) \in A^*\left(\ov{\M}_{0,n+1}\right) [D].
$
\end{theorem}

First we rewrite this statement as a vanishing statement about the coefficients of the polynomial $\mathsf{Dec}_n^i(D)$.
Recall from \eqref{eq:Znijdef-m} the cycle
\begin{equation*}
\label{eq:Znijdef}
\mathsf{Z}(n,i,j) := \left[\mathsf{Dec}_n^i(D)\right]_{i-n-j} \qquad \mbox{in $A^{*} \left(\ov{\M}_{0,n+1}\right)$}
\end{equation*}
where $[X]_{t}$ denotes the coefficient of $D^t$ in $X$.
This gives
\begin{equation}
\label{eq:Znij-prelimvan}
\mathsf{Z}(n,i,j) = \sum_{(\mathrm{T}, \bm{\psi})}  (-1)^{|\mathrm{H}^+(\mathrm{T})|} \,\, c^i_{\mathrm{T}, \bm{\psi}} \,\, \xi_{\mathrm{T} *}\left( \bm{\psi} \right) \qquad \mbox{in $A^{n-1+j-i} \left(\ov{\M}_{0,n+1}\right)$}
\end{equation}
where the sum is over decorated trees $(\mathrm{T}, \bm{\psi})$ such that $\xi_{\mathrm{T} *}\left( \bm{\psi} \right)$
has  degree $n-1+j-i$ in $A^* \left(\ov{\M}_{0,n+1}\right)$. 

To prove Theorem \ref{thm:van}, it suffices to prove the following. 

\begin{theorem}
\label{thm:van2}
For $i,j\geq 1$, one has
$\mathsf{Z}(n,i,j) =0$ in $A^{n-1+j-i} \left(\ov{\M}_{0,n+1}\right)$.

\end{theorem}

The statement is non-trivial for  $i<n$ (since the coefficients $c^i_{\mathrm{T}, \bm{\psi}}$ vanish by definition otherwise) and  $j < i$ (since the cycle vanishes for  degree reasons otherwise). This statement is proven as part $(b)$ of Theorem \ref{thm:mainthmgenuszeroij} in \S\ref{sec:recid}.

\subsection{Examples}
\label{sec:exvan}
Here we collect some examples of Theorem \ref{thm:van2}. Each graph stands for the sum of the pairwise non-isomorphic graphs obtained by all possible labeling of the unmarked legs by $1,\dots,n$. One has
\[
\mathsf{Z}(3,2,1)= \,
2
\mbox{
\begin{tikzpicture}[baseline={([yshift=-.5ex]current bounding box.center)}]
      \path(0,0) ellipse (2 and 2);
      \tikzstyle{level 1}=[counterclockwise from=-60, level distance=12mm, sibling angle=120]
      \node [draw,circle,fill] (A0) at (0:1.5) {}
	    child  {node [label=-60: {$\scriptstyle{}$}]{}}
    	    child  {node [label=60: {$\scriptstyle{}$}]{}};
       \tikzstyle{level 1}=[counterclockwise from=120,level distance=12mm,sibling angle=120]
        \node [draw,circle,fill] (A1) at (180:1.5) {}
            child {node [label=120: {$\scriptstyle{h_0}$}]{}}
    	    child {node [label=-120: {$\scriptstyle{}$}]{}};
        \path (A0) edge [] node[]{} (A1);
    \end{tikzpicture}
}
- 6
\mbox{
\begin{tikzpicture}[baseline={([yshift=-.5ex]current bounding box.center)}]
      \path(0,0) ellipse (1 and 1);
      \tikzstyle{level 1}=[counterclockwise from=180, level distance=12mm, sibling angle=120]
      \node [draw,circle,fill] (A0) at (0:0) {}
            child {node [label=180: {${\psi_{h_0}}$}]{}}
    	    child {node [label=-60: {$\scriptstyle{}$}]{}}
	    child {node [label=60: {$\scriptstyle{}$}]{}}
    	    child [grow=0]{node [label=0: {$\scriptstyle{}$}]{}};
    \end{tikzpicture}
}=0 \qquad \in A^{1} \left(\ov{\M}_{0,4}\right),
\]
and
\[
\mathsf{Z}(4,3,1)= \,
3
\mbox{
\begin{tikzpicture}[baseline={([yshift=-.5ex]current bounding box.center)}]
      \path(0,0) ellipse (2 and 2);
      \tikzstyle{level 1}=[counterclockwise from=-60, level distance=12mm, sibling angle=120]
      \node [draw,circle,fill] (A0) at (0:1.5) {}
	    child  {node [label=-60: {$\scriptstyle{}$}]{}}
    	    child  {node [label=60: {$\scriptstyle{}$}]{}};
       \tikzstyle{level 1}=[counterclockwise from=120,level distance=12mm,sibling angle=60]
        \node [draw,circle,fill] (A1) at (180:1.5) {}
            child {node [label=120: {$\scriptstyle{h_0}$}]{}}
    	    child {node [label=-180: {$\scriptstyle{}$}]{}}
    	    child {node [label=-120: {$\scriptstyle{}$}]{}};
        \path (A0) edge [] node[]{} (A1);
    \end{tikzpicture}
}
+9
\mbox{
\begin{tikzpicture}[baseline={([yshift=-.5ex]current bounding box.center)}]
      \path(0,0) ellipse (2 and 2);
      \tikzstyle{level 1}=[counterclockwise from=-60, level distance=12mm, sibling angle=60]
      \node [draw,circle,fill] (A0) at (0:1.5) {}
	    child  {node [label=-60: {$\scriptstyle{}$}]{}}
	    child  {node [label=-0: {$\scriptstyle{}$}]{}}
    	    child  {node [label=60: {$\scriptstyle{}$}]{}};
       \tikzstyle{level 1}=[counterclockwise from=120,level distance=12mm,sibling angle=120]
        \node [draw,circle,fill] (A1) at (180:1.5) {}
            child {node [label=120: {$\scriptstyle{h_0}$}]{}}
    	    child {node [label=-120: {$\scriptstyle{}$}]{}};
        \path (A0) edge [] node[]{} (A1);
    \end{tikzpicture}
}
- 18
\mbox{
\begin{tikzpicture}[baseline={([yshift=-.5ex]current bounding box.center)}]
      \path(0,0) ellipse (1 and 1);
      \tikzstyle{level 1}=[counterclockwise from=180, level distance=12mm, sibling angle=120]
      \node [draw,circle,fill] (A0) at (0:0) {}
            child {node [label=180: {${\psi_{h_0}}$}]{}}
    	    child {node [label=-60: {$\scriptstyle{}$}]{}}
	    child {node [label=60: {$\scriptstyle{}$}]{}}
    	    child [grow=0]{node [label=0: {$\scriptstyle{}$}]{}};
    \end{tikzpicture}
}
\!=0 
\]
in $A^{1} \left(\ov{\M}_{0,5}\right)$.
Also, the following cycles vanish in $A^{2} \left(\ov{\M}_{0,5}\right)$:
\begin{align*}
\mathsf{Z}(4,2,1) = & 
-14
\mbox{
\begin{tikzpicture}[baseline={([yshift=-.5ex]current bounding box.center)}]
      \path(0,0) ellipse (1 and 1);
      \tikzstyle{level 1}=[counterclockwise from=-90, level distance=12mm, sibling angle=60]
      \node [draw,circle,fill] (A0) at (0:0) {}
    	    child {node [label=-60: {$\scriptstyle{}$}]{}}
    	    child {node [label=-60: {$\scriptstyle{}$}]{}}	    
	    child {node [label=60: {$\scriptstyle{}$}]{}}
    	    child {node [label=0: {$\scriptstyle{}$}]{}}
	    child [grow=180]{node [label=180: {${\psi^2_{h_0}}$}]{}};
    \end{tikzpicture}
}
+14
\mbox{
\begin{tikzpicture}[baseline={([yshift=-.5ex]current bounding box.center)}]
      \path(0,0) ellipse (2 and 2);
      \tikzstyle{level 1}=[counterclockwise from=-60, level distance=12mm, sibling angle=60]
      \node [draw,circle,fill] (A0) at (0:1.5) {}
	    child  {node [label=-60: {$\scriptstyle{}$}]{}}
	    child  {node [label=-0: {$\scriptstyle{}$}]{}}
    	    child  {node [label=60: {$\scriptstyle{}$}]{}};
       \tikzstyle{level 1}=[counterclockwise from=120,level distance=12mm,sibling angle=120]
        \node [draw,circle,fill] (A1) at (180:1.5) {}
            child {node [label=120: {$\scriptstyle{h_0}$}]{}}
    	    child {node [label=-120: {$\scriptstyle{}$}]{}};
        \path (A0) edge [] node[near start, above=0.1, label={90:${\psi}$}]{} (A1);
    \end{tikzpicture}
}
+6
\mbox{
\begin{tikzpicture}[baseline={([yshift=-.5ex]current bounding box.center)}]
      \path(0,0) ellipse (2 and 2);
      \tikzstyle{level 1}=[counterclockwise from=-60, level distance=12mm, sibling angle=120]
      \node [draw,circle,fill] (A0) at (0:1.5) {}
	    child  {node [label=-60: {$\scriptstyle{}$}]{}}
    	    child  {node [label=60: {$\scriptstyle{}$}]{}};
       \tikzstyle{level 1}=[counterclockwise from=120,level distance=12mm,sibling angle=60]
        \node [draw,circle,fill] (A1) at (180:1.5) {}
            child {node [label=120: {${\psi_{h_0}}$}]{}}
    	    child {node [label=-180: {$\scriptstyle{}$}]{}}
    	    child {node [label=-120: {$\scriptstyle{}$}]{}};
        \path (A0) edge [] node[]{} (A1);
    \end{tikzpicture}
}
\\
&
-6
\mbox{
\begin{tikzpicture}[baseline={([yshift=-.5ex]current bounding box.center)}]
      \path(0,0) ellipse (3 and 2);
      \tikzstyle{level 1}=[counterclockwise from=-60, level distance=12mm, sibling angle=120]
      \node [draw,circle,fill] (A0) at (0:2.5) {}
	    child  {node [label=-60: {$\scriptstyle{}$}]{}}
    	    child  {node [label=60: {$\scriptstyle{}$}]{}};
       \tikzstyle{level 1}=[counterclockwise from=-90,level distance=12mm,sibling angle=60]
        \node [draw,circle,fill] (A1) at (0:0) {}
    	    child {node [label=-90: {$\scriptstyle{}$}]{}};
       \tikzstyle{level 1}=[counterclockwise from=120,level distance=12mm,sibling angle=120]
        \node [draw,circle,fill] (A2) at (180:2.5) {}
            child {node [label=120: {$\scriptstyle{h_0}$}]{}}
    	    child {node [label=-120: {$\scriptstyle{}$}]{}};	    
        \path (A0) edge [] node[]{} (A1);
        \path (A1) edge [] node[]{} (A2);        
    \end{tikzpicture}
} 
-2
\mbox{
\begin{tikzpicture}[baseline={([yshift=-.5ex]current bounding box.center)}]
      \path(0,0) ellipse (3 and 2);
      \tikzstyle{level 1}=[counterclockwise from=-60, level distance=12mm, sibling angle=120]
      \node [draw,circle,fill] (A0) at (0:2.5) {}
	    child  {node [label=-60: {$\scriptstyle{}$}]{}}
    	    child  {node [label=60: {$\scriptstyle{}$}]{}};
       \tikzstyle{level 1}=[counterclockwise from=-90,level distance=12mm,sibling angle=60]
        \node [draw,circle,fill] (A1) at (0:0) {}
    	    child {node [label=-90: {$\scriptstyle{h_0}$}]{}};
       \tikzstyle{level 1}=[counterclockwise from=120,level distance=12mm,sibling angle=120]
        \node [draw,circle,fill] (A2) at (180:2.5) {}
            child {node [label=120: {$\scriptstyle{}$}]{}}
    	    child {node [label=-120: {$\scriptstyle{}$}]{}};	    
        \path (A0) edge [] node[]{} (A1);
        \path (A1) edge [] node[]{} (A2);        
    \end{tikzpicture}
} ,
\end{align*}
\begin{align*}
\mathsf{Z}(4,3,2) = & 
-75
\mbox{
\begin{tikzpicture}[baseline={([yshift=-.5ex]current bounding box.center)}]
      \path(0,0) ellipse (1 and 1);
      \tikzstyle{level 1}=[counterclockwise from=-90, level distance=12mm, sibling angle=60]
      \node [draw,circle,fill] (A0) at (0:0) {}
    	    child {node [label=-60: {$\scriptstyle{}$}]{}}
    	    child {node [label=-60: {$\scriptstyle{}$}]{}}	    
	    child {node [label=60: {$\scriptstyle{}$}]{}}
    	    child {node [label=0: {$\scriptstyle{}$}]{}}
	    child [grow=180]{node [label=180: {${\psi^2_{h_0}}$}]{}};
    \end{tikzpicture}
}
+21
\mbox{
\begin{tikzpicture}[baseline={([yshift=-.5ex]current bounding box.center)}]
      \path(0,0) ellipse (2 and 2);
      \tikzstyle{level 1}=[counterclockwise from=-60, level distance=12mm, sibling angle=60]
      \node [draw,circle,fill] (A0) at (0:1.5) {}
	    child  {node [label=-60: {$\scriptstyle{}$}]{}}
	    child  {node [label=-0: {$\scriptstyle{}$}]{}}
    	    child  {node [label=60: {$\scriptstyle{}$}]{}};
       \tikzstyle{level 1}=[counterclockwise from=120,level distance=12mm,sibling angle=120]
        \node [draw,circle,fill] (A1) at (180:1.5) {}
            child {node [label=120: {$\scriptstyle{h_0}$}]{}}
    	    child {node [label=-120: {$\scriptstyle{}$}]{}};
        \path (A0) edge [] node[near start, above=0.1, label={90:${\psi}$}]{} (A1);
    \end{tikzpicture}
}
+18
\mbox{
\begin{tikzpicture}[baseline={([yshift=-.5ex]current bounding box.center)}]
      \path(0,0) ellipse (2 and 2);
      \tikzstyle{level 1}=[counterclockwise from=-60, level distance=12mm, sibling angle=120]
      \node [draw,circle,fill] (A0) at (0:1.5) {}
	    child  {node [label=-60: {$\scriptstyle{}$}]{}}
    	    child  {node [label=60: {$\scriptstyle{}$}]{}};
       \tikzstyle{level 1}=[counterclockwise from=120,level distance=12mm,sibling angle=60]
        \node [draw,circle,fill] (A1) at (180:1.5) {}
            child {node [label=120: {${\psi_{h_0}}$}]{}}
    	    child {node [label=-180: {$\scriptstyle{}$}]{}}
    	    child {node [label=-120: {$\scriptstyle{}$}]{}};
        \path (A0) edge [] node[]{} (A1);
    \end{tikzpicture}
}
\\
&
-9
\mbox{
\begin{tikzpicture}[baseline={([yshift=-.5ex]current bounding box.center)}]
      \path(0,0) ellipse (3 and 2);
      \tikzstyle{level 1}=[counterclockwise from=-60, level distance=12mm, sibling angle=120]
      \node [draw,circle,fill] (A0) at (0:2.5) {}
	    child  {node [label=-60: {$\scriptstyle{}$}]{}}
    	    child  {node [label=60: {$\scriptstyle{}$}]{}};
       \tikzstyle{level 1}=[counterclockwise from=-90,level distance=12mm,sibling angle=60]
        \node [draw,circle,fill] (A1) at (0:0) {}
    	    child {node [label=-90: {$\scriptstyle{}$}]{}};
       \tikzstyle{level 1}=[counterclockwise from=120,level distance=12mm,sibling angle=120]
        \node [draw,circle,fill] (A2) at (180:2.5) {}
            child {node [label=120: {$\scriptstyle{h_0}$}]{}}
    	    child {node [label=-120: {$\scriptstyle{}$}]{}};	    
        \path (A0) edge [] node[]{} (A1);
        \path (A1) edge [] node[]{} (A2);        
    \end{tikzpicture}
} 
-3
\mbox{
\begin{tikzpicture}[baseline={([yshift=-.5ex]current bounding box.center)}]
      \path(0,0) ellipse (3 and 2);
      \tikzstyle{level 1}=[counterclockwise from=-60, level distance=12mm, sibling angle=120]
      \node [draw,circle,fill] (A0) at (0:2.5) {}
	    child  {node [label=-60: {$\scriptstyle{}$}]{}}
    	    child  {node [label=60: {$\scriptstyle{}$}]{}};
       \tikzstyle{level 1}=[counterclockwise from=-90,level distance=12mm,sibling angle=60]
        \node [draw,circle,fill] (A1) at (0:0) {}
    	    child {node [label=-90: {$\scriptstyle{h_0}$}]{}};
       \tikzstyle{level 1}=[counterclockwise from=120,level distance=12mm,sibling angle=120]
        \node [draw,circle,fill] (A2) at (180:2.5) {}
            child {node [label=120: {$\scriptstyle{}$}]{}}
    	    child {node [label=-120: {$\scriptstyle{}$}]{}};	    
        \path (A0) edge [] node[]{} (A1);
        \path (A1) edge [] node[]{} (A2);        
    \end{tikzpicture}
} .
\end{align*}

\subsection{Preliminary identities}

Here we start proving some identities which will help to prove Theorem \ref{thm:van2} in \S\ref{sec:recid}. 

\begin{lemma}
\label{lemma:imaxj1van}
For $n\geq 3$, one has $\mathsf{Z}(n,n-1,1) =0$ in $A^{1} \left(\ov{\M}_{0,n+1}\right)$.
\end{lemma}

\begin{proof}
From \eqref{eq:mu1minus1}, a decorated tree $(\mathrm{T}, \bm{\psi})$ contributes to $\mathsf{Z}(n,n-1,1)$ if and only if 
$|\mathrm{E}(\mathrm{T})| +\deg (\bm{\psi})=1$, i.e., $|\mathrm{H}(\mathrm{T}, \bm{\psi})|=2$.
Recall from \S\ref{sec:RRT} that $\mathrm{H}^+(\mathrm{T})$ always contains the element $h_0$.
There are two cases. 

If $\deg \bm{\psi}=1$, then the tree $\mathrm{T}$ has only one vertex, and $\bm{\psi}=\psi_{h_0}$.
Hence, one has $\mathrm{H}(\mathrm{T}, \bm{\psi})=\{(h_0,0), (h_0,1)\}$. In this case, the set of $(n-1)$-weightings $\mathsf{W}^{n-1}_{\mathrm{T}, \bm{\psi}}$ from \eqref{eq:Wi1tipsi} consists of the functions $w$ such that $w(h_0,0)=n-1$ and $w(h_0,1)=a$, with $a\in \{1,\dots, n-1\}$. From \eqref{eq:cimtipsi}, it follows that $c^{n-1}_{\mathrm{T}, \bm{\psi}}=(n-1){n \choose 2}$.

If $\deg \bm{\psi}=0$, then the tree $\mathrm{T}$ has exactly one edge $e=(h^+,h^-)$. In this case, one has $\mathrm{H}(\mathrm{T}, \bm{\psi})=\{(h_0,0), (h^+,0)\}$.
The set $\mathsf{W}^{n-1}_{\mathrm{T}, \bm{\psi}}$ consists of the functions $w$ such that $w(h_0,0)=n-1$ and $w(h^+,0)=a$, with $a\in \{1,\dots, \ell_{h^+}-1 \}$. 
It follows that $c^{n-1}_{\mathrm{T}, \bm{\psi}}=(n-1){\ell_{h^+} \choose 2}$.

We deduce that 
\begin{equation}
\label{eq:Znnminus11expl}
\mathsf{Z}(n,n-1,1) = - (n-1){n \choose 2}\psi_{h_0} + (n-1)\sum_{m=2}^{n-1} {m \choose 2} \delta_{m}
\end{equation}
where $\delta_{m}$ is the sum of the classes of the boundary divisors in $\ov{\M}_{0,n+1}$ whose general elements consist of a rational component containing the marked point $h_0$ and a rational component containing precisely $m$ other marked points. The vanishing of $\mathsf{Z}(n,n-1,1)$ is then due to the identity 
\begin{equation}
\label{eq:divid1}
{n \choose 2}\psi_{h_0} = \sum_{m=2}^{n-1} {m \choose 2} \delta_{m} \qquad \mbox{in $A^{1} \left(\ov{\M}_{0,n+1}\right)$.}
\end{equation}
This  follows from the well-known identity
$\psi_{h_0} = \sum_{M} \delta_{M}$ in $A^{1} \left(\ov{\M}_{0,n+1}\right)$
obtained by fixing two elements $a,b\in\{1,\dots,n\}$,
where the sum is over $\{a,b\}\subseteq M\subset \{1, \dots, n\}$ with $|M|\leq n-1$, and  $\delta_{M}$ is the class of the divisor in $\ov{\M}_{0,n+1}$ whose general element consists of a rational component containing the marked point $h_0$ and a rational component containing precisely the marked points in~$M$. 
Summing over all choices of $a$ and $b$, one deduces \eqref{eq:divid1}, hence the vanishing of $\mathsf{Z}(n,n-1,1)$ from~\eqref{eq:Znnminus11expl}.
\end{proof}

We end this section with the following recursive identity, which will be used in the proof of Theorem \ref{thm:mainthmgenuszeroij}$(b)$ to deduce the vanishing of $\mathsf{Z}(n,n-1,j)$  for  $j\geq 1$ from Lemma  \ref{lemma:imaxj1van} and the vanishing of  $\mathsf{Z}(n,n-2,j)$ for $j\geq 1$:

\begin{lemma}
\label{lemma:vanimaxrec}
For $j\geq 2$, the following holds in $A^{j} \left(\ov{\M}_{0,n+1}\right)$:
\begin{align*}
\mathsf{Z}(n,n-1,j) = 
  \frac{n-1}{n-2} \, \mathsf{Z}(n,n-2,j-1)
  +(n-1)\,\psi_{h_0} \, \mathsf{Z}(n,n-1,j-1) .
\end{align*}
\end{lemma}

\begin{proof}
Let $j\geq 2$, and let $(\mathrm{T},\bm{\psi})$ be a decorated tree such that $\xi_{\mathrm{T} *}\left( \bm{\psi} \right)$ has degree
 $j$ in $A^* \left(\ov{\M}_{0,n+1}\right)$. We show that the contributions of $(\mathrm{T},\bm{\psi})$ to the two sides of the equality in the statement match.
 
The maximal capacity $\ell_h = n$ for $h$ in $\mathrm{H}^+(\mathrm{T})$ is attained only at $h=h_0$. It follows that the only elements of  $\mathrm{H}(\mathrm{T}, \bm{\psi})$ where an $(n-1)$-weighting can have value $n-1$ are the pairs $(h,e)$ where $h=h_0$. 
We distinguish two cases.

First, assume the degree of $\psi_{h_0}$ in $\bm{\psi}$ is zero. In this case, $(\mathrm{T},\bm{\psi})$ does not contribute to $\psi_{h_0} \, \mathsf{Z}(n,n-1,j-1)$.
Moreover, the pair $(h_0,0)$ is the only element of  $\mathrm{H}(\mathrm{T}, \bm{\psi})$ where an $(n-1)$-weighting can have value $n-1$. 
It follows that the map
\[
\mathsf{W}^{n-1}_{\mathrm{T},\bm{\psi}} \rightarrow \mathsf{W}^{n-2}_{\mathrm{T},\bm{\psi}}, \quad w \mapsto \overline{w} \quad \mbox{where} \quad\overline{w}(h,e):=
\left\{
\begin{array}{ll}
n-2&\mbox{if $(h,e)=(h_0,0)$,}\\
w(h,e)&\mbox{otherwise}
\end{array}
\right.
\]
is a bijection. If $w\mapsto \ov{w}$, then  $w$ and $\ov{w}$ only differ at $(h_0,0)$, where $w(h_0,0)=n-1$ and $\ov{w}(h_0,0)=n-2$.
Using the definition \eqref{eq:imweighttipsi}, one has $w(\mathrm{T},\bm{\psi})= \frac{n-1}{n-2} \, \overline{w}(\mathrm{T},\bm{\psi})$.
From the definition \eqref{eq:cimtipsi}, one has
\[
c^{n-1}_{\mathrm{T}, \bm{\psi}} =   \frac{n-1}{n-2} \, c^{n-2}_{\mathrm{T}, \bm{\psi}}.
\]
From \eqref{eq:Znij-prelimvan}, it follows that $(\mathrm{T},\bm{\psi})$ contributes equally to the two sides of the equality in the statement.

Finally, assume the degree of $\psi_{h_0}$ in $\bm{\psi}$ is positive. Let $d_0>0$ be the degree of $\psi_{h_0}$ in $\bm{\psi}$. Then $\mathrm{H}(\mathrm{T},\bm{\psi})$ contains  $(h_0,e)$ precisely for $0\leq e\leq d_0$.
Consider the partition
\[
\mathsf{W}^{n-1}_{\mathrm{T},\bm{\psi}} = \mathsf{W}^a \sqcup \mathsf{W}^b
\]
where
\begin{align*}
\mathsf{W}^a &:= \left\{w\in \mathsf{W}^{n-1}_{\mathrm{T},\bm{\psi}} \, | \, w(h_0,1)\leq n-2 \right\},\\
\mathsf{W}^b &:= \left\{w\in \mathsf{W}^{n-1}_{\mathrm{T},\bm{\psi}} \, | \, w(h_0,1)=n-1 \right\} .
\end{align*}
The map $w\mapsto \ov{w}$ from the previous case gives a bijection $\mathsf{W}^a \xrightarrow{\cong} \mathsf{W}^{n-2}_{\mathrm{T},\bm{\psi}}$ such that $w(\mathrm{T},\bm{\psi})= \frac{n-1}{n-2} \, \overline{w}(\mathrm{T},\bm{\psi})$. Next, let $\bm{\psi}^-:=\psi_{h_0}^{-1}\bm{\psi}$. 
The contribution of $(\mathrm{T},\bm{\psi})$ to $\psi_{h_0} \, \mathsf{Z}(n,n-1,j-1)$ equals the contribution of $(\mathrm{T},\bm{\psi}^-)$ to $\mathsf{Z}(n,n-1,j-1)$.
The set $\mathrm{H}(\mathrm{T},\bm{\psi}^-)$ contains $(h_0,e)$ only for $0\leq e\leq d_0-1$.
One has a bijection
\[
\mathsf{W}^b \rightarrow \mathsf{W}^{n-1}_{\mathrm{T},\bm{\psi}^-}, \quad w \mapsto w^- \quad \mbox{where} \quad w^-(h,e):=
\left\{
\begin{array}{ll}
w(h_0, e+1)&\mbox{if $h=h_0$,}\\
w(h,e)&\mbox{otherwise.}
\end{array}
\right.
\]
Moreover, if $w\mapsto w^-$, then $w(\mathrm{T},\bm{\psi})= (n-1) \, w^-(\mathrm{T},\bm{\psi}^-)$.
Hence one has
\[
c^{n-1}_{\mathrm{T}, \bm{\psi}} =   \frac{n-1}{n-2} \, c^{n-2}_{\mathrm{T}, \bm{\psi}} + (n-1)\, c^{n-1}_{\mathrm{T}, \bm{\psi}^-}.
\]
It follows that $(\mathrm{T},\bm{\psi})$ contributes equally to the two sides of the equality in the statement.
This implies the statement.
\end{proof}


\section{Recursive identities on moduli  of pointed rational curves}
\label{sec:recid}

Here we prove certain recursive identities satisfied by the cycles $\mathsf{Z}(n,i,j)$ and $\mathsf{Dec}_n^i(D)$. These will be a key ingredient in the proof of Theorem \ref{thm:graphFor}.
The proof  of the recursive identities is intertwined with the proof of the vanishing statement from Theorem \ref{thm:van2}.

\smallskip

Recall the cycle $\mathsf{Z}(n,i,j)$ in $A^* \left(\ov{\M}_{0,n+1}\right)$ from \eqref{eq:Znij-prelimvan}.
Expanding the coefficients $c^i_{\mathrm{T}, \bm{\psi}}= c^{i,1}_{\mathrm{T}, \bm{\psi}}$ as in \eqref{eq:cimtipsi}, one has
\begin{equation}
\label{eq:Znijw}
\mathsf{Z}(n,i,j) = \sum_{(\mathrm{T}, \bm{\psi},w)}  (-1)^{|\mathrm{H}^+(\mathrm{T})|} \,\, w(\mathrm{T}, \bm{\psi}) \,\, \xi_{\mathrm{T} *}\left( \bm{\psi} \right) 
\quad \mbox{in $A^{n-1+j-i} \left(\ov{\M}_{0,n+1}\right)$}
\end{equation}
where the sum is over $(\mathrm{T}, \bm{\psi},w)$ such that $\xi_{\mathrm{T} *}\left( \bm{\psi} \right)$ 
has  degree $n-1+j-i$ in $A^* \left(\ov{\M}_{0,n+1}\right)$ and $w$ is an $i$-weighting of $(\mathrm{T}, \bm{\psi})$. Recall that the legs of each such $\mathrm{T}$ are labelled by $1,\dots,n, h_0$.

Define the \textit{truncated} cycle
\begin{equation}
\label{eq:Ztnij}
\mathsf{Z}^{\mathrm t}(n,i,j) \qquad \mbox{in $A^{n-1+j-i} \left(\ov{\M}_{0,n+1}\right)$}
\end{equation}
as the cycle obtained from $\mathsf{Z}(n,i,j)$ by restricting the sum in \eqref{eq:Znijw} to those $(\mathrm{T}, \bm{\psi},w)$ additionally satisfying the following property: 
\begin{equation}
\label{eq:truncatedprop}
\left\{
\begin{array}{l}
\mbox{if $h_0$, the leg $n$, and a tail $h^-$ are incident to the same trivalent vertex,}\\[.1cm]
\mbox{then $i=w(h_0,0)\geq w(\iota(h^-),0)$.}
\end{array}
\right.
\end{equation}

Also, for a non-empty $\bm{I}\subseteq \{1,\dots,n-1\}$,  recall the cycle
\begin{equation*}
 \mathsf{E}_{\bm{I}}(i,j) = \sum_{(\mathrm{T}, \bm{\psi})}  (-1)^{|\mathrm{H}^+(\mathrm{T})|-1}  \,\, d^i_{\mathrm{T},\bm{\psi}} \,\, \xi_{\mathrm{T} *}\left( \bm{\psi} \right)  \qquad \mbox{in $A^{n-1+j-i}\left( \ov{\M}_{0,n+1}\right)$}
\end{equation*}
from \eqref{eq:EIij}, where the sum is over those decorated trees $(\mathrm{T}, \bm{\psi})$ in the set $\mathsf{DC}_{\bm{I}}$ from \eqref{eq:DCI} such that $\xi_{\mathrm{T} *}\left( \bm{\psi} \right)$ has degree $n-1+j-i$,
and the coefficients $d^i_{\mathrm{T}, \bm{\psi}}$ are as in \eqref{eq:ditipsi}.

Let $\pi_{n}\colon \ov{\M}_{0,n+1}\rightarrow\ov{\M}_{0,n}$ be the map obtained by forgetting the point $P_{n}$.

\begin{theorem}
\label{thm:mainthmgenuszeroij}
\begin{enumerate}[(a)]

\item For $1\leq i \leq n-2$ and $j\in \mathbb{Z}$, one has
\begin{equation}
\label{eq:partapullback}
\pi_{n}^*\left( \mathsf{Z}(n-1,i,j+1) \right) - \sum_{\bm{I} } |\bm{I}| \, \mathsf{E}_{\bm{I}}(i,j) 
= \mathsf{Z}^{\rm t}(n,i,j)
\end{equation}
in $A^{n-1+j-i}\left(\ov{\M}_{0,n+1}\right)$, where the sum is over all $\varnothing\neq \bm{I}\subseteq \{1,\dots, n-1\}$.

\item For $i,j \geq 1$, one has
\[
\mathsf{Z}^{\rm t}(n,i,j) = 0 \quad \mbox{and} \quad \mathsf{Z}(n,i,j) =0 \qquad \mbox{in $A^{n-1+j-i}\left(\ov{\M}_{0,n+1}\right)$.}
\]
\end{enumerate}
\end{theorem}

As a consequence of Theorem \ref{thm:mainthmgenuszeroij}, we show:

\begin{theorem}
\label{thm:mainthmgenuszero}
For all $i\geq 1$ and $j\in \mathbb{Z}$, one has
\begin{equation}
\label{eq:partapullback2}
\pi_{n}^*\left( \mathsf{Z}(n-1,i,j+1) \right) - \sum_{\bm{I} } |\bm{I}| \, \mathsf{E}_{\bm{I}}(i,j) 
= \mathsf{Z}^{\rm t}(n,i,j)
\end{equation}
in $A^{n-1+j-i}\left(\ov{\M}_{0,n+1}\right)$, where the sum is over all $\varnothing\neq \bm{I}\subseteq \{1,\dots, n-1\}$.
\end{theorem}

Next, we modify this statement by replacing the truncated cycles $\mathsf{Z}^{\rm t}(n,i,j)$ via the following Lemma. Let $\sigma_0\colon \ov{\M}_{0,n} \rightarrow \ov{\M}_{0,n+1}$ be the gluing map obtained by attaching at the marked point $h_0$ a rational component containing the marked points $h_0$ and $P_n$.

\begin{lemma}
\label{lem:DectZt}
For $i\geq 1$ and $j\in \mathbb{Z}$, one has
\begin{equation*}
\mathsf{Z}(n,i,j) = \mathsf{Z}^{\rm t}(n,i,j)
-\sum_{i^+>i} i\,\sigma_{0*} \left( \mathsf{Z}(n-1,i^+, j^+) \right) \,\, \in A^{n-1+j-i}\left( \ov{\M}_{0,n+1}\right)
\end{equation*}
where for each $i^+$, the number $j^+$ is defined by $j^+ - i^+ = j-i$. 
\end{lemma}

\begin{proof}
Since $j^+ - i^+ = j-i$, the cycle $\mathsf{Z}(n-1,i^+, j^+)$ has degree equal to $n-2+j-i$ in $A^*\left(\ov{\M}_{0,n}\right)$, and thus $\sigma_{0*} \left( \mathsf{Z}(n-1,i^+, j^+) \right)$ has degree $n-1+j-i$ in $A^*\left(\ov{\M}_{0,n+1}\right)$.
The statement follows from the definition of $\mathsf{Z}^{\rm t}(n,i,j)$ as the cycle obtained from $\mathsf{Z}(n,i,j)$ by restricting the sum in \eqref{eq:Znijw}  
to those triples $(\mathrm{T}, \bm{\psi}, w)$ satisfying \eqref{eq:truncatedprop}. Hence the difference $\mathsf{Z}(n,i,j) -\mathsf{Z}^{\rm t}(n,i,j)$ is contributed by those triples $(\mathrm{T}, \bm{\psi}, w)$ which do not satisfy \eqref{eq:truncatedprop}, i.e., $h_0$, the leg $n$, and a tail $h^-$ are incident to a trivalent vertex of $\mathrm{T}$ and $i = w(h_0,0) < w(\iota(h^-),0) =: i^+$. Such a triple $(\mathrm{T}, \bm{\psi}, w)$ contributes to $i\,\sigma_{0*} \left( \mathsf{Z}(n-1,i^+, j^+) \right)$, and vice versa. The negative sign before $i\,\sigma_{0*} \left( \mathsf{Z}(n-1,i^+, j^+) \right)$ accounts for the new edge contributed by the gluing map $\sigma_0$.
Hence the statement.
\end{proof}

By  \eqref{eq:Znijdef-m} and \eqref{eq:Eijdef}, the cycles $\mathsf{Z}(n,i,j)$ and $\mathsf{E}_{\bm{I}}(i,j)$ are coefficients of the polynomials 
$\mathsf{Dec}_n^i(D)=\mathsf{Dec}_n^{i,1}(D)$ and $\mathsf{E}_{\bm{I}}^i(D)$, respectively. That is,
\begin{align*}
\mathsf{Dec}_n^{i}(D) & = \sum_{j> i-n} \mathsf{Z}(n,i,j) D^{i-n-j}, &
\mathsf{E}_{\bm{I}}^i(D) & = \sum_{j> i-n} \mathsf{E}_{\bm{I}}(i,j) D^{i-n-j}
\end{align*}
in $A^*\left( \ov{\M}_{0,n+1}\right)\left[ D^{-1} \right]$.
Combining Theorem \ref{thm:mainthmgenuszero} and Lemma \ref{lem:DectZt}, we have:

\begin{theorem}
\label{thm:Decrec}
For $i\geq 1$, one has
\[
\pi_{n}^*\left( \mathsf{Dec}_{n-1}^i(D) \right) 
- \sum_{\bm{I} } |\bm{I}| \, \mathsf{E}_{\bm{I}}^i(D)
-\sum_{i^+>i} i\,\sigma_{0*} \left(   \mathsf{Dec}_{n-1}^{i^+}(D) \right)
= \mathsf{Dec}_{n}^i(D)
\]
in $A^*\left( \ov{\M}_{0,n+1}\right)\left[D^{-1}\right]$. The first sum is over all $\varnothing\neq \bm{I}\subseteq \{1,\dots, n-1\}$.
\end{theorem}

This statement will be used in the proof of Theorem \ref{thm:graphFor-m}. The next Lemma \ref{lemma:Thm1implies2} shows that for any given $n$, Theorem \ref{thm:mainthmgenuszero} follows from Theorem \ref{thm:mainthmgenuszeroij}. The latter is proven in the remainder of \S\ref{sec:mainthmgenuszero}.

\subsection{Partitions of $\mathsf{G}_{0,n+1}$}
\label{sec:partitions}
We consider here partitions of the set of rational rooted trees $\mathsf{G}_{0,n+1}$. These will be used in the proof of Theorem \ref{thm:mainthmgenuszeroij}. 

Let $\mathrm{T}\in \mathsf{G}_{0,n+1}$. 
A vertex $v$ of $\mathrm{T}$ is \textit{external}  if there is no tail half-edge of $\mathrm{T}$  incident to $v$.
A subtree $\mathrm{T}'$ of $\mathrm{T}$ is  \textit{external} if either $\mathrm{T}' = \mathrm{T}$, or there exists an edge $(h^+,h^-)$ in $\mathrm{E}(\mathrm{T})$ such that $\mathrm{T}'$ is the maximal subtree of $\mathrm{T}$ containing the head half-edge $h^+$, but not the tail half-edge $h^-$. 

Let $v_n$ be the vertex incident to the leg $n$, and  $v_0$ the vertex incident to~$h_0$. Let $h_n$ be the unique element of $\mathrm{H}^+(\mathrm{T})$ incident to~$v_n$. For a half-edge $h$, let $\iota(h)$ be the half-edge such that $(h, \iota(h))\in \mathrm{E}(\mathrm{T})$.

Consider the set of decorated rational trees with $i$-weightings:
\begin{equation}
\label{eq:DGW}
\mathsf{DGW}^i_{0,n+1} = \left\{ (\mathrm{T}, \bm{\psi}, w) \, | \, \mathrm{T}\in \mathsf{G}_{0,n+1}, \, \bm{\psi}\in \Psi(\mathrm{T}), \, w\in \mathsf{W}^i_{\mathrm{T}, \bm{\psi}}  \right\}.
\end{equation}
For $1\leq i\leq  n-2$, let
\[
\mathsf{DGW}^i_{0,n+1} = \mathsf{A}_{0,n+1}^i \sqcup \mathsf{B}_{0,n+1}^i \sqcup \mathsf{C}_{0,n+1}^i
\]
be the partition where 
\begin{align*}
\mathsf{A}_{0,n+1}^i := &
\left\{ (\mathrm{T}, \bm{\psi}, w) \left|  
\begin{array}{l}
\mbox{either $v_n$ is incident to $h_0$;}\\[.1cm]
\mbox{or $v_n$ is trivalent, adjacent to $v_0$, and incident to} \\[.1cm]
\mbox{a tail $h^-$, and $w(h_n,0)<w(\iota(h^-),0)$}
\end{array}
\right.
\!
\right\}, \\
\mathsf{B}_{0,n+1}^i :=& \left\{ (\mathrm{T}, \bm{\psi}, w) \left|  
\begin{array}{l}
\mbox{$v_n$ is not incident to $h_0$; and}\\[.1cm]
\mbox{if $v_n$ is trivalent, adjacent to $v_0$, and incident to}\\[.1cm]
\mbox{a tail $h^-$, then $w(h_n,0)\geq w(\iota(h^-),0)$; and}\\[.1cm]
\mbox{if $v_n$ is external,} \\[.1cm]
\mbox{then the degree of $\psi_{h_n}$ in $\bm{\psi}$ is positive}
\end{array}
\right.
\!
\right\},\\
\mathsf{C}_{0,n+1}^i :=& \, \left\{ (\mathrm{T}, \bm{\psi}, w) \left|
\begin{array}{l}
\mbox{$v_n$ is external and not incident to $h_0$,}\\[.1cm]
\mbox{and the degree of $\psi_{h_n}$ in $\bm{\psi}$ is zero}
\end{array}
\right.
\!
\right\}.
\end{align*}

For $(\mathrm{T}, \bm{\psi}, w)\in \mathsf{A}_{0,n+1}^i$, the vertex $v_n$ could be external. In this case, $v_n$ is the only vertex of $\mathrm{T}$. 

For $n= 3$ and $i=1$, one has $\mathsf{B}_{0,n+1}^1 = \varnothing$, and $\mathsf{G}_{0,n+1} = \mathsf{A}_{0,n+1}^1 \sqcup \mathsf{C}_{0,n+1}^1$.

\subsection{Proof of Theorems \ref{thm:mainthmgenuszeroij} and \ref{thm:mainthmgenuszero}}
\label{sec:mainthmgenuszero}
We start with two preliminary Lemmata:

\begin{lemma}
\label{lemma:leftrightterms}
The LHS of \eqref{eq:partapullback} expands as a $\mathbb{Z}$-linear combination of classes $\xi_{\mathrm{T}*}\left( \bm{\psi}\right)$ for $\mathrm{T}\in \mathsf{G}_{0,n+1}$ and  $\bm{\psi}\in \Psi(\mathrm{T})$, as does the RHS.
\end{lemma}

\begin{proof}
Let us first consider $\pi_{n}^*\left( \mathsf{Z}(n-1,i,j+1) \right)$.
By definition, this is a $\mathbb{Z}$-linear combination of terms of type $\pi_n^*\left(\xi_{\overline{\mathrm{T}}*}\left( \bm{\psi}\right) \right)$ with $\overline{\mathrm{T}}$ in $\mathsf{G}_{0,n}$ and $\bm{\psi}$ in $\Psi\left( \ov{\mathrm{T}}\right)$. Each such term  expands as 
\begin{equation}
\label{eq:pullbackTbar}
\pi_n^*\left(\xi_{\ov{\mathrm{T}}*}\left( \bm{\psi}\right) \right) = \sum_{\mathrm{T}} \xi_{\mathrm{T}*}\left( \pi_n^*\left(\bm{\psi}\right) \right)
\end{equation}
where $\mathrm{T}\in \mathsf{G}_{0,n+1}$ ranges over all graphs obtained from $\ov{\mathrm{T}}$ by adding a leg $n$ to one of its vertices. 
Let $\mathrm{T}$ be one of such graphs,  let $v_n$ be the vertex incident to the leg $n$, and $h_n$ be the unique element of $\mathrm{H}^+(\mathrm{T})$ incident to $v_n$.

Consider first the case when for all tail half-edges $h^-$ incident to $v_n$, the degree of $\psi_{h^-}$ in $\bm{\psi}$ is zero.
Pull-back formulas for $\psi$-classes give
\begin{equation*}
\xi_{\mathrm{T}*}\left( \pi_n^*\left(\bm{\psi}\right) \right) = \xi_{\mathrm{T}*}\left( \bm{\psi}\right) - \xi_{\mathrm{T}^\circ *}\left( \bm{\psi}^\circ\right) 
\qquad  \in A^*\left( \ov{\M}_{0,n+1}\right)
\end{equation*}
where: 
\begin{align}
\label{eq:Tcirc}
\mathrm{T}^\circ &:=\left\{
\begin{array}{l}
\mbox{the tree obtained from $\mathrm{T}$ by}\\
\mbox{splitting the vertex $v_n$ into two  vertices $v_n^\circ$  and $v^\perp$}\\
\mbox{such that $v_n^\circ$ is trivalent and incident to: the leg $n$,}\\
\mbox{the element in $\mathrm{H}^+(\mathrm{T}^\circ)$ equivalent to  $h_n$ from $\mathrm{T}$,}\\ 
\mbox{and an edge incident to $v^\perp$;}
\end{array}
\right. \\
\allowdisplaybreaks
\label{eq:psicirc}
\bm{\psi}^\circ &:= \left\{
\begin{array}{l}
\mbox{the monomial obtained from $\bm{\psi}$  by}\\ 
\mbox{replacing the factor $\psi_{h_n}^{d_n}$,}\\ 
\mbox{where $d_n$ is the degree of $\psi_{h_n}$ in $\bm{\psi}$,}\\
\mbox{with $\psi_{h^\perp}^{d^\perp}$ where $h^\perp$ is the head half-edge incident to $v^\perp$}\\
\mbox{and  $d^\perp := d_n-1$;}\\
\mbox{if $d_n=0$, then one sets $\bm{\psi}^\circ =0$.}
\end{array}
\right.
\end{align}
An example is given in Figure \ref{fig:Tcircpsicirc}.

\begin{figure}[htb]
\begin{align*}
(\mathrm{T}, \bm{\psi}) =
\begin{tikzpicture}[baseline={([yshift=-.5ex]current bounding box.center)}]
      \path(0,0) ellipse (3 and 2);
      \tikzstyle{level 1}=[counterclockwise from=-60, level distance=12mm, sibling angle=120]
      \node [draw,circle,fill] (A0) at (0:3) {}
	    child  {node [label=-60: {$\scriptstyle{}$}]{}}
    	    child  {node [label=60: {$\scriptstyle{}$}]{}};
       \tikzstyle{level 1}=[counterclockwise from=-120,level distance=12mm,sibling angle=30]
        \node [draw,circle,fill] (A1) at (0:0) {}
    	    child {node [label=-120: {$\scriptstyle{n}$}]{}}
    	    child {node [label=-90: {$\scriptstyle{}$}]{}}
    	    child {node [label=-90: {$\scriptstyle{}$}]{}};
       \tikzstyle{level 1}=[counterclockwise from=120,level distance=12mm,sibling angle=120]
             \node [draw,circle,fill] (A2) at (180:3) {}
             child  {node [label=120: {$\scriptstyle{h_0}$}]{}}
    	     child  {node [label=60: {$\scriptstyle{}$}]{}};
        \path (A0) edge [] node[pos=0.6, above=0.1, label={90:$\psi_{h^-}$}]{} (A1);
        \path (A1) edge [] node[pos=0.3, above=0.1, label={90:$\psi_{h_n}$}]{} (A2);        
    \end{tikzpicture}  \mapsto \left\{
    \begin{array}{r}
\left(\mathrm{T}^\circ, \bm{\psi}^\circ \right) = 
\begin{tikzpicture}[baseline={([yshift=-.5ex]current bounding box.center)}]
      \path(0,0) ellipse (4.5 and 2);
      \tikzstyle{level 1}=[counterclockwise from=-60, level distance=12mm, sibling angle=120]
      \node [draw,circle,fill] (A0) at (0:4.5) {}
	    child  {node [label=-60: {$\scriptstyle{}$}]{}}
    	    child  {node [label=60: {$\scriptstyle{}$}]{}};
       \tikzstyle{level 1}=[counterclockwise from=-90,level distance=12mm,sibling angle=30]
        \node [draw,circle,fill] (A1) at (0:1.5) {}
    	    child {node [label=-90: {$\scriptstyle{}$}]{}}
    	    child {node [label=-90: {$\scriptstyle{}$}]{}};
       \tikzstyle{level 1}=[counterclockwise from=-90,level distance=12mm,sibling angle=60]
        \node [draw,circle,fill] (A2) at (180:1.5) {}
    	    child {node [label=-90: {$\scriptstyle{n}$}]{}};
       \tikzstyle{level 1}=[counterclockwise from=120,level distance=12mm,sibling angle=120]
             \node [draw,circle,fill] (A3) at (180:4.5) {}
             child  {node [label=120: {$\scriptstyle{h_0}$}]{}}
    	     child  {node [label=60: {$\scriptstyle{}$}]{}};
        \path (A0) edge [] node[pos=0.6, above=0.1, label={90:$\psi_{h^-}$}]{} (A1);
        \path (A1) edge [] node[]{} (A2);
        \path (A2) edge [] node[]{} (A3);        
    \end{tikzpicture}
\\
\left(\mathrm{T}^{h^-}, \bm{\psi}^{h^-} \right) = 
\begin{tikzpicture}[baseline={([yshift=-.5ex]current bounding box.center)}]
      \path(0,0) ellipse (4.5 and 2);
      \tikzstyle{level 1}=[counterclockwise from=-60, level distance=12mm, sibling angle=120]
      \node [draw,circle,fill] (A0) at (0:4.5) {}
	    child  {node [label=-60: {$\scriptstyle{}$}]{}}
    	    child  {node [label=60: {$\scriptstyle{}$}]{}};
       \tikzstyle{level 1}=[clockwise from=-90,level distance=12mm,sibling angle=30]
        \node [draw,circle,fill] (A1) at (0:1.5) {}
    	    child {node [label=-90: {$\scriptstyle{n}$}]{}};
       \tikzstyle{level 1}=[clockwise from=-90,level distance=12mm,sibling angle=30]
        \node [draw,circle,fill] (A2) at (180:1.5) {}
    	    child {node [label=-90: {$\scriptstyle{}$}]{}}
    	    child {node [label=-90: {$\scriptstyle{}$}]{}};
       \tikzstyle{level 1}=[counterclockwise from=120,level distance=12mm,sibling angle=120]
             \node [draw,circle,fill] (A3) at (180:4.5) {}
             child  {node [label=120: {$\scriptstyle{h_0}$}]{}}
    	     child  {node [label=60: {$\scriptstyle{}$}]{}};
        \path (A0) edge [] node[]{} (A1);
        \path (A1) edge [] node[]{} (A2);
        \path (A2) edge [] node[pos=0.3, above=0.1, label={90:$\psi_{h^+}$}]{} (A3);        
    \end{tikzpicture}
    \end{array}
    \right.
\end{align*}
    \caption{An example of the decorated trees $\left(\mathrm{T}^\circ, \bm{\psi}^\circ \right)$ and $\left(\mathrm{T}^{h^-}, \bm{\psi}^{h^-} \right)$.}
    \label{fig:Tcircpsicirc}
\end{figure}

More generally, if for some tail half-edge $h^-$ incident to $v_n$, the degree of $\psi_{h^-}$ in $\bm{\psi}$ is positive, then pull-back formulas for $\psi$-classes give more terms:
\begin{equation}
\label{eq:pullbackpsi}
\xi_{\mathrm{T}*}\left( \pi_n^*\left(\bm{\psi}\right) \right) = \xi_{\mathrm{T}*}\left( \bm{\psi}\right) - \xi_{\mathrm{T}^\circ *}\left( \bm{\psi}^\circ\right) 
-\sum_{h^-} \xi_{\mathrm{T}^{h^-} *}\left( \bm{\psi}^{h^-}\right) 
\end{equation}
where $\mathrm{T}^\circ$ and $\bm{\psi}^\circ$ are as in \eqref{eq:Tcirc} and \eqref{eq:psicirc}; the sum is over tail half-edges $h^-$ incident to $v_n$; and for each such tail $h^-$:
\begin{align}
\label{eq:Thminus}
\mathrm{T}^{h^-} &:=\left\{
\begin{array}{l}
\mbox{the tree obtained from $\mathrm{T}$ by}\\ 
\mbox{splitting the vertex $v_n$ into two  vertices $v_n^{h^-}$  and $v^\perp$}\\
\mbox{such that $v_n^{h^-}$ is trivalent and incident to: the leg $n$,}\\ 
\mbox{the tail half-edge equivalent to  $h^-$ from $\mathrm{T}$,}\\ 
\mbox{and an edge incident to $v^\perp$;} 
\end{array}
\right.\\ 
\allowdisplaybreaks
\label{eq:psihminus}
\bm{\psi}^{h^-} &:=\left\{
\begin{array}{l}
\mbox{the monomial obtained from $\bm{\psi}$}\\  
\mbox{by replacing the factor $\psi_{h^-}^{d_-}$,}\\ 
\mbox{where $d_-$ is the degree of $\psi_{h^-}$ in $\bm{\psi}$,}\\ 
\mbox{with $\psi_{h^\perp}^{d^\perp}$ where $h^\perp$ is the tail half-edge incident to $v^\perp$}\\ 
\mbox{in the direction of $v_n^{h^-}$, and  $d^\perp := d_- -1$;}\\
\mbox{if $d_-=0$, then one sets $\bm{\psi}^{h^-} =0$.}
\end{array}
\right.
\end{align}
See the example in Figure \ref{fig:Tcircpsicirc}.

Combining \eqref{eq:pullbackTbar} and \eqref{eq:pullbackpsi} shows the claim for the contributions arising from $\pi_{n}^*\left( \mathsf{Z}(n-1,i,j+1) \right)$.

Regarding the cycles $\mathsf{E}_{\bm{I}}(i,j)$,  the claim follows immediately from \eqref{eq:EIij}. We emphasize that after \eqref{eq:DCI} and \eqref{eq:ditipsineq0}, one has
\begin{equation}
\label{eq:tipsiforEI}
\mbox{if} \,\, d^i_{\mathrm{T},\bm{\psi}} \neq 0 \mbox{ and } i\leq n-2, \,\, \mbox{then}\,\, 
\left\{
\begin{array}{l}
\mbox{$v_n$ is external and not incident to $h_0$,}\\[.1cm]
\mbox{and the degree of $\psi_{h_n}$ in $\bm{\psi}$ is zero. }
\end{array}
\right.
\end{equation}
Hence, if  $(\mathrm{T}, \bm{\psi})$  contributes to $\mathsf{E}_{\bm{I}}(i,j)$ on the LHS of \eqref{eq:partapullback}, then $(\mathrm{T}, \bm{\psi}, w)$ for all $i$-weightings $w$ is in the subset $\mathsf{C}_{0,n+1}^i\subseteq \mathsf{DGW}^i_{0,n+1}$ from \S\ref{sec:partitions}. 
\end{proof}

\begin{lemma}
\label{lemma:Thm1implies2}
For any given $n\geq 3$, Theorem \ref{thm:mainthmgenuszeroij} implies Theorem \ref{thm:mainthmgenuszero}.
\end{lemma}

\begin{proof}
Let $n\geq 3$, and assume Theorem \ref{thm:mainthmgenuszeroij} holds for $n$.

For $i\geq n$, both sides of \eqref{eq:partapullback2} vanish. Indeed, 
for the sets of $i$-weightings to be non-empty, one needs $i<n$, as in \eqref{eq:Wimnonempty}. 
Hence, using the definitions \eqref{eq:Znijw} and \eqref{eq:Ztnij}, both  $\mathsf{Z}(n-1,i,j+1)$ and $ \mathsf{Z}^{\rm t}(n,i,j)$ vanish for $i\geq n$. Similarly, using the definition \eqref{eq:EIij} and \eqref{eq:ditipsineq0}, one concludes the vanishing of the cycles $\mathsf{E}_{\bm{I}}(i,j)$ for $i\geq n$.

For $i=n-1$, we need to prove \eqref{eq:partapullback2} in $A^{j}\left(\ov{\M}_{0,n+1}\right)$ for $j\geq 0$.
As in the previous paragraph, the cycle $\mathsf{Z}(n-1,n-1,j+1)$ on the LHS vanishes for all $j$, since the sets of $i$-weightings are empty in this case. 
For $j\geq 1$, the cycle $\mathsf{Z}^{\rm t}(n,n-1,j)$ on the RHS vanishes from Theorem \ref{thm:mainthmgenuszeroij}$(b)$. 
Moreover, when $j\geq 1$, the only cycles $\mathsf{E}_{\bm{I}}(n-1,j)$ of positive degree $j$ are those with $|\bm{I}|<n-1$, and these vanish 
since all coefficients $d^{n-1}_{\mathrm{T},\bm{\psi}}$  in \eqref{eq:EIij} vanish in this case using \eqref{eq:ditipsineq0}.
For $j=0$, using \eqref{eq:ditipsineq0}, the only non-vanishing cycle  $\mathsf{E}_{\bm{I}}(n-1,0)$ in \eqref{eq:partapullback2} is the one with $|\bm{I}|=n-1$. In this case, both cycles $\mathsf{E}_{\bm{I}}(n-1,0)$ and $\mathsf{Z}^{\rm t}(n,n-1,0)$ consist of only one term contributed by $(\mathrm{T},\bm{\psi})$ where $\mathrm{T}$ is the tree with only one vertex and $\bm{\psi}=0$. In this case, both sides of \eqref{eq:partapullback2} are equal to $-i=-(n-1)$.

Finally, for $1\leq i\leq n-2$, the identity \eqref{eq:partapullback2} follows from part (a) of Theorem \ref{thm:mainthmgenuszeroij}. Hence the statement.
\end{proof}

We now turn to the proof of Theorem \ref{thm:mainthmgenuszeroij}:

\begin{proof}[Proof of Theorem \ref{thm:mainthmgenuszeroij}]
We prove part $(a)$ and part $(b)$ simultaneously by induction on~$n$. 

For part $(b)$, one can assume that $1\leq j < i$, since the vanishing of $\mathsf{Z}^{\rm t}(n,i,j)$  and $\mathsf{Z}(n,i,j)$ in $A^{n-1+j-i}\left(\ov{\M}_{0,n+1}\right)$ for  $1\leq i \leq j$ is  trivial  for degree reasons.
In addition, one can assume $i<n$, since otherwise the sets of $i$-weightings are empty, hence $\mathsf{Z}^{\rm t}(n,i,j)$  and $\mathsf{Z}(n,i,j)$ vanish.

\medskip

The base case is $n=3$, where $i=1$. 
In this case, \eqref{eq:partapullback} is an identity in $A^{1+j}\left(\ov{\M}_{0,4}\right)$, hence non-trivial only for $j\in\{-1,0\}$.

When $j=-1$, the only cycle $\mathsf{E}_{\bm{I}}(1,-1)$ of degree zero is the one with $|\bm{I}|=n-1$, and this vanishes from \eqref{eq:ditipsineq0} since $i<n-1$.
Both sides of \eqref{eq:partapullback} are then equal to the fundamental class of $\ov{\M}_{0,4}$. 

When $j=0$, the term $\pi_{n}^*\left( \mathsf{Z}(n-1,i,j+1) \right)$ does not contribute any cycle in dimension zero.
Hence the LHS of \eqref{eq:partapullback} is provided only by the two $\mathsf{E}_{\bm{I}}(1,0)$, namely
\[
\mbox{
\begin{tikzpicture}[baseline={([yshift=-.5ex]current bounding box.center)}]
      \path(0,0) ellipse (2 and 2);
      \tikzstyle{level 1}=[counterclockwise from=-60, level distance=12mm, sibling angle=120]
      \node [draw,circle,fill] (A0) at (0:1.5) {}
	    child  {node [label=-60: {$\scriptstyle{3}$}]{}}
    	    child  {node [label=60: {$\scriptstyle{1}$}]{}};
       \tikzstyle{level 1}=[counterclockwise from=120,level distance=12mm,sibling angle=120]
        \node [draw,circle,fill] (A1) at (180:1.5) {}
            child {node [label=120: {$\scriptstyle{h_0}$}]{}}
    	    child {node [label=-120: {$\scriptstyle{2}$}]{}};
        \path (A0) edge [] node[]{} (A1);
    \end{tikzpicture}
}
+
\mbox{
\begin{tikzpicture}[baseline={([yshift=-.5ex]current bounding box.center)}]
      \path(0,0) ellipse (2 and 2);
      \tikzstyle{level 1}=[counterclockwise from=-60, level distance=12mm, sibling angle=120]
      \node [draw,circle,fill] (A0) at (0:1.5) {}
	    child  {node [label=-60: {$\scriptstyle{3}$}]{}}
    	    child  {node [label=60: {$\scriptstyle{2}$}]{}};
       \tikzstyle{level 1}=[counterclockwise from=120,level distance=12mm,sibling angle=120]
        \node [draw,circle,fill] (A1) at (180:1.5) {}
            child {node [label=120: {$\scriptstyle{h_0}$}]{}}
    	    child {node [label=-120: {$\scriptstyle{1}$}]{}};
        \path (A0) edge [] node[]{} (A1);
    \end{tikzpicture}
}\]
while the RHS of \eqref{eq:partapullback} has in addition the following terms
\begin{equation}
\label{eq:divM04}
\mbox{
\begin{tikzpicture}[baseline={([yshift=-.5ex]current bounding box.center)}]
      \path(0,0) ellipse (2 and 2);
      \tikzstyle{level 1}=[counterclockwise from=-60, level distance=12mm, sibling angle=120]
      \node [draw,circle,fill] (A0) at (0:1.5) {}
	    child  {node [label=-60: {$\scriptstyle{2}$}]{}}
    	    child  {node [label=60: {$\scriptstyle{1}$}]{}};
       \tikzstyle{level 1}=[counterclockwise from=120,level distance=12mm,sibling angle=120]
        \node [draw,circle,fill] (A1) at (180:1.5) {}
            child {node [label=120: {$\scriptstyle{h_0}$}]{}}
    	    child {node [label=-120: {$\scriptstyle{3}$}]{}};
        \path (A0) edge [] node[]{} (A1);
    \end{tikzpicture}
}
-
\mbox{
\begin{tikzpicture}[baseline={([yshift=-.5ex]current bounding box.center)}]
      \path(0,0) ellipse (1 and 1);
      \tikzstyle{level 1}=[counterclockwise from=180, level distance=12mm, sibling angle=120]
      \node [draw,circle,fill] (A0) at (0:0) {}
            child {node [label=180: {${\psi_{h_0}}$}]{}}
    	    child {node [label=-60: {$\scriptstyle{3}$}]{}}
	    child {node [label=60: {$\scriptstyle{1}$}]{}}
    	    child [grow=0]{node [label=0: {$\scriptstyle{2}$}]{}};
    \end{tikzpicture}
}.
\end{equation}
Since \eqref{eq:divM04} vanishes in $A^{1}\left(\ov{\M}_{0,4}\right)$, part $(a)$ follows.
 
For $n=3$, the  only non-trivial case of part $(b)$  is given when $i=2$ and $j=1$, and it is a special case of Lemma~\ref{lemma:imaxj1van}.

\medskip

Next, we show the following: 

\begin{claim}
\label{claim:part_a}
Assume parts $(a)$ and $(b)$ for $n-1, n-2, \dots, 3$. Then part $(a)$ holds for $n$, i.e.,  
the contributions of all $(\mathrm{T}, \bm{\psi})$
to the two sides of \eqref{eq:partapullback} match, modulo relations in $A^{*}\left(\ov{\M}_{0,n+1}\right)$. 
\end{claim}

\noindent \textit{Proof of Claim \ref{claim:part_a}.}
Let $n\geq 4$.
The RHS of \eqref{eq:partapullback} is $\mathsf{Z}^{\rm t}(n,i,j)$, given by
\begin{equation*}
\mathsf{Z}^{\rm t}(n,i,j) = \sum_{(\mathrm{T}, \bm{\psi},w)}  (-1)^{|\mathrm{H}^+(\mathrm{T})|} \,\, w(\mathrm{T}, \bm{\psi}) \,\, \xi_{\mathrm{T} *}\left( \bm{\psi} \right) \qquad \mbox{in $A^{n-1+j-i} \left(\ov{\M}_{0,n+1}\right)$}
\end{equation*}
where the sum is over $(\mathrm{T}, \bm{\psi},w)$ in the set $\mathsf{DGW}^i_{0,n+1}$ from \eqref{eq:DGW} such that $\xi_{\mathrm{T} *}\left( \bm{\psi} \right)$ 
has  degree $n-1+j-i$ in $A^* \left(\ov{\M}_{0,n+1}\right)$ and the property \eqref{eq:truncatedprop} is satisfied.

For each such $(\mathrm{T}, \bm{\psi},w)$, the argument depends on the position of the leg $n$ in $\mathrm{T}$. 
For this, let $1\leq i \leq  n-2$, and consider the partition of $\mathsf{DGW}^i_{0,n+1}$ from \S\ref{sec:partitions}:
\[
\mathsf{DGW}^i_{0,n+1} = \mathsf{A}_{0,n+1}^i \sqcup \mathsf{B}_{0,n+1}^i \sqcup \mathsf{C}_{0,n+1}^i.
\]

\medskip

\noindent \textit{Contributions from $\mathsf{A}_{0,n+1}^i$.} 
Let us consider the contributions of $(\mathrm{T}, \bm{\psi},w)$ in $\mathsf{A}_{0,n+1}^i$ to $\mathsf{Z}^{\rm t}(n,i,j)$. 
First, we arrange such contributions in parts, and then we distinguish Subcases 1 and 2 below. 

We partition the contributions from $\mathsf{A}_{0,n+1}^i$ to $\mathsf{Z}^{\rm t}(n,i,j)$ as follows. 
The idea is that the partition is suggested by the set of decorated trees contributing to the pull-back of $\psi$-classes, as in \eqref{eq:pullbackpsi},
and each part consists~of:
\begin{enumerate}[(i)]

\item a triple $(\mathrm{T}, \bm{\psi},w)$ where the vertex $v_n$ incident to the leg $n$ has valence at least four and is incident to $h_0$; 

\item possibly a triple $(\mathrm{T}^\circ, \bm{\psi}^\circ, w^\circ)$ where the vertex $v^\circ_n$ incident to the leg $n$ is trivalent and incident to $h_0$;  

\item possibly triples $(\mathrm{T}^{h^-}, \bm{\psi}^{h^-}, w^{h^-})$ where the vertex $v^{h^-}_n$ incident to the leg $n$ is trivalent and adjacent to the vertex incident to $h_0$. 

\end{enumerate}
To define such parts, consider a triple $(\mathrm{T}, \bm{\psi}, w)$ in $\mathsf{A}_{0,n+1}^i$  such that 
 the valence of the vertex $v_n$ incident  to the leg $n$ is at least four. In particular, $v_n$ is incident to $h_0$. 

The possible triple $(\mathrm{T}^\circ, \bm{\psi}^\circ, w^\circ)$ in the part containing $(\mathrm{T}, \bm{\psi}, w)$ is defined as follows: $\mathrm{T}^\circ$ is the graph obtained from $\mathrm{T}$ as in  \eqref{eq:Tcirc}; $\bm{\psi}^\circ$ is the decoration obtained from $\bm{\psi}$ as in \eqref{eq:psicirc}; and $w^\circ$ is the $i$-weighting defined by 
\begin{align*}
w^\circ \left( h^\circ_n,0 \right) &= w \left( h_n,0 \right),\\
w^\circ \left( h^\perp,e \right) &= w \left( h_n,e+1 \right) && \mbox{for $0\leq e\leq d^\perp$,}\\
w^\circ \left( h,e \right) &= w \left( h,e \right) && \mbox{otherwise.}
\end{align*}
Here, we have used the identification
\begin{align*}
\mathrm{H}(\mathrm{T}, \bm{\psi}) &\xrightarrow{\cong} \mathrm{H}(\mathrm{T}^\circ, \bm{\psi}^\circ)\\
(h_n,0)&\mapsto (h^\circ_n,0), \\
(h_n,e+1) &\mapsto (h^\perp, e), && \mbox{for  $0\leq e \leq d^\perp$,}\\
(h,e)& \mapsto (h,e), && \mbox{for $h\neq h_n$ and all $e$.}
\end{align*}
The triple $(\mathrm{T}^\circ, \bm{\psi}^\circ, w^\circ)$ is considered only if $\bm{\psi}^\circ \neq 0$.
Since $(\mathrm{T}, \bm{\psi},w)$ is in $\mathsf{A}_{0,n+1}^i$, by definition the triple $(\mathrm{T}^\circ, \bm{\psi}^\circ, w^\circ)$ is in $\mathsf{A}_{0,n+1}^i$ as well. Moreover, one has the following equality of weights
\begin{equation}
\label{eq:wtipsiiswtipsicirc}
w\left(\mathrm{T},\bm{\psi}\right)=w^\circ\left(\mathrm{T}^\circ,\bm{\psi}^\circ\right).
\end{equation}
Since a divisorial $\psi$-class is exchanged with an edge, the classes $\xi_{\mathrm{T}*}(\bm{\psi})$ and $\xi_{\mathrm{T}^\circ*}(\bm{\psi}^\circ) $ have equal degree in $A^*\left( \ov{\M}_{0,n+1}\right)$.

In addition, the possibly remaining triples of the part containing $(\mathrm{T}, \bm{\psi}, w)$ are triples $(\mathrm{T}^{h^-}, \bm{\psi}^{h^-}, w^{h^-})$, one for each tail half-edge $h^-$ incident to $v_n$ in $\mathrm{T}$, where:  $\mathrm{T}^{h^-}$ is the graph obtained from $\mathrm{T}$ as in  \eqref{eq:Thminus}; $\bm{\psi}^{h^-}$ is the decoration obtained from $\bm{\psi}$ as in \eqref{eq:psihminus}; and $w^{h^-}$ is the $i$-weighting defined by 
\begin{align*}
w^{h^-} \left( h^\perp,e \right) &= w \left( h^-,e \right),&& \mbox{for  $0\leq e \leq d^\perp$,}\\
w^{h^-} \left( \iota(h^\perp), 0 \right) &= w \left( h^-, d^\perp +1 \right), \\
w^{h^-} \left( h,e \right) &= w \left( h,e \right) && \mbox{otherwise.}
\end{align*}
Here, we have used the identification
\begin{align*}
\mathrm{H}(\mathrm{T}, \bm{\psi}) &\xrightarrow{\cong} \mathrm{H}(\mathrm{T}^{h^-}, \bm{\psi}^{h^-})\\
(h^-,e)&\mapsto (h^\perp,e), && \mbox{for  $0\leq e \leq d^\perp$,}\\
(h^-,d^\perp +1) &\mapsto (\iota(h^\perp), 0), \\
(h,e)& \mapsto (h,e), && \mbox{for $h\neq h^-$ and all $e$.}
\end{align*}
For each $h^-$, the triple $(\mathrm{T}^{h^-}, \bm{\psi}^{h^-}, w^{h^-})$ is considered only if $\bm{\psi}^{h^-} \neq 0$.
Since $(\mathrm{T}, \bm{\psi},w)$ is in $\mathsf{A}_{0,n+1}^i$, by definition the triple $(\mathrm{T}^{h^-}, \bm{\psi}^{h^-}, w^{h^-})$ is in $\mathsf{A}_{0,n+1}^i$ as well, and one has
\begin{equation}
\label{eq:wtipsiiswtipsihminus}
w\left(\mathrm{T},\bm{\psi}\right)=w^{h^-}\left(\mathrm{T}^{h^-},\bm{\psi}^{h^-}\right).
\end{equation}
Since a divisorial $\psi$-class is exchanged with an edge, the classes $\xi_{\mathrm{T}*}(\bm{\psi})$ and $\xi_{\mathrm{T}^{h^-}*}(\bm{\psi}^{h^-})$ have equal degree in $A^*\left( \ov{\M}_{0,n+1}\right)$.

\smallskip

For $(\mathrm{T}, \bm{\psi},w)$ to contribute to $\mathsf{Z}^{\rm t}(n,i,j)$,  the degree of $\xi_{\mathrm{T}*}(\bm{\psi})$  is required to be $n-1+j-i$ in $A^*\left( \ov{\M}_{0,n+1}\right)$.
This implies that the corresponding $\xi_{\mathrm{T}^\circ*}(\bm{\psi}^\circ) $ and all $\xi_{\mathrm{T}^{h^-}*}(\bm{\psi}^{h^-})$ have the same degree \mbox{$n-1+j-i$}.
In addition, $(\mathrm{T}^\circ, \bm{\psi}^\circ, w^\circ)$ satisfies the property \eqref{eq:truncatedprop} by construction;
$(\mathrm{T}, \bm{\psi},w)$ and all $(\mathrm{T}^{h^-}, \bm{\psi}^{h^-}, w^{h^-})$ trivially satisfy \eqref{eq:truncatedprop}.
This implies that if $(\mathrm{T}, \bm{\psi},w)$  contributes to $\mathsf{Z}^{\rm t}(n,i,j)$, then $(\mathrm{T}^\circ, \bm{\psi}^\circ, w^\circ)$ and all $(\mathrm{T}^{h^-}, \bm{\psi}^{h^-}, w^{h^-})$ contribute to $\mathsf{Z}^{\rm t}(n,i,j)$ as well.

The number of head half-edges of $\mathrm{T}$ and corresponding $\mathrm{T}^\circ$ varies by one. Similarly, the number of head half-edges of $\mathrm{T}$ and each corresponding $\mathrm{T}^{h^-}$ varies by one. Using \eqref{eq:wtipsiiswtipsicirc} and \eqref{eq:wtipsiiswtipsihminus}, it follows that the sum of the contributions to $\mathsf{Z}^{\rm t}(n,i,j)$ of all triples in  the part of $\mathsf{A}_{0,n+1}^i$ corresponding to $(\mathrm{T}, \bm{\psi},w)$ is:
\begin{equation}
\label{eq:Aexp}
(-1)^{|\mathrm{H}^+(\mathrm{T})|} \, w(\mathrm{T}, \bm{\psi}) \left( \xi_{\mathrm{T}*}(\bm{\psi}) - \xi_{\mathrm{T}^\circ*}(\bm{\psi}^\circ) - \sum_{h^-} \xi_{\mathrm{T}^{h^-}*}(\bm{\psi}^{h^-})\right).
\end{equation}
As in \eqref{eq:pullbackpsi}, this is
\begin{equation}
\label{eq:totaltipsipair}
(-1)^{|\mathrm{H}^+(\mathrm{T})|} \, w(\mathrm{T}, \bm{\psi}) \, \xi_{\mathrm{T}*}(\pi_n^* \,\bm{\psi}) \quad \in A^*\left(\ov{\M}_{0,n+1} \right).
\end{equation}
This is the total contribution to $\mathsf{Z}^{\rm t}(n,i,j)$ of  the part of $\mathsf{A}_{0,n+1}^i$ corresponding to $(\mathrm{T}, \bm{\psi},w)$.
We distinguish two subcases:

\medskip

\noindent \textit{Subcase 1.}
Consider first the case when the sum of the degrees of $\psi_{h}$ in $\bm{\psi}$ for $h$ incident to $v_n$ in $\mathrm{T}$ is
$\mathrm{valence}(v_n)-3$. 
In this case,  the total contribution  to $\mathsf{Z}^{\rm t}(n,i,j)$ as given in \eqref{eq:totaltipsipair} vanishes. 
Indeed, our assumption implies that the sum of the degrees of $\pi_n^* \,\psi_{h}$ in the term $\pi_n^*\, \bm{\psi}$ from \eqref{eq:totaltipsipair} for $h$ incident to $v_n$  is
$\mathrm{valence}(v_n)-3$. 
The product of $\pi_n^* \,\psi_{h}$ in $\pi_n^*\, \bm{\psi}$ for $h$ incident to $v_n$ is then the pull-back via $\pi_n$ of a cycle of degree $\mathrm{valence}(v_n)-3$ supported on a moduli space of rational  curves with $\mathrm{valence}(v_n)-1$ marked points. Such cycle vanishes for degree reasons.

Similarly, no such contributions arise in $\pi_{n}^*\left( \mathsf{Z}(n-1,i,j+1) \right)$. Indeed, arguing as in Lemma \ref{lemma:leftrightterms}, contributions to $\pi_{n}^*\left( \mathsf{Z}(n-1,i,j+1) \right)$ arise precisely from terms like $\xi_{\mathrm{T}*}(\pi_n^* \,\bm{\psi})$, hence they are null in this case, as in the previous paragraph.
Also, since $(\mathrm{T}, \bm{\psi},w)$,  $(\mathrm{T}^\circ, \bm{\psi}^\circ, w^\circ)$ and $(\mathrm{T}^{h^-}, \bm{\psi}^{h^-}, w^{h^-})$ are all in $\mathsf{A}_{0,n+1}^i$, there are no contributions to the cycles  $\mathsf{E}_{\bm{I}}(i,j)$. 

It follows that the sum of the contributions of such decorated trees vanishes on both sides of~\eqref{eq:partapullback}.

\medskip

\noindent \textit{Subcase 2.}
It remains to consider the case when the sum of the degrees of $\psi_{h}$ in $\bm{\psi}$ for $h$ incident to $v_n$ in $\mathrm{T}$ is less than $\mathrm{valence}(v_n)-3$.
After \eqref{eq:pullbackTbar}, the contribution \eqref{eq:totaltipsipair}  to $\mathsf{Z}^{\rm t}(n,i,j)$ in this case matches a contribution to 
 \[
 (-1)^{|\mathrm{H}^+(\mathrm{T})|} \, w(\mathrm{T}, \bm{\psi}) \, \pi_n^*\left( \xi_{\ov{\mathrm{T}}*}(\bm{\psi})\right), 
 \]
 where $\ov{\mathrm{T}}$ is the tree obtained by removing the leg $n$ from $\mathrm{T}$. Finally, this matches the contribution 
 \[
 (-1)^{|\mathrm{H}^+(\ov{\mathrm{T}})|} \, w(\ov{\mathrm{T}}, \bm{\psi}) \, \pi_n^*\left( \xi_{\ov{\mathrm{T}}*}(\bm{\psi})\right), 
 \] 
 to $\pi_n^*\left( \mathsf{Z}(n-1,i,j+1)\right)$ in the LHS of \eqref{eq:partapullback}. 
Indeed, recall that $h_0$ is incident to $v_n$ in $\mathrm{T}$, and 
$i\leq n-2$. In particular, 
removing the leg $n$ in $\mathrm{T}$ does not affect the computation of the capacities of the head half-edges, as defined in \eqref{eq:lenheads}. It follows that $(\mathrm{T}, \bm{\psi})$ and $(\ov{\mathrm{T}}, \bm{\psi})$ have the same $i$-weightings, hence $w(\mathrm{T}, \bm{\psi})=w(\ov{\mathrm{T}}, \bm{\psi})$. Also, one has $|\mathrm{H}^+(\mathrm{T})|=|\mathrm{H}^+(\ov{\mathrm{T}})|$.

Similarly to Subcase 1, there are no contributions to the cycles  $\mathsf{E}_{\bm{I}}(i,j)$. Hence such decorated trees have equal contributions to the two sides of \eqref{eq:partapullback}.

This concludes the discussion of the case $(\mathrm{T}, \bm{\psi},w)\in \mathsf{A}_{0,n+1}^i$.

\medskip

\noindent \textit{Contributions from  $\mathsf{B}_{0,n+1}^i \sqcup \mathsf{C}_{0,n+1}^i$.} 
Consider  $(\mathrm{T}, \bm{\psi},w)\in \mathsf{B}_{0,n+1}^i \sqcup \mathsf{C}_{0,n+1}^i$ contributing to $\mathsf{Z}^{\rm t}(n,i,j)$. 
We define a subset $\mathsf{P}_{\mathrm{T}, \bm{\psi},w}$ of $\mathsf{B}_{0,n+1}^i \sqcup \mathsf{C}_{0,n+1}^i$
 assigned to $(\mathrm{T}, \bm{\psi},w)$, and use the inductive hypothesis to show that the sum of the contributions of the triples in $\mathsf{P}_{\mathrm{T}, \bm{\psi},w}$ to 
 $\mathsf{Z}^{\rm t}(n,i,j)$ is equivalent to a contribution to the LHS of \eqref{eq:partapullback}.

To define the set $\mathsf{P}_{\mathrm{T}, \bm{\psi},w}$, we proceed as follows.
Consider 
\[
\mathrm{T}^{\rm ext} := \!\!
\begin{array}{l}
\mbox{ the maximal connected  subtree of $\mathrm{T}$}\\
 \mbox{ containing the leg $n$, but not containing $h_0$.}
 \end{array}
\]
Let $n^{\rm ext}$ be the number of legs of $\mathrm{T}$ within $\mathrm{T}^{\rm ext}$. One has $n^{\rm ext}< n$.

Define $\mathrm{T}^{\rm int}$ to be the complement of $\mathrm{T}^{\rm ext}$ in $\mathrm{T}$. 
There is a unique edge 
in $\mathrm{E}(\mathrm{T})$ with tail half-edge 
in $\mathrm{T}^{\rm int}$ and head half-edge in $\mathrm{T}^{\rm ext}$. Let $h_0^{\rm ext}$ be such head half-edge in $\mathrm{T}^{\rm ext}$. 

The set $\mathrm{H}^\pm(\mathrm{T}^{\rm ext})$ is the restriction of $\mathrm{H}^\pm(\mathrm{T})$ to $\mathrm{T}^{\rm ext}$, and define similarly $\mathrm{H}^\pm(\mathrm{T}^{\rm int})$ as the restriction of $\mathrm{H}^\pm(\mathrm{T})$ to $\mathrm{T}^{\rm int}$.
One has
\[
\mathrm{H}(\mathrm{T}) = \mathrm{H}(\mathrm{T}^{\rm ext}) \sqcup \mathrm{H}(\mathrm{T}^{\rm int}).
\]

Factoring $\bm{\psi}$ as $\bm{\psi}= \prod_{h\in \mathrm{H}(\mathrm{T})} \psi_{h}^{d_h}$,  consider 
its restrictions 
\begin{align*}
\bm{\psi}^{\rm ext} &:= \prod_{h\in \mathrm{H}(\mathrm{T}^{\rm ext})} \psi_{h}^{d_h} 
& \mbox{and} &&
\bm{\psi}^{\rm int} &:=\prod_{h\in \mathrm{H}(\mathrm{T}^{\rm int})} \psi_{h}^{d_h}
\end{align*}
to $\mathrm{T}^{\rm ext}$ and $\mathrm{T}^{\rm int}$, respectively. 
One has $\bm{\psi}= \bm{\psi}^{\rm int} \,\bm{\psi}^{\rm ext}$. This induces the partition
\[
\mathrm{H}(\mathrm{T}, \bm{\psi}) =
\mathrm{H}(\mathrm{T}^{\rm int}, \bm{\psi}^{\rm int}) \sqcup \mathrm{H}(\mathrm{T}^{\rm ext}, \bm{\psi}^{\rm ext}).
\]

The $i$-weighting $w$ in $\mathsf{W}^i_{\mathrm{T}, \bm{\psi}}$ is the union of two functions 
\[
w=w^{\rm int}\sqcup w^{\rm ext},
\]
where $w^{\rm int}$ and $w^{\rm ext}$ are the restrictions of $w$ to $\mathrm{H}(\mathrm{T}^{\rm int}, \bm{\psi}^{\rm int})$ and $\mathrm{H}(\mathrm{T}^{\rm ext}, \bm{\psi}^{\rm ext})$, respectively. One has
\begin{align}
\label{eq:wintwext}
w^{\rm int} \in \mathsf{W}^{i,a}_{{\mathrm{T}}^{\rm int}, {\bm{\psi}}^{\rm int}}&& \mbox{and} &&
w^{\rm ext} \in \mathsf{W}^a_{{\mathrm{T}}^{\rm ext}, {\bm{\psi}}^{\rm ext}} && \mbox{for some $a$.} 
\end{align}
Here, $\mathsf{W}^{i,a}_{{\mathrm{T}}^{\rm int}, {\bm{\psi}}^{\rm int}}$ is as in \eqref{eq:Wimtipsi}, with the tail half-edge $\iota(h_0^{\rm ext})$ in ${\mathrm{T}}^{\rm int}$ playing the role of the first leg in \eqref{eq:Wimtipsi},
 and $\mathsf{W}^a_{{\mathrm{T}}^{\rm ext}, {\bm{\psi}}^{\rm ext}}$ is as in \eqref{eq:Wi1tipsi}.

The function $w^{\rm int}$ gives a lower bound for $a$. For this, consider the tail $h^-:=\iota(h_0^{\rm ext})$ which is half of the edge connecting $\mathrm{T}^{\rm int}$ and $\mathrm{T}^{\rm ext}$.
If $d_-$ is the degree of $\psi_{h^-}$ in $\bm{\psi}$, let 
\begin{equation}
\label{eq:a0}
a_0:= w^{\rm int}\left(h^-,  d_- \right).
\end{equation}
By the definition of the weighting $w$, one has $w^{\rm ext}(h_0^{\rm ext}, 0)=a$ for some $a> a_0$.
 
 Since $(\mathrm{T}, \bm{\psi},w)$ is in $\mathsf{B}_{0,n+1}^i \sqcup \mathsf{C}_{0,n+1}^i$, one has that $w^{\rm ext}$ is an $a$-weighting of $({\mathrm{T}}^{\rm ext}, {\bm{\psi}}^{\rm ext})$, for some $a$, satisfying the property \eqref{eq:truncatedprop}.

The set $\mathsf{P}_{\mathrm{T}, \bm{\psi},w}$ consists of 
those triples $(\mathrm{T}', \bm{\psi}', w')$ in $\mathsf{B}_{0,n+1}^i \sqcup \mathsf{C}_{0,n+1}^i$ contributing to $\mathsf{Z}^{\rm t}(n,i,j)$ which are obtained by fixing $\mathrm{T}^{\rm int}$, $\bm{\psi}^{\rm int}$,  $w^{\rm int}$, and varying $\mathrm{T}^{\rm ext}$, $\bm{\psi}^{\rm ext}$, $w^{\rm ext}$. 
Namely, define $\mathsf{P}_{\mathrm{T}, \bm{\psi},w} \subseteq  \mathsf{B}_{0,n+1}^i \sqcup \mathsf{C}_{0,n+1}^i$ as
\[
\mathsf{P}_{\mathrm{T}, \bm{\psi},w} := \left\{ 
\left(\mathrm{T}', \bm{\psi}', w' \right) \, \Bigg| \, 
\begin{array}{l}
{\mathrm{T}'}^{\rm int} = \mathrm{T}^{\rm int}, \quad {\bm{\psi}'}^{\rm int} = \bm{\psi}^{\rm int}, \quad w'^{\rm int} = w^{\rm int}, \\[4pt]
\deg \xi_{\mathrm{T}' *}\left( \bm{\psi}'\right) = n-1+j-i \quad \mbox{in $A^*\left( \ov{\M}_{0,n+1}\right)$}
\end{array}
\right\}.
\]
Since $\mathsf{P}_{\mathrm{T}, \bm{\psi},w} \subseteq  \mathsf{B}_{0,n+1}^i \sqcup \mathsf{C}_{0,n+1}^i$, for each $\left(\mathrm{T}', \bm{\psi}', w' \right)$ in $\mathsf{P}_{\mathrm{T}, \bm{\psi},w}$ one has that 
\[
w'^{\rm ext} \in \mathsf{W}^a_{{\mathrm{T}'}^{\rm ext}, {\bm{\psi}'}^{\rm ext}} \qquad \mbox{for some $a > a_0$}
\]
and $w'^{\rm ext}$ satisfies the property \eqref{eq:truncatedprop}.

The sum of the contributions of all $(\mathrm{T}', \bm{\psi}', w')$ in $\mathsf{P}_{\mathrm{T}, \bm{\psi},w}$ to $\mathsf{Z}^{\rm t}(n,i,j)$ consists of the following term:
\begin{equation}
\label{eq:localsumBC}
\sum_{(\mathrm{T}', \bm{\psi}', w')\in \mathsf{P}_{\mathrm{T}, \bm{\psi},w}}  (-1)^{|\mathrm{H}^+(\mathrm{T}')|} \,\, w'(\mathrm{T}', \bm{\psi}') \,\,\xi_{\mathrm{T}' *}\left( \bm{\psi}' \right).
\end{equation}

We need to show that \eqref{eq:localsumBC} is equivalent to a contribution  to the LHS of \eqref{eq:partapullback}.
For this, we first rewrite \eqref{eq:localsumBC} so that we can apply the inductive hypothesis for the triples $\left({\mathrm{T}'}^{\rm ext}, {\bm{\psi}'}^{\rm ext}, {w'}^{\rm ext} \right)$.

\smallskip

For $\left(\mathrm{T}', \bm{\psi}', w' \right) \in \mathsf{P}_{\mathrm{T}, \bm{\psi},w}$,
the $i$-weighting $w'\in \mathsf{W}^i_{\mathrm{T}', \bm{\psi}'}$ is the union of two functions $w'=w^{\rm int}\sqcup {w'}^{\rm ext}$ such that 
\begin{align*}
w^{\rm int} \in \mathsf{W}^{i,a}_{{\mathrm{T}}^{\rm int}, {\bm{\psi}}^{\rm int}}&& \mbox{and} &&
w^{\rm ext} \in \mathsf{W}^a_{{\mathrm{T}'}^{\rm ext}, {\bm{\psi}'}^{\rm ext}} && \mbox{for some $a > a_0$}
\end{align*}
as in \eqref{eq:wintwext}.
From \eqref{eq:imweighttipsi}, the corresponding $i$-weight of $(\mathrm{T}', \bm{\psi}')$  factors as
\[
w'(\mathrm{T}', \bm{\psi}') =  w^{\rm int} \left( {\mathrm{T}}^{\rm int}, {\bm{\psi}}^{\rm int} \right) \,\, {w'}^{\rm ext} \left( {\mathrm{T}'}^{\rm ext}, {\bm{\psi}'}^{\rm ext} \right).
\]
We can thus rewrite \eqref{eq:localsumBC} as
\begin{equation}
\label{eq:localsumBCtake2}
w^{\rm int} \left( {\mathrm{T}}^{\rm int}, {\bm{\psi}}^{\rm int} \right) 
 \,\, \sum_{(\mathrm{T}', \bm{\psi}', w') \in \mathsf{P}_{\mathrm{T}, \bm{\psi},w}}  (-1)^{|\mathrm{H}^+(\mathrm{T}')|} \, \,
{w'}^{\rm ext}\left({{\mathrm{T}'}^{\rm ext}, {\bm{\psi}'}^{\rm ext}}\right)
 \,\,\xi_{\mathrm{T}' *}\left( \bm{\psi}' \right).
\end{equation}
One has 
\begin{equation}
\label{eq:EHTintText}
\mathrm{H}^+\left(\mathrm{T}' \right) = \mathrm{H}^+\left(\mathrm{T}^{\rm int}\right) \sqcup \mathrm{H}^+\left({\mathrm{T}'}^{\rm ext}\right).
\end{equation}
Moreover, each map $\xi_{\mathrm{T}'}$ factors as $\xi_{\mathrm{T}'} = \zeta_{\mathrm{T}}\circ \xi_{{\mathrm{T}'}^{\rm ext}}$, where 
\[
\zeta_{\mathrm{T}}\colon   \ov{\M}_{0,n^{\rm ext}+1} \times \prod_{v\in \mathrm{V}(\mathrm{T}^{\rm int})} \ov{\M}_{0, n(v)} \longrightarrow \ov{\M}_{0,n+1}
\]
is the gluing map of degree one  defined by $\mathrm{T}^{\rm int}$ and the edge connecting $\mathrm{T}^{\rm int}$ and $\mathrm{T}^{\rm ext}$.
Hence we have
\begin{equation}
\label{eq:xiT'factored}
\xi_{\mathrm{T}' *}\left( \bm{\psi}' \right) = \zeta_{\mathrm{T} *}\left( {\bm{\psi}}^{\rm int} \,\, \xi_{{\mathrm{T}'}^{\rm ext} *}\left( {\bm{\psi}'}^{\rm ext} \right)\right).
\end{equation}
By definition of the truncated cycles $\mathsf{Z}^{\rm t}(n^{\rm ext},a,b)$, we have
\begin{multline}
\label{eq:Znextab}
\sum_{(\mathrm{T}', \bm{\psi}', w') \in \mathsf{P}_{\mathrm{T}, \bm{\psi},w}}  (-1)^{|\mathrm{H}^+({\mathrm{T}'}^{\rm ext})|} \, \,
{w'}^{\rm ext}\left({{\mathrm{T}'}^{\rm ext}, {\bm{\psi}'}^{\rm ext}}\right)
 \,\,\xi_{{\mathrm{T}'}^{\rm ext} *}\left( {\bm{\psi}'}^{\rm ext} \right)\\
= \sum_{a> a_0}\,\, \mathsf{Z}^{\rm t}(n^{\rm ext},a,b)
\end{multline}
where for each $a$, the number $b$ is determined by
\[
n^{\rm ext} -1 +b-a= \deg\, \xi_{{\mathrm{T}}^{\rm ext}\, *}\left( {\bm{\psi}}^{\rm ext}\right).
\]
Since $\mathsf{W}^a_{{\mathrm{T}'}^{\rm ext}, {\bm{\psi}'}^{\rm ext}}=\varnothing$ for $a\geq  n^{\rm ext}$, 
the sum on the RHS of \eqref{eq:Znextab} is finite. 

Using \eqref{eq:EHTintText}, \eqref{eq:xiT'factored}, and \eqref{eq:Znextab}, it follows that \eqref{eq:localsumBCtake2} can be rewritten as
\begin{equation}
\label{eq:localsumBCtake3}
(-1)^{|\mathrm{H}^+(\mathrm{T}^{\rm int})|}\,\, w^{\rm int} \left( {\mathrm{T}}^{\rm int}, {\bm{\psi}}^{\rm int} \right) 
\sum_{a > a_0} \,\, \zeta_{\mathrm{T}*} \left( {\bm{\psi}}^{\rm int} \, \mathsf{Z}^{\rm t}(n^{\rm ext},a,b) \right).
\end{equation}
Formula \eqref{eq:localsumBCtake3} expresses the term of the RHS of \eqref{eq:partapullback} contributed by all triples  in $\mathsf{P}_{\mathrm{T}, \bm{\psi},w}$.

\smallskip

Since $n^{\rm ext}<n$, the induction hypothesis together with Lemma \ref{lemma:Thm1implies2} for $n^{\rm ext}$ imply that for each $a$, one has
\begin{equation}
\label{eq:Znext}
\pi_{n^{\rm ext}}^*\left( \mathsf{Z}(n^{\rm ext} -1,a,b +1) \right) - \sum_{\bm{I}} |\bm{I}| \, \mathsf{E}_{\bm{I}}(a,b) = \mathsf{Z}^{\rm t}(n^{\rm ext},a,b)
\end{equation}
in $A^{n^{\rm ext} -1+b-a}\left(\ov{\M}_{0,n^{\rm ext} +1}\right)$. 
In the above sum, $\bm{I}$ ranges over all non-empty subsets of the restriction of the set of legs $\mathrm{L}(\mathrm{T})\setminus \{n\}$ to $\mathrm{T}^{\rm ext}$.

Using \eqref{eq:Znext} to replace $\mathsf{Z}^{\rm t}(n^{\rm ext},a,b)$ in \eqref{eq:localsumBCtake3}, 
we have that the term of the RHS of \eqref{eq:partapullback} contributed by all triples $(\mathrm{T}', \bm{\psi}', w')$  in $\mathsf{P}_{\mathrm{T}, \bm{\psi},w}$~is
\begin{multline}
\label{eq:LHS_ext}
(-1)^{|\mathrm{H}^+(\mathrm{T}^{\rm int})|} \,\,w^{\rm int} \left( {\mathrm{T}}^{\rm int}, {\bm{\psi}}^{\rm int} \right) \\
\times \sum_{a > a_0}\,\, \zeta_{\mathrm{T}*} \left( {\bm{\psi}}^{\rm int} \left(
\pi_{n^{\rm ext}}^*\left( \mathsf{Z}(n^{\rm ext} -1,a,b +1) \right) - \sum_{\bm{I} } |\bm{I}| \, \mathsf{E}_{\bm{I}}(a,b) \right)
\right).
\end{multline}

\smallskip

Next, we verify that  \eqref{eq:LHS_ext}  matches  a contribution  to the LHS of \eqref{eq:partapullback}.  For this, we proceed by considering four subcases. This will conclude the proof of Claim \ref{claim:part_a}.

Let $(\mathrm{T}', \bm{\psi}', w')\in \mathsf{P}_{\mathrm{T}, \bm{\psi},w}$. Let $v_n'$ be the vertex of $\mathrm{T}'$ incident to the leg~$n$, and $h'_n$ be the head half-edge incident to $v_n'$.
Let ${\mathrm{T}'}_0^{\rm ext}$ be the maximal subtree of ${\mathrm{T}'}^{\rm ext}$ containing  $v_n'$ such that 
for every vertex $v$  of ${\mathrm{T}'}_0^{\rm ext}$, the sum of the degrees of $\psi_{h}$ in $\bm{\psi}'$ for $h$ incident to $v$ is $\mathrm{valence}(v) -3$. 
In particular, one has ${\mathrm{T}'}_0^{\rm ext}\not= \varnothing$ if and only if the sum of the degrees of $\psi_{h}$ in $\bm{\psi}'$ for $h$ incident to $v'_n$ is equal to $\mathrm{valence}(v'_n)-3$. For instance, ${\mathrm{T}'}_0^{\rm ext}\not= \varnothing$ when $v'_n$ is trivalent --- in this case, the degree of $\psi_{h}$ in $\bm{\psi}'$ for $h$ incident to $v'_n$ is necessarily zero, otherwise $\xi_{\mathrm{T}' *}\left( \bm{\psi}' \right)=0$.

\medskip

\noindent \textit{Subcase 3.}
Assume first that ${\mathrm{T}'}_0^{\rm ext}\not= \varnothing$, and if trivalent, then $v_n'$ is  adjacent only to vertices in ${\mathrm{T}'}_0^{\rm ext}$. 
Moreover,  assume that  $v_n'$ is not trivalent and external in $\mathrm{T}'$ --- the case when $v_n'$ is trivalent and external  will be treated in Subcase 5.
Following the proof of Lemma \ref{lemma:leftrightterms},
there is no contribution from such $(\mathrm{T}',  \bm{\psi}')$ to the LHS of \eqref{eq:partapullback}. 
Indeed, the conditions that ${\mathrm{T}'}_0^{\rm ext}\not= \varnothing$ and a trivalent $v_n'$ be adjacent only to vertices in ${\mathrm{T}'}_0^{\rm ext}$
imply that such a decorated tree does not contribute to $\pi_n^*\left( \mathsf{Z}(n -1,i,j +1)\right)$; moreover, the conditions that 
${\mathrm{T}'}_0^{\rm ext}\not= \varnothing$ and $v_n'$ is not trivalent and external in $\mathrm{T}'$ imply that such a decorated tree does not contribute
 to  the cycles $\mathsf{E}_{\bm{I}}(i,j)$.
Similarly, the decorated tree $({\mathrm{T}'}^{\rm ext}, {\bm{\psi}'}^{\rm ext} )$ 
does not contribute to the LHS of \eqref{eq:Znext}.
Hence, the decorated tree $(\mathrm{T}', \bm{\psi}' )$ contributes neither to the LHS of \eqref{eq:partapullback}, nor to \eqref{eq:LHS_ext}, as desired.

\medskip

\noindent \textit{Subcase 4.} 
Consider $(\mathrm{T}', \bm{\psi}', w')\in \mathsf{B}_{0,n+1}^i$. 
Then from \eqref{eq:tipsiforEI}, $(\mathrm{T}', \bm{\psi}')$ does not contribute to the classes $\mathsf{E}_{\bm{I}}(i,j)$ in the LHS of \eqref{eq:partapullback}.
Assume that either ${\mathrm{T}'}_0^{\rm ext} = \varnothing$, or $v_n'$ is  trivalent and adjacent to at least one vertex not in ${\mathrm{T}'}_0^{\rm ext}$.
In these cases, $(\mathrm{T}', \bm{\psi}')$ contributes to $\pi_{n}^*\left( \mathsf{Z}(n -1,i,j +1) \right)$ in the LHS of \eqref{eq:partapullback}. 

When ${\mathrm{T}'}_0^{\rm ext} = \varnothing$, the total contribution of $(\mathrm{T}', \bm{\psi}')$ to the LHS of \eqref{eq:partapullback} is
\begin{equation}
\label{eq:contLHSsubcase4}
(-1)^{|\mathrm{H}^+(\mathrm{T}')|} \,\, c^i_{\overline{\mathrm{T}'}, \bm{\psi}'} \,\,\xi_{\mathrm{T}' *}\left( \bm{\psi}' \right) 
\end{equation}
where $\left(\overline{\mathrm{T}'}, \bm{\psi}'\right)$ is the decorated tree with  $\overline{\mathrm{T}'}$  obtained by removing the leg $n$ from $\mathrm{T}'$. 
Similarly, consider the decorated tree $\left( {\overline{\mathrm{T}'}}^{\rm ext}, {\bm{\psi}' }^{\rm ext}\right)$ where ${\overline{\mathrm{T}'}}^{\rm ext}$ is the graph obtained by removing the leg $n$ from ${\mathrm{T}'} ^{\rm ext}$. 

The term in \eqref{eq:contLHSsubcase4} can be decomposed as follows.
As in \eqref{eq:wintwext}, an $i$-weighting $\overline{w}'$ of $\left(\overline{\mathrm{T}'}, \bm{\psi}' \right)$ decomposes as $\overline{w}'=w^{\rm int}\sqcup \overline{w}'^{\rm ext}$, where $w^{\rm int}$ is an $i$-weighting of $\left( {\mathrm{T}}^{\rm int}, {\bm{\psi}}^{\rm int} \right)$ and $\overline{w}'^{\rm ext}$ is an $a$-weighting of $({\overline{\mathrm{T}'}}^{\rm ext}, {\bm{\psi}'}^{\rm ext})$ for some $a$.
 Hence, we have
\begin{equation}
\label{eq:cidecomp}
c^i_{\overline{\mathrm{T}'}, \bm{\psi}'} = \sum_{w^{\rm int}} \,\,w^{\rm int} \left( {\mathrm{T}}^{\rm int}, {\bm{\psi}}^{\rm int} \right) 
\,\sum_{a > a_0} \,\,
 c^a_{{\overline{\mathrm{T}'}}^{\rm ext}, {\bm{\psi}'}^{\rm ext}}.
\end{equation}
The number $a_0$ is determined by each $w^{\rm int}$, as in \eqref{eq:a0}.

Fixing $w^{\rm int}$, consider the triples $(\mathrm{T}', \bm{\psi}', \overline{w}')$ for all $\overline{w}'$ whose restriction to $\left( {\mathrm{T}}^{\rm int}, {\bm{\psi}}^{\rm int} \right)$ is  $w^{\rm int}$. Restricting \eqref{eq:contLHSsubcase4}, the  contribution of  all such triples to the LHS of \eqref{eq:partapullback} is
\begin{equation*}
w^{\rm int} \left( {\mathrm{T}}^{\rm int}, {\bm{\psi}}^{\rm int} \right) \, \sum_{a> a_0} \,\,
(-1)^{|\mathrm{H}^+(\mathrm{T}')|} \,\, c^a_{{\overline{\mathrm{T}'}}^{\rm ext}, {\bm{\psi}'}^{\rm ext}}
\,\,\xi_{\mathrm{T}' *}\left( \bm{\psi}' \right).
\end{equation*}
Using \eqref{eq:EHTintText} and \eqref{eq:xiT'factored}, this is
\begin{multline}
\label{eq:contLHSsubcase4final}
 (-1)^{|\mathrm{H}^+(\mathrm{T}^{\rm int})|} \,\, w^{\rm int} \left( {\mathrm{T}}^{\rm int}, {\bm{\psi}}^{\rm int} \right)   \\
\times \sum_{a > a_0} \,\, \zeta_{\mathrm{T}*} \left( {\bm{\psi}}^{\rm int} \, (-1)^{|\mathrm{H}^+({\mathrm{T}'}^{\rm ext})|} 
\,\, c^a_{{\overline{\mathrm{T}'}}^{\rm ext}, {\bm{\psi}'}^{\rm ext}}
\,\,\xi_{{\mathrm{T}'}^{\rm ext} *}\left( {\bm{\psi}'}^{\rm ext} \right) \right).
\end{multline}
This formula gives the contribution to the LHS of \eqref{eq:partapullback} of the triples $(\mathrm{T}', \bm{\psi}', \overline{w}')$ for all $\overline{w}'$ whose restriction to $\left( {\mathrm{T}}^{\rm int}, {\bm{\psi}}^{\rm int} \right)$ is  $w^{\rm int}$.
Here the term 
\[
(-1)^{|\mathrm{H}^+({\mathrm{T}'}^{\rm ext})|} 
\,\, c^a_{{\overline{\mathrm{T}'}}^{\rm ext}, {\bm{\psi}'}^{\rm ext}}
\,\,\xi_{{\mathrm{T}'}^{\rm ext} *}\left( {\bm{\psi}'}^{\rm ext} \right)
\]
matches the contribution of $( {\mathrm{T}'}^{\rm ext}, {\bm{\psi}'}^{\rm ext} )$ to $\pi_{n^{\rm ext}}^*\left( \mathsf{Z}(n^{\rm ext} -1,a,b +1) \right)$. 

It follows that the contribution to the LHS of \eqref{eq:partapullback}  of the triples $(\mathrm{T}', \bm{\psi}', \overline{w}')$ for all $\overline{w}'$ with fixed $w^{\rm int}$  
matches the contribution to \eqref{eq:LHS_ext} of $(\mathrm{T}', \bm{\psi}', w')$ for all $w'$ with fixed $w^{\rm int}$, as desired.

When $v_n'$ is trivalent, the term in \eqref{eq:contLHSsubcase4} has to be replaced by
\begin{equation*}
(-1)^{|\mathrm{H}^+(\mathrm{T}')|} \,\, \left( c^i_{\ov{\mathrm{T}'}, \bm{\psi}^+} + c^i_{\ov{\mathrm{T}'}, \bm{\psi}^-} \right)\,\,\xi_{\mathrm{T}' *}\left( \bm{\psi}' \right) 
\end{equation*}
where: the tree $\ov{\mathrm{T}'}$ is obtained from $\mathrm{T}'$ by removing the leg $n$ and contracting an edge to stabilize; 
the decoration $\bm{\psi}^+$ is defined so that $\bm{\psi}'$ is obtained from $\bm{\psi}^+$ as in \eqref{eq:psicirc}; and $\bm{\psi}^-$ is defined so that $\bm{\psi}'$ is obtained from $\bm{\psi}^-$ as in \eqref{eq:psihminus}. 
If $v_n'$ is trivalent and adjacent to $v_0$, then the contribution  $c^i_{\ov{\mathrm{T}'}, \bm{\psi}^-}\,\,\xi_{\mathrm{T}' *}\left( \bm{\psi}' \right) $ arises as in \eqref{eq:Aexp}, and so it has already been discussed in Subcases 1--2. For all the other contributions, the argument then continues as in the case when ${\mathrm{T}'}_0^{\rm ext} = \varnothing$.

\medskip

\noindent \textit{Subcase 5.} 
Next, consider $(\mathrm{T}', \bm{\psi}', w')\in \mathsf{C}_{0,n+1}^i $, and
assume that $v'_n$ is trivalent and external in $\mathrm{T}'$. 
In this case, $(\mathrm{T}', \bm{\psi}')$ contributes to the class $\mathsf{E}_{\bm{I}}(i,j)$  such that $\bm{I}\sqcup \{n\}$ is the set of the two legs incident to $v'_n$,
but does not contribute to the class $\pi_{n}^*\left( \mathsf{Z}(n -1,i,j +1) \right)$ in the LHS of \eqref{eq:partapullback}. 
Using \eqref{eq:EIij}, it follows that the total contribution of $(\mathrm{T}', \bm{\psi}')$ to the LHS of \eqref{eq:partapullback} is
\begin{equation}
\label{eq:contLHSsubcase5.2}
-(-1)^{|\mathrm{H}^+(\mathrm{T}')|-1} \,\, |\bm{I}| \,\,d^i_{\mathrm{T}', \bm{\psi}'} \,\,\xi_{\mathrm{T}' *}\left( \bm{\psi}' \right).
\end{equation}
As in \eqref{eq:cidecomp}, we have
\[
d^i_{{\mathrm{T}'}, \bm{\psi}'} = \sum_{w^{\rm int}} \,\,w^{\rm int} \left( {\mathrm{T}}^{\rm int}, {\bm{\psi}}^{\rm int} \right) 
\, \sum_{a > a_0} \,\, d^a_{{{\mathrm{T}'}}^{\rm ext}, {\bm{\psi}'}^{\rm ext}}.
\]
Here as well, the number $a_0$ is determined by each $w^{\rm int}$ as in \eqref{eq:a0}.
Using this with \eqref{eq:EHTintText} and \eqref{eq:xiT'factored}, the  contribution  to the LHS of \eqref{eq:partapullback}
of $(\mathrm{T}', \bm{\psi}')$ for a fixed $w^{\rm int}$ becomes
\begin{multline}
\label{eq:contLHSsubcase5.2final}
 (-1)^{|\mathrm{H}^+(\mathrm{T}^{\rm int})|} \,w^{\rm int} \left( {\mathrm{T}}^{\rm int}, {\bm{\psi}}^{\rm int} \right)  \\
\times \sum_{a> a_0} \,\, \zeta_{\mathrm{T}*} \left( {\bm{\psi}}^{\rm int} \, (-1)^{|\mathrm{H}^+(\mathrm{T}^{\rm ext})|} 
\,\, |\bm{I}|\,\,d^a_{{{\mathrm{T}'}}^{\rm ext}, {\bm{\psi}'}^{\rm ext}}
\,\,\xi_{{\mathrm{T}'}^{\rm ext} *}\left( {\bm{\psi}'}^{\rm ext} \right) \right).
\end{multline}
This is the summand of \eqref{eq:contLHSsubcase5.2} corresponding to a fixed $w^{\rm int}$.
Here the term
\[
(-1)^{|\mathrm{H}^+(\mathrm{T}^{\rm ext})|} 
\,\,|\bm{I}|\,\, d^a_{{{\mathrm{T}'}}^{\rm ext}, {\bm{\psi}'}^{\rm ext}}
\,\,\xi_{{\mathrm{T}'}^{\rm ext} *}\left( {\bm{\psi}'}^{\rm ext} \right) 
\]
matches the contribution of $( {\mathrm{T}'}^{\rm ext}, {\bm{\psi}'}^{\rm ext}  )$ to $-|\bm{I}|\,\,\mathsf{E}_{\bm{I}}(a,b)$. 

Hence the contribution to the LHS of \eqref{eq:partapullback} of $(\mathrm{T}', \bm{\psi}' )$ for a fixed $w^{\rm int}$ matches the contribution to \eqref{eq:LHS_ext}  of $(\mathrm{T}', \bm{\psi}' )$ for a fixed $w^{\rm int}$, as desired.

\medskip

\noindent \textit{Subcase 6.} 
Finally, it remains to discuss the case when $(\mathrm{T}', \bm{\psi}', w')\in \mathsf{C}_{0,n+1}^i$ and the external vertex $v'_n$ has valence at least four. In this case, 
$(\mathrm{T}', \bm{\psi}')$ contributes both to $\pi_{n}^*\left( \mathsf{Z}(n -1,i,j +1) \right)$ and to the cycle $\mathsf{E}_{\bm{I}}(i,j)$ where $\bm{I}\sqcup \{n\}$ is the set of the  legs incident to $v'_n$. Then, the total contribution of $(\mathrm{T}', \bm{\psi}')$ to the LHS of \eqref{eq:partapullback} is
\[
(-1)^{|\mathrm{H}^+(\mathrm{T}')|} \,\, \left( c^i_{\overline{\mathrm{T}}', \bm{\psi}'}+ |\bm{I}|\,\,d^i_{\mathrm{T}', \bm{\psi}'}\right) \,\,\xi_{\mathrm{T}' *}\left( \bm{\psi}' \right).
\]
This is the sum of  terms similar to \eqref{eq:contLHSsubcase4} and \eqref{eq:contLHSsubcase5.2}.
Hence, as above, for a fixed $w^{\rm int}$ this reduces to the sum of \eqref{eq:contLHSsubcase4final} and \eqref{eq:contLHSsubcase5.2final}, and this matches the contribution to \eqref{eq:LHS_ext}, as desired.

This concludes the proof of Claim \ref{claim:part_a}. \hfill $\triangle$

\medskip

Finally, we show the following.

\begin{claim}
\label{claim:part_b}
Part $(b)$ for $n-1$ and part $(a)$ for $n$ imply part $(b)$ for $n$. 
\end{claim}

\noindent \textit{Proof of Claim \ref{claim:part_b}.}
As argued at the beginning of the proof of Theorem \ref{thm:mainthmgenuszeroij}, one can assume that $1\leq j < i<n$, since the vanishing of $\mathsf{Z}^{\rm t}(n,i,j)$ and $\mathsf{Z}(n,i,j)$  is otherwise trivial.

For $2\leq i \leq n-2$, assume part $(b)$ for $n-1$. Specifically, assume the vanishing
$\mathsf{Z}(n-1,i,j)=0$ in $A^{n-1+j-i}\left(\ov{\M}_{0,n}\right)$ for $j \geq 1$. 
Using Theorem \ref{thm:collidinggenus0}, this implies the vanishing of all cycles $\mathsf{Z}^m(n-m,i,j)$ obtained from $\mathsf{Z}(n-1,i,j)$ by colliding points.
Using \eqref{eq:EIijsarepfwds}, by pushing forward, this implies the vanishing of all cycles $\mathsf{E}_{\bm{I}}(i,j)$ in $A^{n-1+j-i}\left(\ov{\M}_{0,n+1}\right)$ for $j \geq 1$. 
Then part $(a)$ for $n$ implies 
\begin{equation}
\label{eq:Ztvanishing1}
\mathsf{Z}^{\rm t}(n,i,j) =0 \quad \mbox{in}\quad A^{n-1+j-i}\left(\ov{\M}_{0,n+1}\right) \quad \mbox{for $1\leq j <  i \leq n-2$.}
\end{equation}
 We note that in part $(a)$ with $j=0$, while $\mathsf{Z}(n-1,i,1)$  vanishes,  the cycles $\mathsf{E}_{\bm{I}}(i,0)$ are non-zero, and so is $\mathsf{Z}^{\rm t}(n,i,0)$.

Regarding the cycle $\mathsf{Z}(n,i,j)$, the vanishing
\[
\mathsf{Z}(n,i,j) =0 \quad \mbox{in}\quad A^{n-1+j-i}\left(\ov{\M}_{0,n+1}\right) \quad \mbox{for $1\leq j <  i \leq n-2$}
\]
follows from Lemma \ref{lem:DectZt}, \eqref{eq:Ztvanishing1}, and the vanishing of  $\mathsf{Z}(n-1,i^+, j^+)$ for $i^+>i$ and $j^+ - i^+ = j-i$. 
Indeed, since $j^+\geq 1$ for all $i^+>i$ and $j\geq 1$, the cycles $\mathsf{Z}(n-1,i^+, j^+)$ in Lemma \ref{lem:DectZt} vanish  by the induction hypothesis.

It remains to establish the case $i=n-1$, i.e., the vanishing of $\mathsf{Z}^{\rm t}(n,n-1,j)$ and $\mathsf{Z}(n,n-1,j)$ in $A^{j}\left(\ov{\M}_{0,n+1}\right)$ for $1\leq j \leq n-2$.
First we note that $\mathsf{Z}^{\rm t}(n,n-1,j)=\mathsf{Z}(n,n-1,j)$, since for any $i$-weighting $w$ with $i=n-1$, one has that $w(h_0,0)=n-1$ is the maximal value of $w$, hence the condition \eqref{eq:truncatedprop} is automatically satisfied.
The case $j=1$ is proven by Lemma \ref{lemma:imaxj1van}. For $j\geq 2$, assume the vanishing of $\mathsf{Z}(n,n-1,j-1)$. The case $i=n-2$ discussed above establishes the vanishing of $\mathsf{Z}(n,n-2,j-1)$. Then Lemma \ref{lemma:vanimaxrec} implies the vanishing of $\mathsf{Z}(n,n-1,j)$ in $A^{j}\left(\ov{\M}_{0,n+1}\right)$. Hence the case $i=n-1$ follows for all $j \geq 1$.

This concludes the proof of Claim \ref{claim:part_b}. \hfill $\triangle$

\medskip

Combining Claims \ref{claim:part_a} and \ref{claim:part_b}, Theorem \ref{thm:mainthmgenuszeroij} follows by induction. 
\end{proof}


\section{Colliding points on rational tail trees}

We introduce the cycle $\mathsf{F}^k_{g,\bm{m}}$ in Definition \ref{def:Fm} generalizing $\mathsf{F}^k_{g,n}$ from Definition \ref{def:graphfor}. As shown in Theorem \ref{thm:collidinggenusg}, $\mathsf{F}^k_{g,\bm{m}}$ is obtained from  $\mathsf{F}^k_{g,n}$ by colliding marked points. The cycle $\mathsf{F}^k_{g,\bm{m}}$ will feature in Theorem \ref{thm:graphFor-m}, a generalization of Theorem \ref{thm:graphFor}.

Let $\bm{m}=(m_1,\dots, m_n)$ be an $n$-tuple of  positive integers.
Let $g>0$ and $\Gamma\in \mathsf{G}^\mathsf{rt}_{g,n}$. The $n$-tuple $\bm{m}$ induces a weight on the legs of $\Gamma$, i.e., 
let the leg $\ell$  have weight $m_\ell$ for all $\ell=1,\dots, n$.

Define the set of decorations $\Psi(\Gamma)$ of the graph $\Gamma$ as
\[
\Psi(\Gamma):=\left\{\bm{\psi}= \prod_{\ell=1}^n \psi_\ell^{d_\ell} \prod_{h\in \mathrm{H}(\Gamma)} \psi_h^{d_h} \,\Bigg| \, 
\begin{array}{l}
d_h\geq 0 \mbox{ for all }h,\\[.2cm]
 d_\ell \geq 0 \mbox{ for all }\ell
\end{array}
\right\} \subset A^*\left(\ov{\M}_ \Gamma \right).
\]
Let $\bm{\psi}=\prod_{\ell=1}^n \psi_\ell^{d_\ell}\prod_{h\in \mathrm{H}(\Gamma)} \psi_h^{d_h}$ in $\Psi(\Gamma)$, and define accordingly 
\begin{eqnarray*}
\begin{split}
\mathrm{H}(\Gamma,\bm{\psi}) &:= && \{ (\ell, e) \, | \, \ell\in \mathrm{L}(\Gamma), \, 1\leq e \leq d_\ell  \} \\
&&&\sqcup \{ (h, e) \, | \, h\in \mathrm{H}^-(\Gamma), \, 1\leq e \leq d_h  \} \\
&&&\sqcup \{ (h, e) \, | \, h\in \mathrm{H}^+(\Gamma), \, 0\leq e \leq d_h  \} .
\end{split}
\end{eqnarray*}
One has $|\mathrm{H}(\Gamma,\bm{\psi})| = |\mathrm{E}(\Gamma)| + \deg \bm{\psi}$.

For a half-edge $h$ of $\Gamma$, let $\iota(h)$ be the half-edge such that $(h, \iota(h))\in \mathrm{E}(\Gamma)$.
For $h^+\in \mathrm{H}^+(\Gamma)$, define the weighted capacity $\ell^{\bm{m}}_{h^+}$ of $h^+$ as:
\begin{equation*}
\ell^{\bm{m}}_{h^+} := 
\left\{
\begin{array}{l}
\mbox{sum of the weights of the legs of $\Gamma$ in the maximal}\\
\mbox{connected subtree containing $h^+$, but not $\iota(h^+)$.}
\end{array}
\right.
\end{equation*}

The set of \textit{weightings of $(\Gamma, \bm{\psi})$ compatible with the legs weight} is
\begin{align*}
\mathsf{W}^{\bm{m}}_{\Gamma,\bm{\psi}} := \left\{ 
\begin{array}{l}
w \colon \mathrm{H}(\Gamma,\bm{\psi}) \rightarrow \mathbb{N}  \,: \, \\[.2cm]
\quad \ell^{\bm{m}}_{h^+} -1 \geq w(h^+,0) \geq w(h^+,1)  \geq \cdots \geq w(h^+,d_{h^+})  \\ [.2cm]
\qquad\qquad\mbox{for all heads $h^+$,}\\ [.2cm]
\quad 1 \leq w(h^-,1) < \cdots <  w(h^-,d_{h^-}) < w(\iota(h^-),0) \\ [.2cm]
\qquad\qquad\mbox{for all tails $h^-$,}\\ [.2cm]
\quad 1\leq w(\ell,1) < \cdots <  w(\ell,d_\ell) < m_\ell \\ [.2cm]
\qquad\qquad\mbox{for all legs $\ell$}
\end{array}
\right\}.
\end{align*}
One has
\begin{equation}
\label{eq:Wmnonempty}
\mbox{if} \quad \mathsf{W}^{\bm{m}}_{\Gamma,\bm{\psi}} \neq \varnothing, \quad \mbox{then} \quad
 d_\ell <m_\ell \quad \mbox{for all legs $\ell$}. 
\end{equation}
For each leg $\ell$, the inequality $d_\ell< m_\ell$ follows from the conditions on $w(\ell,e)$ for $e=1, \dots, d_\ell$.

For $w\in \mathsf{W}^{\bm{m}}_{\Gamma,\bm{\psi}}$, the corresponding \textit{weight} of $(\Gamma,\bm{\psi})$ is 
\begin{equation*}
w(\Gamma, \bm{\psi}) := \prod_{(h,e)\in \mathrm{H}(\Gamma, \bm{\psi})} w(h,e).
\end{equation*}
Define 
\begin{equation*}
{c}^{\bm{m}}_{\Gamma, \bm{\psi}} := \sum_{w\in \mathsf{W}^{\bm{m}}_{\Gamma, \bm{\psi}}} w(\Gamma, \bm{\psi}).
\end{equation*}

For $\bm{\psi}=\prod_{\ell=1}^n \psi_\ell^{d_\ell}\prod_{h\in \mathrm{H}(\Gamma)} \psi_h^{d_h}$ in $\Psi(\Gamma)$, define
\[
\bm{\beta}_{\Gamma, \bm{\psi}}:=
\prod_{\ell=1}^n (k \,\omega_\ell-\eta)^{d_\ell}
\prod_{h\in \mathrm{H}^+(\Gamma)} (k \,\omega_{h}-\eta)^{1+d_{h}}
\prod_{h\in \mathrm{H}^-(\Gamma)} (k \,\omega_{h}-\eta)^{d_{h}} 
\]
in $R^*\left(\mathbb{P}\mathbb{E}^k_\Gamma \right)$. One has $\deg \bm{\beta}_{\Gamma, \bm{\psi}} = |\mathrm{E}(\Gamma)|+\deg \bm{\psi}$.

\begin{definition}
\label{def:Fm}
For $\bm{m}=(m_1,\dots,m_n)$ with $n^+:=\sum_{\ell=1}^n m_\ell$,
define $\mathsf{F}^k_{g,\bm{m}}$ in $R^{n^+}\left(\mathbb{P} \mathbb{E}^k_{g,n} \right)$ as
\[
\mathsf{F}^k_{g,\bm{m}}:=
\sum_{\Gamma\in \mathsf{G}^\mathsf{rt}_{g,n}} \, (-1)^{|\mathrm{E}(\Gamma)|}  \,  {\xi_{\Gamma}}_* \left[  \prod_{\ell=1}^n (k \,\omega_\ell-\eta)^{m_\ell}
\sum_{\bm{\psi}\in \Psi(\Gamma)} {c}^{\bm{m}}_{\Gamma, \bm{\psi}}\, \bm{\psi} \, \bm{\beta}_{\Gamma, \bm{\psi}}^{-1}
\right].
\]
\end{definition}

For $\bm{m}=(1,\dots,1)=1^n$, the cycle $\mathsf{F}^k_{g,\bm{m}}$ specializes to the cycle $\mathsf{F}^k_{g,n}$ from Definition \ref{def:graphfor}.

\begin{theorem}
\label{thm:collidinggenusg}
The cycle $\mathsf{F}^k_{g,\bm{m}}$ in $R^{n^+}\left(\mathbb{P} \mathbb{E}^k_{g,n} \right)$ is obtained from $\mathsf{F}^k_{g,n^+}$ in $R^{n^+}\left(\mathbb{P} \mathbb{E}^k_{g,n^+} \right)$ by colliding the first $m_1$ marked points into the marked point $P_1$, the next $m_2$ marked points into the marked point $P_2$, etc.
\end{theorem}

\begin{proof}
The argument is similar to the one for Theorem \ref{thm:collidinggenus0}. Here, trees have a positive genus vertex, and more than one leg can be decorated with $\psi$-classes. However, these differences do not basically alter the argument used for Theorem \ref{thm:collidinggenus0}.

Namely, it is enough to show that the cycle $\mathsf{F}^k_{g,\bm{m}}$ in $R^{n^+}\left(\mathbb{P} \mathbb{E}^k_{g,n} \right)$ for 
\[
\bm{m}=(m_1, \dots, m_{i-1}, m_i, 1^{n-i}) \qquad \mbox{with $i\leq n$}
\]
is obtained from $\mathsf{F}^k_{g,\bm{m}^-}$ in $R^{n^+}\left(\mathbb{P} \mathbb{E}^k_{g,n+1} \right)$ with 
\[
\bm{m}^-:=(m_1, \dots, m_{i-1}, m_i-1, 1^{n-i+1})
\]
by colliding the points $P_{i}$ and $P_{i+1}$. Here $1^a=(1,\dots, 1)$ is an $a$-tuple of elements $1$ for $a\geq 0$.
In other words, we show that
\begin{equation}
\label{eq:diffcollg}
\pi_{i+1 *}\left(  \delta_{0:\{i,i+1\}} \cdot \mathsf{F}^k_{g,\bm{m}^-} \right) = \mathsf{F}^k_{g,\bm{m}}  \in R^{n^+}\left(\mathbb{P} \mathbb{E}^k_{g,n} \right)
\end{equation}
where $\pi_{i+1}\colon \mathbb{P} \mathbb{E}^k_{g,n+1} \rightarrow \mathbb{P} \mathbb{E}^k_{g,n}$ is the map obtained by forgetting the point $P_{i+1}$, and $\delta_{0:\{i,i+1\}} $ is the class of the divisor in $\mathbb{P} \mathbb{E}^k_{g,n+1}$ which is the preimage of the divisor in $\ov{\M}_{g,n+1}$ whose general element is a nodal curve with two components such that one  component is rational and contains only two marked points, namely $P_{i}$ and $P_{i+1}$.

The non-zero terms arising in the expansion of the LHS of \eqref{eq:diffcollg} are obtained by colliding $P_{i}$ and $P_{i+1}$ on the terms of $\mathsf{F}^k_{g,\bm{m}^-}$ contributed by decorated graphs where the legs $i$ and $i+1$ are incident to the same vertex, say $v_i$. 
The argument  then continues as in the proof of Theorem \ref{thm:collidinggenus0}. 
Namely, the contributions to the LHS of \eqref{eq:diffcollg} obtained by decorated graphs where the valence of $v_i$ is three match  the contributions to $\mathsf{F}^k_{g,\bm{m}} $ where the degree of $\psi_i$ is positive, as in Case 1 there. Similarly, the contributions to the LHS of \eqref{eq:diffcollg} obtained by decorated graphs where the valence of $v_i$ is at least four match  the contributions to $\mathsf{F}^k_{g,\bm{m}} $ where the degree of $\psi_i$ is zero, as in Case 2 there. The statement follows.
\end{proof}


\section{Proof of Theorem \ref{thm:graphFor}}
\label{sec:proofgraphFor}

Here we complete the proof of Theorem \ref{thm:graphFor}. 
We prove more generally the following statement. 
Given an $n$-tuple $\bm{m}=(m_1,\dots,m_n)$ of positive integers, define 
\begin{equation}
\label{eq:Hkgm}
{\H}^k_{g,\bm{m}} \subset \mathbb{PE}^k_{g,n}
\end{equation}
as the locus consisting of  smooth $n$-pointed genus $g$ curves $(C,P_1, \dots, P_n)$ together with the class of a stable $k$-differential having a zero of order at least $m_\ell$ at the marked point $P_\ell$ for each $\ell=1,\dots, n$.

Let $n^+:=|\bm{m}|=\sum_{\ell=1}^n m_\ell$. Unless $k\geq 2$, $n^+ = k(2g-2)$ and $k \mid m_\ell$ for all $\ell$, the locus ${\H}^k_{g,\bm{m}}$ has pure codimension $n^+$ in $\mathbb{PE}^k_{g,n}$.

For $\bm{m}=(1,\dots,1)=1^n$, the locus ${\H}^k_{g,\bm{m}}$ specializes to ${\H}^k_{g,n}$ from \eqref{eq:Hkgndef}.

\begin{theorem}
\label{thm:graphFor-m}
Let $g\geq 2$ and $\bm{m}=(m_1,\dots,m_n)$ with $|\bm{m}|=n^+$. 
One has
\[
\left[\ov{\H}^k_{g,\bm{m}} \right] = \mathsf{F}^k_{g,\bm{m}}
\]
in:
\begin{enumerate}[(i)]
\item $A^{n^+}\left( \P\E^k_{g,n}\big|_{\M_{g,n}^\mathsf{rt}}  \right)$ when
\begin{equation}
\label{eq:hypknm}
\left\{
\begin{array}{l}
\mbox{$k=1$; or}\\[.1cm]
\mbox{$k\geq 2$ and $n^+ \leq k(2g-2)-1$;}\\[.1cm]
\end{array}
\right.
\end{equation}

\item $A^{n^+}\left( \P\E^k_{g,n} \right)$ when $n^+\leq k$. 
\end{enumerate}
\end{theorem}

One recovers Theorem \ref{thm:graphFor} when $\bm{m}=(1, \dots,1)=1^n$.
When $k=1$ and $n^+=2g-2$, the statement gives all strata classes  in the projectivized Hodge bundle over curves with rational tails.

Sometimes, the locus \eqref{eq:Hkgm} is  the union of multiple  components. In this case, \eqref{eq:hypknm} implies that all components have the same dimension, and
 the class given in Theorem \ref{thm:graphFor-m} is a weighted sum of the classes of the components. For instance, see the example in \eqref{eq:H222}--\eqref{eq:F222}.

\smallskip

The key steps of the proof are provided by Theorems \ref{thm:pinZnrhonZ1plus}, \ref{thm:Decrec}, \ref{thm:collidinggenusg}, and the next Theorem \ref{thm:Frec}. This next statement is a recursive identity about the cycle $\mathsf{F}^k_{g,n}$ which is the counterpart of Theorem \ref{thm:pinZnrhonZ1plus}.
For this, we first define the cycles $\mathsf{E}_{\bm{I}}$.
For a non-empty $\bm{I}\subseteq \{1,\dots, n-1\}$ of size $|\bm{I}|=m$, recall the gluing map of degree one
$\gamma\colon \mathbb{PE}^k_{g,n-m} \times \ov{\M}_{0,\bm{I}\sqcup\{n,h_n\}} \rightarrow \mathbb{PE}^k_{g,n}$
from \eqref{eq:gamma0}.
Let
\begin{equation}
\label{eq:EIclass}
\mathsf{E}_{\bm{I}} := \gamma_* \,\mathsf{F}^k_{g,\bm{m}} \qquad \mbox{where}\qquad \bm{m}=\left(m, 1^{n-m-1}\right).
\end{equation}

Recall the map $\pi_n\colon \mathbb{PE}^k_{g,n} \rightarrow  \mathbb{PE}^k_{g,n-1}$ obtained by forgetting the point $P_n$, and the map $\rho_n\colon  \mathbb{PE}^k_{g,n} \rightarrow \mathbb{PE}^k_{g,1}$ obtained by forgetting all but the point $P_n$.

\begin{theorem}
\label{thm:Frec}
For $n\geq 2$, one has
\begin{equation}
\label{eq:recFkgn}
\pi_{n}^* \left( \mathsf{F}^k_{g,n-1} \right)\cdot \rho_{n}^* \left( \mathsf{F}^k_{g,1}\right) - \sum_{\bm{I}}|\bm{I}|\, \mathsf{E}_{\bm{I}} = \mathsf{F}^k_{g,n}
\quad \in R^n\left(\mathbb{P}\mathbb{E}^k_{g,n} \right)
\end{equation}
where the sum is over all non-empty $\bm{I}\subseteq \{1,\dots, n-1\}$.
\end{theorem}

\begin{proof}
We  verify that the contributions of each graph $\Gamma$ in $\mathsf{G}^{\rm rt}_{g,n}$ to the two sides of \eqref{eq:recFkgn} match, modulo tautological relations. For this, we distinguish two cases, depending on the position of the leg $n$ in each graph.

In the second case, it will be convenient to expand $\mathsf{F}^k_{g,n} \in R^n\left(\mathbb{P}\mathbb{E}^k_{g,n} \right)$   as
\begin{equation}
\label{eq:Fkgnw}
\mathsf{F}^k_{g,n} = \sum_{(\Gamma, \bm{\psi}, w)} \, (-1)^{|\mathrm{E}(\Gamma)|}  \, w(\Gamma, \bm{\psi})\, {\xi_{\Gamma}}_* \left[ 
 \bm{\psi} \, \bm{\beta}_{\Gamma, \bm{\psi}}^{-1} \,
\prod_{\ell=1}^n  (k \,\omega_\ell-\eta) 
\right]
\end{equation}
where the sum is over triples $(\Gamma, \bm{\psi}, w)$ such that $\Gamma\in \mathsf{G}^{\rm rt}_{g,n}$, $\bm{\psi}\in\Psi(\mathrm{T})$, and $w\in \mathsf{W}_{\Gamma, \bm{\psi}}$.
This formula follows by expanding the coefficients $c_{\Gamma, \bm{\psi}}$ in Definition \ref{def:graphfor}.

\smallskip

\noindent \textit{Case 1.} 
Consider $\Gamma\in\mathsf{G}^{\rm rt}_{g,n}$ where the leg $n$ is incident to the genus $g$ vertex.
On the LHS of \eqref{eq:recFkgn}, the graph $\Gamma$ contributes only to $\pi_{n}^* \left( \mathsf{F}^k_{g,n-1} \right)\cdot \rho_{n}^* \left( \mathsf{F}^k_{g,1}\right)$, and using Definition \ref{def:graphfor}, the contribution is
\[
(-1)^{|\mathrm{E}(\Gamma)|}\,{\xi_{\Gamma}}_* \!\left[\prod_{\ell=1}^{n-1} (k \,\omega_\ell-\eta) \!\! 
\sum_{\bm{\psi}\in \Psi(\Gamma)} c_{\Gamma, \bm{\psi}}\, \bm{\psi}\, \bm{\beta}_{\Gamma,\bm{\psi}}^{-1}
\right]
\cdot(k \,\omega_n-\eta).
\]
Rearranging factors, this is
\[
(-1)^{|\mathrm{E}(\Gamma)|}\,{\xi_{\Gamma}}_* \!\left[\prod_{\ell=1}^{n} (k \,\omega_\ell-\eta) \!\! 
\sum_{\bm{\psi}\in \Psi(\Gamma)} c_{\Gamma, \bm{\psi}}\, \bm{\psi}\, \bm{\beta}_{\Gamma,\bm{\psi}}^{-1}
\right]
\]
thus matching the contribution of $\Gamma$ to $\mathsf{F}^k_{g,n}$ in the RHS of \eqref{eq:recFkgn}.

\smallskip

\noindent \textit{Case 2.} Consider $\Gamma\in\mathsf{G}^{\rm rt}_{g,n}$ where the leg $n$ is incident to a rational vertex.  
We argue by applying Theorem \ref{thm:Decrec} locally to a rational subtree of $\Gamma$ containing the leg $n$.

Namely, let $\mathrm{T}$ be the maximal rational subtree of $\Gamma$ which contains the leg $n$, and let $\Gamma_\circ:=\Gamma \setminus \mathrm{T}$ be its complement. 
If $e= (h_0, t_0)\in \mathrm{E}(\Gamma)$ is the edge connecting the genus $g$ vertex in $\Gamma_\circ$ and the tree $\mathrm{T}$, then the head $h_0$ is in $\mathrm{T}$ and  the tail $t_0=\iota(h_0)$ is in $\Gamma_\circ$.

Let $\mathsf{G}_{\Gamma_\circ}$ be the subset of $\mathsf{G}_{g,n}^{\rm rt}$ consisting of graphs obtained by fixing $\Gamma_\circ$ and varying $\mathrm{T}$. This is:
\[
\mathsf{G}_{\Gamma_\circ} := \left\{ \widehat{\Gamma}:=\Gamma_\circ \sqcup_{e} \mathrm{T}'  \in \mathsf{G}_{g,n}^{\rm rt} \,\, |\,\,  \mathrm{T}'  \in \mathsf{G}_{0,n' +1} \right\}
\]
where $n'$ is the number of legs of $\Gamma$ within $\mathrm{T}$.
For each $\widehat{\Gamma}=\Gamma_\circ \sqcup_{e} \mathrm{T}' $ in $\mathsf{G}_{\Gamma_\circ}$, the gluing map $\xi_{\widehat{\Gamma}}$ from \eqref{eq:xiGamma} factors as $\xi_{\widehat{\Gamma}} = \zeta_{\Gamma} \circ \xi_{\mathrm{T}'}$ where $\xi_{\mathrm{T}'}$ is as in \eqref{eq:xiT} and
\[
\zeta_{\Gamma} \colon \ov{\M}_{0, n' + 1} \times \mathbb{PE}^k_{g,n(g)}  \times \prod_{v \in \mathrm{V}_0(\Gamma_\circ)} \ov{\M}_{0, n(v)} \longrightarrow  \mathbb{PE}^k_{g,n}
\]
is the gluing map of degree one defined by $\Gamma_\circ$ and the edge $e$. The map $\zeta_{\Gamma}$ extends the $k$-differentials of elements of $\mathbb{PE}^k_{g,n(g)}$ by zero on the attached rational tails. Recall that $n(g)$ is the valence of the genus $g$ vertex, and $\mathrm{V}_0(\Gamma_\circ)$ is the set of rational vertices of $\Gamma_\circ$.

Consider a triple $\left(\widehat{\Gamma}, \widehat{\bm{\psi}}, \widehat{w}\right)$ with $\widehat{\Gamma}\in \mathsf{G}_{\Gamma_\circ}$, $\widehat{\bm{\psi}}\in \Psi(\widehat{\Gamma})$, and $\widehat{w}\in \mathsf{W}_{\widehat{\Gamma}, \widehat{\bm{\psi}}}$.
The contribution of $\left(\widehat{\Gamma}, \widehat{\bm{\psi}}, \widehat{w}\right)$ to $\mathsf{F}^k_{g,n}$ as given in \eqref{eq:Fkgnw} is
\begin{equation}
\label{eq:contGammaetchat}
(-1)^{|\mathrm{E}(\widehat{\Gamma})|}  \,\, \widehat{w}\left(\widehat{\Gamma}, \widehat{\bm{\psi}}\right) \, {\xi_{\widehat{\Gamma}}}_* \left[ 
 \widehat{\bm{\psi}}  \,\, \bm{\beta}_{\widehat{\Gamma}, \widehat{\bm{\psi}}}^{-1} \,\,
\prod_{\ell=1}^n  (k \,\omega_\ell-\eta) 
\right].
\end{equation}
Given the decomposition $\widehat{\Gamma}=\Gamma_\circ \sqcup \mathrm{T}'$,
let $\bm{\psi}_\circ$ and  ${\bm{\psi}}'$ be the restrictions of $\widehat{\bm{\psi}}$ to $\Gamma_\circ$ and $\mathrm{T}'$ respectively. Similarly, let $w_\circ$ and $w'$ be the restrictions of $\widehat{w}$ to $\left(\Gamma_\circ, \bm{\psi}_\circ\right)$ and  $\left(\mathrm{T}', \bm{\psi}'\right)$, respectively. If $d_0$ is the degree of $\psi_{t_0}$ in $\bm{\psi}_\circ$, let
\begin{equation}
\label{eq:i0}
i_0 := w_\circ\left(t_0, d_0 \right).
\end{equation}
By definition of the weighting $\widehat{w}$, one has that $w'(h_0,0)=i$ for some $i>i_0$, i.e., $w'\in \mathsf{W}^i_{\mathrm{T}', \bm{\psi}'}$ for some $i>i_0$.

Decomposing as  
\begin{align*}
|\mathrm{E}(\widehat{\Gamma})| = |\mathrm{E}(\Gamma_\circ)| + |\mathrm{H}^+(\mathrm{T}')|, \qquad
 \widehat{w}\left(\widehat{\Gamma}, \widehat{\bm{\psi}}\right) =  w_\circ\left(\Gamma_\circ, \bm{\psi}_\circ\right) \,  w'\left(\mathrm{T}', \bm{\psi}'\right), \\
\xi_{\widehat{\Gamma}}= \zeta_{\Gamma} \circ \xi_{\mathrm{T}'}, \qquad
\widehat{\bm{\psi}} =  \bm{\psi}_\circ \,{\bm{\psi}}', \qquad
\bm{\beta}_{\widehat{\Gamma}, \widehat{\bm{\psi}}} = \bm{\beta}_{\Gamma_\circ, \bm{\psi}_\circ} \,\, \bm{\beta}_{\mathrm{T}', \bm{\psi}'}, 
\end{align*}
and using the identity
\[
{\xi_{\widehat{\Gamma}}}_* \left[ \bm{\alpha} \,\bm{\beta}_{\mathrm{T}', \bm{\psi}'}^{-1}\,\,\prod_{\ell=1}^n (k \,\omega_\ell-\eta) \right]= {\xi_{\widehat{\Gamma}}}_* \left[ \bm{\alpha}\,(k \,\omega_{t_0}-\eta)^{-|\mathrm{H}(\mathrm{T}', \bm{\psi}')|}\,\,\prod_{\ell=1}^n (k \,\omega_\ell-\eta) \right]  
\]
in $A^*\left(\mathbb{PE}^k_{g,n(g)}\right)$ for an arbitrary cycle $\bm{\alpha}$, the  contribution in \eqref{eq:contGammaetchat} is
\begin{multline*}
(-1)^{|\mathrm{E}(\Gamma_\circ)|} \, w_\circ\left(\Gamma_\circ, \bm{\psi}_\circ\right) \, {\zeta_{\Gamma}}_* \!\Bigg[
 \, \bm{\psi}_\circ  \,\, \bm{\beta}_{\Gamma_\circ, \bm{\psi}_\circ}^{-1} \,\, \prod_{\ell=1}^n (k \,\omega_\ell-\eta) \\
\times (-1)^{|\mathrm{H}^+(\mathrm{T}')|} \, w'\left(\mathrm{T}', \bm{\psi}'\right) \,{\xi_{\mathrm{T}'}}_* \!\left(  \bm{\psi}'  \right)  \, (k \,\omega_{t_0}-\eta)^{-|\mathrm{H}(\mathrm{T}', \bm{\psi}')|}
\Bigg].
\end{multline*}

Summing over all $(\bm{\psi}_\circ, w_\circ)$ and  $(\mathrm{T}', \bm{\psi}', w')$, one has
\begin{multline}
\label{eq:totalGovGamma}
\sum_{(\bm{\psi}_\circ, w_\circ)} (-1)^{|\mathrm{E}(\Gamma_\circ)|} \, w_\circ\left(\Gamma_\circ, \bm{\psi}_\circ\right) \, {\zeta_{\Gamma}}_* \!\Bigg[
 \, \bm{\psi}_\circ  \,\, \bm{\beta}_{\Gamma_\circ, \bm{\psi}_\circ}^{-1} \,\, \prod_{\ell=1}^n (k \,\omega_\ell-\eta) \\
\times  \sum _{i> i_0}   \mathsf{Dec}_{n'}^i(k \,\omega_{t_0}-\eta)  
\Bigg]
\end{multline}
where for each $i$ 
\[
\mathsf{Dec}_{n'}^i(D) = \sum_{(\mathrm{T}' , \bm{\psi}', w')} (-1)^{|\mathrm{H}^+(\mathrm{T}')|} \, w'\left(\mathrm{T}', \bm{\psi}'\right) \,{\xi_{\mathrm{T}'}}_* \!\left(  \bm{\psi}'  \right)   \, D^{-|\mathrm{H}(\mathrm{T}', \bm{\psi}')|} 
\]
in $A^*\left( \ov{\M}_{0,n'+1}\right)\left[D^{-1}\right]$, as in \eqref{eq:Decimdef} with $\mathsf{Dec}_{n'}^i(D) = \mathsf{Dec}_{n'}^{i,1}(D)$. The index $i_0$ appearing in \eqref{eq:totalGovGamma} is determined by $w_\circ$, as in~\eqref{eq:i0}.

From Theorem \ref{thm:Decrec}, we have
\begin{equation}
\label{eq:applicationofrec}
\pi_{n}^*\left( \mathsf{Dec}_{n'-1}^i(D) \right) 
- \sum_{\bm{I} } |\bm{I}| \, \mathsf{E}_{\bm{I}}^i(D)
-\sum_{i^+>i} i\,\sigma_{0*} \left(   \mathsf{Dec}_{n'-1}^{i^+}(D) \right)
= \mathsf{Dec}_{n'}^i(D).
\end{equation}
 
We next verify that the sum of the contributions of all $\widehat{\Gamma} \in \mathsf{G}_{\Gamma_\circ}$ to $\mathsf{F}^k_{g,n}$, as given by the expression obtained by
replacing \eqref{eq:applicationofrec} for  $i>i_0$ in \eqref{eq:totalGovGamma},  matches the sum of the contributions of such graphs to the LHS of~\eqref{eq:recFkgn}.

Indeed, the sum of the contributions of all $\widehat{\Gamma} \in \mathsf{G}_{\Gamma_\circ}$ to $\pi_{n}^* \left( \mathsf{F}^k_{g,n-1} \right)\cdot \rho_{n}^* \left( \mathsf{F}^k_{g,1}\right)$ is computed as follows.
Consider the term of $\pi_{n}^* \left( \mathsf{F}^k_{g,n-1} \right)\cdot \rho_{n}^* \left( \mathsf{F}^k_{g,1}\right)$ given by
\begin{equation}
\label{eq:pull-backtermbeforeexpanding}
\sum_{(\ov{\Gamma}, \ov{\bm{\psi}}, \ov{w})} (-1)^{|\mathrm{E}(\ov{\Gamma})|}  \, \ov{w}(\ov{\Gamma}, \ov{\bm{\psi}})\, 
\pi_n^* \left({\xi_{\ov{\Gamma}}}_* \left[ 
\ov{\bm{\psi}} \, \bm{\beta}_{\ov{\Gamma}, \ov{\bm{\psi}}}^{-1} \,
\prod_{\ell=1}^{n-1}  (k \,\omega_\ell-\eta) 
\right] \right) \cdot (k \,\omega_n-\eta) 
\end{equation}
where the sum is over: the graphs $\ov{\Gamma}$ obtained from the graphs $\widehat{\Gamma}\in \mathsf{G}_{\Gamma_\circ}$ by removing the leg $n$ (and contracting an edge if the vertex $v_n$ incident to the leg $n$ is trivalent); the decorations $\ov{\bm{\psi}}$ of $\ov{\Gamma}$; and the weightings $\ov{w}$ of $(\ov{\Gamma}, \ov{\bm{\psi}})$. 

For $\widehat{\Gamma}=\Gamma_\circ \sqcup_{e} \mathrm{T}' $ in $\mathsf{G}_{\Gamma_\circ}$, the graph obtained by removing the leg $n$ can be decomposed as $\ov{\Gamma}= \Gamma_\circ \sqcup \ov{\mathrm{T}'}$, where $\ov{\mathrm{T}'}$ is the tree obtained from $\mathrm{T}'$ by removing the leg $n$ (and contracting an edge when necessary, as before).

As in \eqref{eq:pullbackTbar} and \eqref{eq:pullbackpsi}, the expansion of \eqref{eq:pull-backtermbeforeexpanding} consists of a sum of the decorated graphs obtained by first adding the leg $n$ to all vertices of each $\ov{\Gamma}$, and then expanding the decorations $\pi^*(\ov{\bm{\psi}})$ accordingly.
The sum of the terms obtained by adding the leg $n$ to the subtree $\ov{\mathrm{T}'}$ of each $\ov{\Gamma}$, and expanding the decorations $\pi^*(\ov{\bm{\psi}})$ accordingly, is 
\begin{multline}
\label{eq:case2cont1}
\sum_{(\bm{\psi}_\circ, w_\circ)} (-1)^{|\mathrm{E}(\Gamma_\circ)|} \, w_\circ\left(\Gamma_\circ, \bm{\psi}_\circ\right) \, {\zeta_{\Gamma}}_* \!\Bigg[
 \, \bm{\psi}_\circ  \,\, \bm{\beta}_{\Gamma_\circ, \bm{\psi}_\circ}^{-1} \,\, \prod_{\ell=1}^{n} (k \,\omega_\ell-\eta) \\
\times  \sum _{i> i_0}   \pi_{n}^*\left( \mathsf{Dec}_{n'-1}^i(k \,\omega_{t_0}-\eta) \right) 
\Bigg] .
\end{multline}

When adding the leg $n$ to the genus $g$ vertex of a graph $\ov{\Gamma}$, the expansion of each decoration $\pi^*(\ov{\bm{\psi}})$ consists of $\ov{\bm{\psi}}$ decorating the graph with the leg $n$  attached to genus $g$ vertex (this case has been treated in Case 1), \textit{minus} additional  terms contributed by the pull-back of $\psi$-classes decorating the tails incident to the genus $g$ vertex. 
In particular, when the degree of $\psi_{t_0}$ in $\ov{\bm{\psi}}$ is positive, the expansion of $\pi^*(\ov{\bm{\psi}})$ includes the term corresponding to a decorated graph $(\widehat{\Gamma}, \widehat{\bm{\psi}})$ where $\widehat{\Gamma}=\Gamma_\circ \sqcup v_n\sqcup \ov{\mathrm{T}'}$ is obtained from $\ov{\Gamma}=\Gamma_\circ \sqcup_e \ov{\mathrm{T}'}$ with $e=(\ov{h}_0, t_0)$ by  inserting a trivalent vertex $v_n$ incident to the leg $n$, attached to the genus $g$ vertex in $\Gamma_\circ$ via an edge with tail equal to $t_0$, and attached to $\ov{\mathrm{T}'}$ via an edge with head equal to $\ov{h}_0$. The decoration $\widehat{\bm{\psi}}$ is obtained from $\ov{\bm{\psi}}$ by decreasing the degree of $\psi_{t_0}$ by one, i.e., $\widehat{\bm{\psi}}\cdot \psi_{t_0}= \ov{\bm{\psi}}$. 

A weighting $\ov{w}$ of such $(\ov{\Gamma}, \ov{\bm{\psi}})$ can be decomposed as follows.
Let $\bm{\psi}_\circ$ and $\ov{\bm{\psi}'}$ be the restrictions of $\widehat{\bm{\psi}}$ to $\Gamma_\circ$ and $\ov{\mathrm{T}'}$, respectively, so that $\widehat{\bm{\psi}}=\bm{\psi}_\circ \,\ov{\bm{\psi}'}$, and thus $\ov{\bm{\psi}}=\bm{\psi}_\circ \, \psi_{t_0}\,\ov{\bm{\psi}'}$. 
Let $d_0$ be the degree of $\psi_{t_0}$ in $\bm{\psi}_\circ$, so that $d_0+1$ is the degree of $\psi_{t_0}$ in $\ov{\bm{\psi}}$.
Similarly, let $w_\circ$ and $\ov{w'}$ be the restrictions of $\ov{w}$ to $(\Gamma_\circ, \bm{\psi}_\circ)$ and $(\ov{\mathrm{T}'}, \ov{\bm{\psi}'})$, respectively.
One has the decomposition
\[
\ov{w}(\ov{\Gamma}, \ov{\bm{\psi}}) = w_\circ\left(\Gamma_\circ, \bm{\psi}_\circ\right) \cdot i\cdot \ov{w'}(\ov{\mathrm{T}'}, \ov{\bm{\psi}'}) \qquad \mbox{for some $i> i_0$}
\]
where $i_0=\ov{w}(t_0, d_0)$ as in \eqref{eq:i0} and $i= \ov{w}(t_0, d_0+1)$. Moreover, one concludes that 
\[
\ov{w'}\in \mathsf{W}^{i^+}_{\ov{\mathrm{T}'}, \ov{\bm{\psi}'}} \qquad \mbox{for some $i^+>i$.} 
\]
The inequalities $i^+ > i > i_0$ follow from the definition of the weighting $\ov{w}$.

The gluing map $\xi_{\widehat{\Gamma}}$ for such $\widehat{\Gamma}=\Gamma_\circ \sqcup v_n\sqcup \ov{\mathrm{T}'}$ factors as 
\[
\xi_{\widehat{\Gamma}} = \zeta_{\Gamma} \circ \sigma_0 \circ \xi_{\ov{\mathrm{T}'}},
\]
where $\sigma_0\colon \ov{\M}_{0,n'}\rightarrow \ov{\M}_{0,n'+1}$ is the map obtained by attaching at $\ov{h}_0$ a rational component containing the leg $n$ and $h_0$.

Summing over all $(\bm{\psi}_\circ, w_\circ)$ and  $(\ov{\mathrm{T}'}, \ov{\bm{\psi}'}, \ov{w'})$, one obtains
\begin{multline}
\label{eq:case2cont2}
\sum_{(\bm{\psi}_\circ, w_\circ)} (-1)^{|\mathrm{E}(\Gamma_\circ)|} \, w_\circ\left(\Gamma_\circ, \bm{\psi}_\circ\right) \, {\zeta_{\Gamma}}_* \!\Bigg[
 \, \bm{\psi}_\circ  \,\, \bm{\beta}_{\Gamma_\circ, \bm{\psi}_\circ}^{-1} \,\, \prod_{\ell=1}^{n} (k \,\omega_\ell-\eta) \\
\times  \Bigg(-\sum _{i^+> i > i_0}  i\,\sigma_{0*} \left(   \mathsf{Dec}_{n'-1}^{i^+}(D) \right)\Bigg)
\Bigg] .
\end{multline}

Finally, after \eqref{eq:EIclass} and Definition \ref{def:Fm}, and applying Lemma \ref{lemma:EIisgammapfwd} when $|\bm{I}|<n'-1$, the sum of the contributions of all $\Gamma' \in \mathsf{G}_{\Gamma_\circ}$ to $\mathsf{E}_{\bm{I}}$ is
\begin{equation}
\label{eq:case2cont3}
\sum_{(\bm{\psi}_\circ, w_\circ)} (-1)^{|\mathrm{E}(\Gamma_\circ)|} \, w_\circ\left(\Gamma_\circ, \bm{\psi}_\circ\right) \, {\zeta_{\Gamma}}_* \!\Bigg[
 \, \bm{\psi}_\circ  \,\, \bm{\beta}_{\Gamma_\circ, \bm{\psi}_\circ}^{-1} \,\, \prod_{\ell=1}^{n} (k \,\omega_\ell-\eta) 
  \sum _{i> i_0}    \mathsf{E}_{\bm{I}}^i (k \,\omega_{t_0}-\eta)
\Bigg]  .
\end{equation}

The sum of \eqref{eq:case2cont1} and \eqref{eq:case2cont2} gives the contribution of all $\widehat{\Gamma} \in \mathsf{G}_{\Gamma_\circ}$  to $\pi_{n}^* \left( \mathsf{F}^k_{g,n-1} \right)\cdot \rho_{n}^* \left( \mathsf{F}^k_{g,1}\right)$. Additionally, using  \eqref{eq:case2cont3} for each $\bm{I}$, one obtains that
the sum of the contributions of all $\widehat{\Gamma} \in \mathsf{G}_{\Gamma_\circ}$  to the LHS of~\eqref{eq:recFkgn} matches
the sum of the contributions of such graphs to $\mathsf{F}^k_{g,n}$, as given by the expression obtained by replacing \eqref{eq:applicationofrec} for  $i>i_0$ in \eqref{eq:totalGovGamma}.
This ends the proof.
\end{proof}

We now complete the proof of Theorem \ref{thm:graphFor-m}:

\begin{proof}[Proof of Theorem \ref{thm:graphFor-m}]
First we argue that the case  $\bm{m}=(1,\dots,1)$  implies the case of arbitrary  $\bm{m}$.
Precisely, assume the statement holds for a given $n$ and $\bm{m}=(1,\dots,1)$, so that if \eqref{eq:hypknm} holds, then one has
\begin{equation}
\label{eq:rechypm11}
\left[\ov{\H}^k_{g,n} \right] = \mathsf{F}^k_{g,n} \quad \in A^{n}\left( \P\E^k_{g,n}\big|_{\M_{g,n}^\mathsf{rt}}  \right)
\end{equation}
and if $n\leq k$, then \eqref{eq:rechypm11} extends over $\ov{\M}_{g,n}$.
Given a sequence $\bm{m}=(m_1,\dots,m_{n^-})$ with $|\bm{m}|=n$ of length $n^-< n$,
 the locus $\ov{\H}^k_{g,\bm{m}}$ in $\P\E^k_{g,n^-}$ is obtained from the locus $\ov{\H}^k_{g,n}$ by colliding the first $m_1$ marked points into the marked point $P_1$, the next $m_2$ marked points into the marked point $P_2$, etc. From Theorem \ref{thm:collidinggenusg}, our assumption implies the statement for~$\bm{m}$. 

\smallskip

Next, we proceed by recursion on $n$ to show the case $\bm{m}=(1,\dots,1)$. The base case $n=1$ is established by  \eqref{eq:Z1class}. 
Assume \eqref{eq:hypknm} and assume the statement holds for $n-1$, i.e., one has
\[
\left[\ov{\H}^k_{g,n-1} \right] = \mathsf{F}^k_{g,n-1} \quad \in A^{n-1}\left( \P\E^k_{g,n-1}\big|_{\M_{g,n-1}^\mathsf{rt}}  \right)
\]
and this extends over $\ov{\M}_{g,n-1}$ when  $n-1\leq k$.
 The argument in the previous paragraph shows that given $\bm{m}=(m,1,\dots,1)$ with $|\bm{m}|=n-1$ of length $(n-1)^-=n-m\leq n-1$, one has
\begin{equation}
\label{eq:HmisFm}
\left[ \ov{\H}^k_{g,\bm{m}} \right] = \mathsf{F}^k_{g,\bm{m}}  \quad \in A^{n-1}\left( \P\E^k_{g,n-m}\big|_{\M_{g,n-m}^\mathsf{rt}}  \right)
\end{equation}
and this extends over $\ov{\M}_{g,n-m}$ when  $n-1\leq k$.
Recall the loci ${E}_{\bm{I}}$ from \eqref{eq:EI} defined as ${E}_{\bm{I}}=\gamma_*\, \ov{\H}^k_{g,\bm{m}}$, where the map $\gamma$ is as in \eqref{eq:gamma0}.
From \eqref{eq:HmisFm} and \eqref{eq:EIclass}, it follows that 
\[
 [{E}_{\bm{I}}]  = \mathsf{E}_{\bm{I}} \quad \in A^{n}\left( \P\E^k_{g,n}\big|_{\M_{g,n}^\mathsf{rt}}  \right)
\]
and over $\ov{\M}_{g,n}$ when  $n\leq k$. The assumption \eqref{eq:hypknm} allows us to use Theorem \ref{thm:pinZnrhonZ1plus}.
 Then the statement for  $n$ and  $\bm{m}=(1,\dots,1)$ follows from Theorems \ref{thm:pinZnrhonZ1plus} and \ref{thm:Frec}. This ends the proof.
\end{proof}

\begin{remark}
\label{rmk:2}
In the case $k\geq 2$, $n=k(2g-2)$, and $\bm{m}=(1,\dots,1)$,
the proof of Theorem \ref{thm:graphFor-m} together with \eqref{eq:k2plusnmax} show that
\[
\mathsf{F}^k_{g,n} = \left[\ov{\H}_{g, n}^k \right] + \left[ E_{\rm{max}}^{\textrm{ab}} \right] \quad \in A^{n}\left( \P\E^k_{g,n}\big|_{\M_{g,n}^\mathsf{rt}}  \right).
\]
\end{remark}


\section{Examples}
\label{sec:ex}

We collect here some examples of the cycles $\mathsf{F}^k_{g,n}$ from Definition \ref{def:graphfor} and   $\mathsf{F}^k_{g,\bm{m}}$ from Definition \ref{def:Fm}, and verify our results in some special cases.

\subsection{The case $n=3$}
The cycle $\mathsf{F}^k_{g,3} \in R^3\left(\P\E^k_{g,3} \right)$ is
\begin{align*}
\mathsf{F}^k_{g,3} = &
\mbox{
\begin{tikzpicture}[baseline={([yshift=-.5ex]current bounding box.center)}]
      \path(0,0) ellipse (2 and 3);
      \tikzstyle{level 1}=[counterclockwise from=30,level distance=15mm,sibling angle=120]
      \node [draw,circle,inner sep=2.5] (A0) at (0:0) {$\scriptstyle{g}$}
            child {node [label=30: {$k\,\omega_2 - \eta$}]{}}
            child {node [label=150: {$k\,\omega_1 - \eta$}]{}}
            child {node [label=-90: {$k\,\omega_3 - \eta$}]{}};
    \end{tikzpicture}
}
-
\mbox{
\begin{tikzpicture}[baseline={([yshift=-.5ex]current bounding box.center)}]
      \path(0,0) ellipse (2 and 3);
      \tikzstyle{level 1}=[counterclockwise from=-90,level distance=15mm,sibling angle=120]
      \node [draw,circle,inner sep=2.5] (A0) at (180:2) {$\scriptstyle{g}$}
	    child {node [label=-90: {$k\,\omega - \eta$}]{}};
      \tikzstyle{level 1}=[counterclockwise from=-60,level distance=12mm,sibling angle=120]
      \node [draw,circle,fill] (A1) at (0:2) {$\scriptstyle{}$}
            child {node [label=-60: {$\scriptstyle{}$}]{}}
	    child {node [label=60: {$\scriptstyle{}$}]{}};
      \path (A0) edge [] node[auto, above=0.5, label={180:${k\,\omega-\eta}$}]{} (A1);
    \end{tikzpicture}
}  
-3
\mbox{
\begin{tikzpicture}[baseline={([yshift=-.5ex]current bounding box.center)}]
      \path(0,0) ellipse (2 and 3);
      \tikzstyle{level 1}=[counterclockwise from=-90,level distance=15mm,sibling angle=120]
      \node [draw,circle,inner sep=2.5] (A0) at (180:2) {$\scriptstyle{g}$};
      \tikzstyle{level 1}=[counterclockwise from=-60,level distance=12mm,sibling angle=60]
      \node [draw,circle,fill] (A1) at (0:2) {}
            child {node [label=-60: {}]{}}
	    child {node [label=0: {}]{}}
    	    child {node [label=60: {}]{}};
      \path (A0) edge [] node[auto, above=0.5, label={180:$(k\,\omega-\eta)^2$}]{} (A1);
    \end{tikzpicture}
}  \\
& 
-7
\mbox{
\begin{tikzpicture}[baseline={([yshift=-.5ex]current bounding box.center)}]
      \path(0,0) ellipse (2 and 3);
      \tikzstyle{level 1}=[counterclockwise from=-90,level distance=15mm,sibling angle=120]
      \node [draw,circle,inner sep=2.5] (A0) at (180:2) {$\scriptstyle{g}$};
      \tikzstyle{level 1}=[counterclockwise from=-60,level distance=12mm,sibling angle=60]
      \node [draw,circle,fill] (A1) at (0:2) {}
            child {node [label=-60: {}]{}}
	    child {node [label=0: {}]{}}
    	    child {node [label=60: {}]{}};
      \path (A0) edge []  node[auto, above=0.5, label={180:$k\,\omega-\eta$}]{}
      				  node[near end, below=0.1, label={-90:$\psi$}]{} (A1);
    \end{tikzpicture}
} 
-2
\mbox{
\begin{tikzpicture}[baseline={([yshift=-.5ex]current bounding box.center)}]
      \path(0,0) ellipse (2 and 3);
      \tikzstyle{level 1}=[counterclockwise from=-90,level distance=15mm,sibling angle=120]
      \node [draw,circle,inner sep=2.5] (A0) at (180:2) {$\scriptstyle{g}$};
      \tikzstyle{level 1}=[counterclockwise from=-60,level distance=12mm,sibling angle=60]
      \node [draw,circle,fill] (A1) at (0:2) {}
            child {node [label=-60: {}]{}}
	    child {node [label=0: {}]{}}
    	    child {node [label=60: {}]{}};
      \path (A0) edge []  node[auto, above=0.5, label={180:$\psi(k\,\omega-\eta)$}]{} (A1);
    \end{tikzpicture}
} 
-6
\mbox{
\begin{tikzpicture}[baseline={([yshift=-.5ex]current bounding box.center)}]
      \path(0,0) ellipse (2 and 3);
      \tikzstyle{level 1}=[counterclockwise from=-90,level distance=15mm,sibling angle=120]
      \node [draw,circle,inner sep=2.5] (A0) at (180:2) {$\scriptstyle{g}$};
      \tikzstyle{level 1}=[counterclockwise from=-60,level distance=12mm,sibling angle=60]
      \node [draw,circle,fill] (A1) at (0:2) {}
            child {node [label=-60: {}]{}}
	    child {node [label=0: {}]{}}
    	    child {node [label=60: {}]{}};
      \path (A0) edge []  node[near start, above=0.1, label={90:$\psi$}]{}
      				  node[near end, below=0.1, label={-90:$\psi$}]{} (A1);
    \end{tikzpicture}
} \\
&
+3
\mbox{
\begin{tikzpicture}[baseline={([yshift=-.5ex]current bounding box.center)}]
      \path(0,0) ellipse (3 and 2);
      \tikzstyle{level 1}=[counterclockwise from=-60, level distance=12mm, sibling angle=120]
      \node [draw,circle,fill] (A0) at (0:2.5) {}
	    child  {node [label=-60: {$\scriptstyle{}$}]{}}
    	    child  {node [label=60: {$\scriptstyle{}$}]{}};
       \tikzstyle{level 1}=[counterclockwise from=-90,level distance=12mm,sibling angle=60]
        \node [draw,circle,fill] (A1) at (0:0) {}
    	    child {node [label=-90: {$\scriptstyle{}$}]{}};
       \tikzstyle{level 1}=[counterclockwise from=120,level distance=12mm,sibling angle=120]
             \node [draw,circle,inner sep=2.5] (A2) at (180:3) {$\scriptstyle{g}$};
        \path (A0) edge [] node[]{} (A1);
        \path (A1) edge [] node[auto, above=0.5, label={180:$k\,\omega-\eta$}]{} (A2);        
    \end{tikzpicture}
} 
+2
\mbox{
\begin{tikzpicture}[baseline={([yshift=-.5ex]current bounding box.center)}]
      \path(0,0) ellipse (3 and 2);
      \tikzstyle{level 1}=[counterclockwise from=-60, level distance=12mm, sibling angle=120]
      \node [draw,circle,fill] (A0) at (0:2.5) {}
	    child  {node [label=-60: {$\scriptstyle{}$}]{}}
    	    child  {node [label=60: {$\scriptstyle{}$}]{}};
       \tikzstyle{level 1}=[counterclockwise from=-90,level distance=12mm,sibling angle=60]
        \node [draw,circle,fill] (A1) at (0:0) {}
    	    child {node [label=-90: {$\scriptstyle{}$}]{}};
       \tikzstyle{level 1}=[counterclockwise from=120,level distance=12mm,sibling angle=120]
             \node [draw,circle,inner sep=2.5] (A2) at (180:3) {$\scriptstyle{g}$};
        \path (A0) edge [] node[]{} (A1);
        \path (A1) edge [] node[near end, above=0.1, label={90:$\psi$}]{} (A2);        
    \end{tikzpicture}
} .
\end{align*}

\subsection{A vanishing cycle}
\label{sec:BP}
Specializing the above cycle at $g=2$ and $k=1$, we obtain the vanishing cycle
\[
\mathsf{F}^1_{2,3}=0 \qquad \in R^3\left( \P\E^1_{2,3}\big|_{\M_{2,3}^\mathsf{rt}} \right).
\]
This is a special case of Corollary \ref{cor:tautrel}, and can be verified as follows.
Since $\eta^2 = \eta\,\lambda_1-\lambda_2$ in $A^*(\P\E^1_{2,3})$ where $\lambda_i:=c_i(\E^1_{2,3})=0$ in $A^*(\M_{2,3}^\mathsf{rt})$, all terms with a factor of $\eta^2$ vanish. 
The coefficient of $\eta$ is a cycle in $R^2(\M_{2,3}^\mathsf{rt})$ which recovers the restriction of the  relation in $A^2(\ov{\M}_{2,3})$ from \cite{belorousski2000descendent} plus the pull-back of a relation in $A^2(\M_{2,2}^\mathsf{rt})$ from \cite{MR1672112}.
Finally, the constant term in $\eta$ is in $R^3(\M_{2,3}^\mathsf{rt})$
and vanishes due to known relations (see e.g.,~\cite{tavakol2014tautological}).

\subsection{The Logan divisor}
\label{eq:Logandiv}
Here we verify how the cycle $\mathsf{F}^k_{g,n}$ for $n=g$ and $k=1$ recovers a well-known divisor on moduli spaces of curves. 

The \textit{Logan divisor} $\mathfrak{L}_g$  in ${\M}_{g,g}$ is the locus of $g$-pointed genus $g$ curves $(C,P_1,\dots,P_g)$ such that the divisor $P_1+\cdots+P_g$ moves in a pencil \cite{MR1953519}.
Equivalently, via the Riemann-Roch theorem, the $g$ marked points are required to be zeros of a common abelian differential. 
Hence the Logan divisor is the push-forward of the incidence locus $\ov{\H}^1_{g,g}$ via the forgetful map $\varphi \colon \mathbb{PE}^1_{g,g}\rightarrow \ov{\M}_{g,g}$, and one has
\[
\Big[ \ov{\mathfrak{L}}_g \Big] = \varphi_* \left[ \ov{\H}^1_{g,g} \right] \in \mathrm{Pic}\left( \ov{\M}_{g,g} \right).
\]

As a check on Theorem \ref{thm:graphFor}, we verify that $\varphi_*\, \mathsf{F}^1_{g,g} \equiv \ov{\mathfrak{L}}_g$ on $\M^{\rm rt}_{g,g}$.
The cycle $\mathsf{F}^1_{g,g}$ is a polynomial in $\eta$ of degree $g$. First, we replace $(-\eta)^g$ via the relation $(-\eta)^g = - \lambda_1 (-\eta)^{g-1} + \cdots$ in $A^*\left(\P\E^1_{g,g} \right)$. Then we use that $\varphi_*\left((-\eta)^d\right)=1$ if $d=g-1$, and vanishes if $d<g-1$. It follows that
\[
\varphi_*\, \mathsf{F}^1_{g,g} = \left[ \mathsf{F}^1_{g,g} \right]_{g-1} -\lambda_1 \, \left[ \mathsf{F}^1_{g,g} \right]_{g}
\]
where $[X]_i$ if the coefficient of $(-\eta)^i$ in $X$. Hence we have
\[
\varphi_*\, \mathsf{F}^1_{g,g} = \sum_{i=1}^g \omega _i - \lambda_1 - \sum_{i=2}^g {i \choose 2} \,\delta_{0:i} \in \mathrm{Pic}\left( \M^{\rm rt}_{g,g} \right).
\]
This checks with the restriction on $\M^{\rm rt}_{g,g}$ of the  class of $\mathfrak{L}_g$ from \cite{MR1953519}. Here $\delta_{0:i}$ is the class of the divisor whose general element has an elliptic tail containing $i$ marked points.
Additionally, for $g=2$, using Theorem \ref{thm:classZ2}, one verifies  $\varphi_*\,\mathsf{F}^1_{2,2} \equiv \ov{\mathfrak{L}}_2$
over $\ov{\M}_{2,2}$, as in \eqref{eq:phipfwdH122}.

\subsection{The case of one heavy marked point}
When $\bm{m}=(a)$ for some positive~$a$, the cycle $\mathsf{F}^k_{g,\bm{m}}$ from Definition \ref{def:Fm} is
\[
\mathsf{F}^k_{g,(a)}= \sum_{b=0}^{a-1} e_b\left(1, \dots,  a-1\right)\,\left(k\,\omega_1 - \eta\right)^{a-b} \, \psi_1^{b} \qquad \in R^a\left(\P\E^k_{g,1} \right)
\]
where $e_b(x_1,\dots,x_{a-1})$ is the $b$-th elementary symmetric function in $a-1$ variables, here evaluated at $1,\dots, a-1$. Since $\omega_1=\psi_1$ in $R^*\left(\P\E^k_{g,1} \right)$, this cycle can be rewritten as
\begin{equation}
\label{eq:Fkga}
\mathsf{F}^k_{g,(a)} = \prod_{b=0}^{a-1}\left((k+b)\psi_1 - \eta\right) \qquad \in R^a\left(\P\E^k_{g,1} \right).
\end{equation}
Theorem \ref{thm:graphFor-m} shows that
\begin{equation}
\label{eq:HisFkga}
\left[ \ov{\H}^k_{g,(a)} \right] = \mathsf{F}^k_{g,(a)} \quad \in A^{a}\left( \P\E^k_{g,1}\big|_{\M_{g,1}}  \right)
\end{equation}
when either $k=1$, or $k\geq 2$ and $a \leq k(2g-2)-1$, and this holds in $A^{a}\left( \P\E^k_{g,1} \right)$ when $a\leq k$.

For instance, when $a=k=g=2$, the  locus $\ov{\H}^2_{2,(2)}$ consists of two equidimensional components:
\begin{equation}
\label{eq:H222}
\ov{\H}^2_{2,(2)} = \ov{\H}^2_{2,1}(2,1,1) \cup \ov{\H}^{2, \mathrm{ab}}_{2,1}(2,2) \subseteq\P\E^2_{2,1}
\end{equation}
where $\ov{\H}^2_{2,1}(2,1,1)$ consists of quadratic differentials that  have generically a double zero at the marked point and two simple zeros elsewhere, and $\ov{\H}^{2, \mathrm{ab}}_{2,1}(2,2)$ consists of squares of abelian differentials that have generically double zeros at both the marked point and  its (unmarked) hyperelliptic conjugate point. In this case, the class of $\ov{\H}^2_{2,(2)}$ decomposes as:
\begin{equation}
\label{eq:F222}
\mathsf{F}^2_{2,(2)} = \left[ \ov{\H}^2_{2,1}(2,1,1) \right] +  \left[ \ov{\H}^{2, \mathrm{ab}}_{2,1}(2,2) \right] \quad\in A^2\left( \P\E^2_{2,1} \right).
\end{equation}
As a check, one has $\varphi_* \,\mathsf{F}^2_{2,(2)} = \varphi_*(\eta^2)=1$, where $\varphi\colon \P\E^2_{2,1} \rightarrow \ov{\M}_{2,1}$ is the forgetful map. Indeed, the restriction of $\varphi$ to $\ov{\H}^{2, \mathrm{ab}}_{2,1}(2,2)$ has  degree one, while $\varphi$ has  positive-dimensional fibers when restricted to the other component.

\subsection{The case of one heavy unmarked point}
By forgetting the marked point, the locus ${\H}^k_{g,(a)}$ gives rise to the stratum 
\[
{\H}^k_{g}(a)\subseteq\P\E^k_g 
\]
consisting of $k$-differentials vanishing with order at least $a$ at some point. 
Let $\pi\colon \P\E^k_{g,1} \rightarrow \P\E^k_g$ be the map obtained by forgetting the marked point, and let $\kappa_b:=\pi_*\left(\psi^{b+1} \right)$.
We deduce:

\begin{corollary}
\label{cor:heavypt}
For $g\geq 2$, one has
\[
\left[\ov{\H}^k_{g}(a)\right] = \sum_{b=1}^{a} (-1)^{a-b} \, e_b(k,\dots,k+a-1) \, \kappa_{b-1} \,\eta^{a-b} 
\]
in:
\begin{enumerate}[(i)]
\item $A^{a-1}\left( \P\E^k_{g}\big|_{\M_{g}}  \right)$ when
\[
\left\{
\begin{array}{l}
\mbox{$k=1$; or}\\[.1cm]
\mbox{$k\geq 2$ and $a \leq k(2g-2)-1$;}
\end{array}
\right.
\]
\item $A^{a-1}\left( \P\E^k_{g} \right)$ when $a\leq k$. 
\end{enumerate}
\end{corollary}

\begin{proof}
The hypotheses allow one to use Theorem \ref{thm:graphFor-m} to conclude that \eqref{eq:HisFkga} holds
when either $k=1$, or $k\geq 2$ and $a \leq k(2g-2)-1$, and holds in $A^{a}\left( \P\E^k_{g,1} \right)$ when $a\leq k$.
 It  remains to compute the push-forward  via $\pi$. 
Expanding $\mathsf{F}^k_{g,(a)}$ from \eqref{eq:Fkga} as a polynomial in $\psi_1$, one has
\[
\pi_* \left( \mathsf{F}^k_{g,(a)} \right) = \sum_{b=1}^{a} (-1)^{a-b} \, e_b(k,\dots,k+a-1) \, \kappa_{b-1} \,\eta^{a-b} \quad \in R^{a-1}\left(\P\E^k_g \right).
\]
Hence the statement.
\end{proof}

For instance, when $a=2$ and $k\geq 2$, one has 
\begin{equation}
\label{eq:Hkg2}
\left[\ov{\H}^k_{g}(2)\right] = k(k+1)\, \kappa_1 - (2k+1)(2g-2)\,\eta \quad
\in \mathrm{Pic}\left( \P\E^k_{g} \right).
\end{equation}
This recovers the class of the divisorial stratum for $k\geq 2$ computed in \cite{korotkin2019tau}.
When $a=k=g=2$, the locus $\ov{\H}^2_{2}(2)$ is the union of the two divisors 
$\ov{\H}^2_{2}(2,1,1)$ and $\ov{\H}^{2, \mathrm{ab}}_{2}(2,2)$ obtained by forgetting the marked point on the two components in \eqref{eq:H222}, and \eqref{eq:Hkg2} decomposes as the weighted sum 
\[
\left[\ov{\H}^2_{2}(2)\right] = \left[ \ov{\H}^2_{2}(2,1,1) \right] +  2\left[ \ov{\H}^{2, \mathrm{ab}}_{2}(2,2) \right] 
\quad\in \mathrm{Pic}\left( \P\E^2_{2} \right).
\]
 Here the factor of $2$ in front of the second summand is due to the fact that the restriction of $\pi$ to \mbox{$\ov{\H}^2_{2,1}(2,2)\rightarrow \ov{\H}^2_{2}(2,2)$} has generically degree $2$, corresponding to the $2$ zeros of the square differentials.


%



\newpage

\appendix


\section{The incidence locus for $n=2$}
\label{sec:Hkg2}
The class of the incidence locus for $n=1$ is given by:

\begin{lemma}[{\cite[\S 1.6]{sauvagetcohomology}, \cite[\S4]{korotkin2019tau}}]
\label{lem:Z1class}
For $k\geq 1$ and $g\geq 2$, one has 
\[
\ov{\H}_{g, 1}^k \equiv k \,\omega -\eta \quad \in \mathrm{Pic}\left(\P\E^k_{g,1} \right).
\]
\end{lemma}

For $n=2$,  the class of the locus $\ov{\H}_{g, 2}^k$ in $\P\E^k_{g,2}$ follows from \cite{sauvagetcohomology, sauvaget2020volumes}. 
In this section, we provide an alternative proof   by applying the incidence variety compactification from \cite{bcggm1, bcggm}.

First we set some notation. 
Recall the locus $E_{\bm{I}}$ from \eqref{eq:EI}. Here we need the case $\bm{I}=\{1\}$. 
Its class   is the push-forward of the class  in Lemma \ref{lem:Z1class} via the map $\gamma$ in \eqref{eq:gamma0}.
Moreover, let
\[
\mbox{
\begin{tikzpicture}[baseline={([yshift=-.5ex]current bounding box.center)}]
      \path(0,0) ellipse (2 and 2);
      \tikzstyle{level 1}=[counterclockwise from=-60,level distance=15mm,sibling angle=120]
      \node [draw,circle,inner sep=2.5] (A0) at (180:1.5) {$\scriptstyle{g-1}$};
      \tikzstyle{level 1}=[counterclockwise from=-60,level distance=15mm,sibling angle=120]
      \node [draw,circle,inner sep=2.5] (A1) at (0:1.5) {$\scriptstyle{1}$}
            child {node [label=-60: {$\scriptstyle{2}$}]{}}
	    child {node [label=60: {$\scriptstyle{1}$}]{}};
      \path (A0) edge [] node[auto, below=0.5, label={-90:$\scriptstyle{\mu_1=0}$}]{} (A1);
    \end{tikzpicture}
}
\quad\subset \P\E^k_{g,2}
\]
be the locus consisting of curves with an elliptic tail containing both marked points together with a class of a stable $k$-differential $\mu$ 
whose restriction $\mu_1$ to the elliptic component is identically zero.
 Such a stable $k$-differential $\mu$ has a pole of order at most $k-1$ at the preimage of the node in the genus $g-1$ component. 
This locus has codimension $k+1$ in $\P\E^k_{g,2}$. 

\begin{theorem}[{\cite{sauvagetcohomology, sauvaget2020volumes}}]
\label{thm:classZ2}
Let $g\geq 2$.
For $k=1$, one has
\begin{multline}
\ov{\H}_{g, 2}^1 \equiv \left(\omega_1-\eta\right)\left(\omega_2-\eta\right) 
- E_{\{1\}}
-
\mbox{
\begin{tikzpicture}[baseline={([yshift=-.5ex]current bounding box.center)}]
      \path(0,0) ellipse (2 and 2);
      \tikzstyle{level 1}=[counterclockwise from=-60,level distance=15mm,sibling angle=120]
      \node [draw,circle,inner sep=2.5] (A0) at (180:1.5) {$\scriptstyle{g-1}$};
      \tikzstyle{level 1}=[counterclockwise from=-60,level distance=15mm,sibling angle=120]
      \node [draw,circle,inner sep=2.5] (A1) at (0:1.5) {$\scriptstyle{1}$}
            child {node [label=-60: {$\scriptstyle{2}$}]{}}
	    child {node [label=60: {$\scriptstyle{1}$}]{}};
      \path (A0) edge [] node[auto,below=0.4, label={-90:$\scriptstyle{\mu_1=0}$}]{} (A1);
    \end{tikzpicture}
}
-\sum_{i=1}^{\lfloor g/2\rfloor}
\mbox{
\begin{tikzpicture}[baseline={([yshift=-.5ex]current bounding box.center)}]
      \path(0,0) ellipse (2 and 2);
      \tikzstyle{level 1}=[counterclockwise from=-60,level distance=15mm,sibling angle=120]
      \node [draw,circle,inner sep=2.5] (A0) at (120:1.5) {$\scriptstyle{i}$};
      \tikzstyle{level 1}=[counterclockwise from=-45,level distance=12mm,sibling angle=90]
      \node [draw,circle,fill] (A1) at (0:1.5) {$\scriptstyle{}$}
      child {node [label=-45: {$\scriptstyle{2}$}]{}}
      child {node [label=45: {$\scriptstyle{1}$}]{}};
      \tikzstyle{level 1}=[counterclockwise from=-60,level distance=15mm,sibling angle=60]
      \node [draw,circle,inner sep=2.5] (A2) at (240:1.5) {$\scriptstyle{g-i}$};
      \path (A0) edge [] (A1);
      \path (A1) edge [] (A2);
    \end{tikzpicture}
} 
\end{multline}
in $A^2\left( \P\E^1_{g,2} \right)$, and for $k\geq 2$, one has
\[
 \ov{\H}_{g, 2}^k \equiv \left(k\,\omega_1-\eta\right)\left(k\,\omega_2-\eta\right) 
- E_{\{1\}}
\quad \in A^2\left(\P\E^k_{g,2} \right).
\]
\end{theorem}

\noindent We deduce Theorem \ref{thm:classZ2}  from the next  Proposition \ref{prop:pi1Z1pi2Z1} and Lemma~\ref{lem:Z1class}.

\subsection{Preliminaries}
\label{sec:IVCrecap}
In this and the next section, we use the incidence variety compactification from \cite{bcggm1, bcggm}.  We briefly review it here. 

Let $C$ be a nodal curve and $\widetilde{C}\rightarrow C$  its normalization. Let  $Q_1, \dots, Q_s$ be the nodes in $C$, and for each $Q_i$, let $Q_i^+$ and $Q_i^-$ be its two preimages in~$\widetilde{C}$.
Stable $k$-differentials on $C$ are identified via pull-back with  $k$-differentials on $\widetilde{C}$ admitting poles of order at most $k$ at the preimages of the nodes and satisfying the \textit{$k$-residue condition} at $Q_i^+$ and $Q_i^-$ for each $i$. 

A \textit{twisted $k$-differential} $\tau$ on a nodal curve $C$ is a collection $\tau=\{\tau_i\}_i$ of (possibly meromorphic) $k$-differentials, one for each irreducible component of $C$. We will refer to the $\tau_i$  as the \textit{aspects} of $\tau$. 

The closure of strata of smooth pointed curves and $k$-differentials with all zeros and poles marked with prescribed multiplicities 
 is described for $k=1$ in \cite[Thm 1.3]{bcggm1} and for $k\geq 2$ in \cite[Thm 1.5]{bcggm}. 
 An element $(C,P_1,\dots,P_n, \mu)$, where $(C,P_1,\dots,P_n)$ is a stable pointed curve and $\mu$ is a (possibly meromorphic) stable $k$-differential, is in the closure of such strata if and only if it admits a   twisted $k$-differential compatible with a certain full order on the dual graph of $C$. 
The maxima in the full order are required to correspond to the irreducible components where $\mu$ is not identically zero. The twisted $k$-differential is required to: coincide with $\mu$ on the maxima; have zeros and poles at the marked points with prescribed multiplicities; satisfy the matching zero and pole order property at \textit{vertical nodes};
satisfy the $k$-residue condition at \textit{horizontal nodes}; and satisfy the \textit{global $k$-residue condition}.
We refer the reader to \cite{bcggm1, bcggm} for the description of these compatibility conditions.

\subsection{A set-theoretic statement}
The following  statement is a preliminary step toward the proof of Theorem \ref{thm:classZ2}. For $i\in\{1,2\}$, let $\pi_i\colon\P\E^k_{g,2} \rightarrow\P\E^k_{g,1}$ be the map obtained by forgetting the $i$-th marked point (and relabeling the remaining marked point as $P_1$).

\begin{proposition}
\label{prop:pi1Z1pi2Z1loci}
The intersection 
\begin{equation}
\label{eq:pi1pi2Hkg1}
\pi_1^{-1}\left( \ov{\H}_{g, 1}^k\right) \cap \pi_2^{-1}\left( \ov{\H}_{g, 1}^k \right) \quad\subset \P\E^k_{g,2}
\end{equation}
consists of the following components for $k=1$:
\begin{multline}
\label{eq:listZ1Z1}
\ov{\H}_{g, 2}^1, \quad
E_{\{1\}},
\quad
\mbox{
\begin{tikzpicture}[baseline={([yshift=-.5ex]current bounding box.center)}]
      \path(0,0) ellipse (2 and 2);
      \tikzstyle{level 1}=[counterclockwise from=-60,level distance=15mm,sibling angle=120]
      \node [draw,circle,inner sep=2.5] (A0) at (180:1.5) {$\scriptstyle{g-1}$};
      \tikzstyle{level 1}=[counterclockwise from=-60,level distance=15mm,sibling angle=120]
      \node [draw,circle,inner sep=2.5] (A1) at (0:1.5) {$\scriptstyle{1}$}
            child {node [label=-60: {$\scriptstyle{2}$}]{}}
	    child {node [label=60: {$\scriptstyle{1}$}]{}};
      \path (A0) edge [] node[auto,below=0.4, label={-90:$\scriptstyle{\mu_1=0}$}]{} (A1);
    \end{tikzpicture}
}, \quad
\mbox{
\begin{tikzpicture}[baseline={([yshift=-.5ex]current bounding box.center)}]
      \path(0,0) ellipse (2 and 2);
      \tikzstyle{level 1}=[counterclockwise from=-60,level distance=15mm,sibling angle=120]
      \node [draw,circle,inner sep=2.5] (A0) at (120:1.5) {$\scriptstyle{i}$};
      \tikzstyle{level 1}=[counterclockwise from=-45,level distance=12mm,sibling angle=90]
      \node [draw,circle,fill] (A1) at (0:1.5) {$\scriptstyle{}$}
      child {node [label=-45: {$\scriptstyle{2}$}]{}}
      child {node [label=45: {$\scriptstyle{1}$}]{}};
      \tikzstyle{level 1}=[counterclockwise from=-60,level distance=15mm,sibling angle=60]
      \node [draw,circle,inner sep=2.5] (A2) at (240:1.5) {$\scriptstyle{g-i}$};
      \path (A0) edge [] (A1);
      \path (A1) edge [] (A2);
    \end{tikzpicture}
} \quad\mbox{for $1\leq i\leq \lfloor g/2\rfloor$,}
\end{multline}
and for $k\geq 2$:
\begin{equation}
\label{eq:listZ1Z1kgeq2}
\ov{\H}_{g, 2}^k  \quad \mbox{and} \quad 
E_{\{1\}}.
\end{equation}
\end{proposition}

\begin{proof}
By definition, $\ov{\H}_{g, 2}^k$ is  the only component of the intersection \eqref{eq:pi1pi2Hkg1} which intersect the locus of smooth curves for all $k\geq 1$.
To study the additional components over the locus of nodal curves, we consider the incidence variety compactification of  
${\H}_{g, k(2g-2)}^k \subset \P\E^k_{g,k(2g-2)}$
from \cite{bcggm1, bcggm} and regard $\ov{\H}_{g, 1}^k$ and $\ov{\H}_{g, 2}^k$  as the images of $\ov{\H}_{g, k(2g-2)}^k$ under the  maps 
\[
\P\E^k_{g,k(2g-2)}\rightarrow\P\E^k_{g,1} \qquad \mbox{and}\qquad \P\E^k_{g,k(2g-2)}\rightarrow \P\E^k_{g,2} 
\]
obtained by forgetting all but the first one or  two marked points, respectively. This provides the description of the boundaries of $\ov{\H}_{g, 1}^k$ and $\ov{\H}_{g, 2}^k$ in terms of twisted $k$-differentials that we use below.

The intersection \eqref{eq:pi1pi2Hkg1} consists of $(C,P_1, P_2, \mu)$, where $(C,P_1, P_2)$ is a stable two-pointed curve and $\mu$ is the class of a stable $k$-differential, such that each marked point $P_i$ is a  zero of some twisted $k$-differential of type~$\mu$. 

The locus $\ov{\H}_{g, 2}^k$ consists of $(C,P_1, P_2, \mu)$ where \textit{both} marked points are  zeros of the \textit{same} twisted $k$-differential of type $\mu$ for all $k\geq 1$.

For a general $(C,P_1, P_2, \mu)$ in $E_{\{1\}}$ with $k\geq 1$,  the stable $k$-differential $\mu$ has a simple zero at the preimage of the node on the genus $g$  component. Then each marked point is a simple zero of a twisted $k$-differential of type  $\mu$ whose aspect on the rational tail has a pole of order $2k+1$ at the node.

When $k=1$, consider the third locus listed in \eqref{eq:listZ1Z1} and select a general element in it, with  elliptic tail denoted $E$. Since we can always find a twisted abelian differential whose aspect on  $E$ has 
a pole of order $2$ at the node and a zero at one of the marked points, the image of the third locus under  $\pi_i$  is in $\ov{\H}_{g, 1}^1$, for $i\in\{1,2\}$. The aspect of such a twisted abelian differential on  $E$ does not vanish generically at both marked points, hence such a locus is not in $\ov{\H}_{g, 2}^1$.
However, when $k\geq 2$, such a locus has codimension $k+1$ and is contained in $\ov{\H}_{g, 2}^k$. Indeed,  one can always find a twisted $k$-differential whose aspects on $E$ has a pole of order at least $k+1$ at the node and vanishes at \textit{both}  marked points.

For the remaining types of loci in \eqref{eq:listZ1Z1}, consider their general element. The restriction of a general abelian differential $\mu$ to the two components of positive genus is holomorphic and non-zero at the nodes. 
The global residue condition \cite[Def.~1.2(4)]{bcggm1} is satisfied by a twisted abelian differential of type $\mu$ whose aspect on the rational component has simple zeros at one of the marked points and at a smooth point generically away from the other marked point, and poles of order $2$ at both nodes. 
This implies that the last  loci in \eqref{eq:listZ1Z1} are in \eqref{eq:pi1pi2Hkg1} when $k=1$. However, when $k\geq 2$, the restriction of a general stable $k$-differential $\mu$ to the two components of positive genus has a pole of order $k$ at each node, hence the aspect on the rational component of a twisted $k$-differential of type $\mu$ cannot have any zeros. It follows that the last loci in \eqref{eq:listZ1Z1} do not lie in the intersection \eqref{eq:pi1pi2Hkg1} for $k\geq 2$.

Next, we show that the loci in \eqref{eq:listZ1Z1} and \eqref{eq:listZ1Z1kgeq2}  are the \textit{only} ones in \eqref{eq:pi1pi2Hkg1} for $k=1$ and $k\geq 2$, respectively.
To detect components of \eqref{eq:pi1pi2Hkg1} over the locus of nodal curves, we analyze  a general element $(C,P_1, P_2, \mu)$ in the inverse image of the locus of nodal curves from $\ov{\M}_{g,2}$ and in \eqref{eq:pi1pi2Hkg1}.
It suffices to consider the following cases: (a) the curve $C$ has exactly one node; or (b) the curve $C$ has two nodes. 

For the one-nodal case, we  show that $(C, P_1, P_2, \mu)$ is in one of the first three loci in \eqref{eq:listZ1Z1} when $k=1$, or in one of the loci in \eqref{eq:listZ1Z1kgeq2} when $k\geq 2$. 

If $C$ is one-nodal with a non-disconnecting node, then both marked points need to be in the support of $\mu$. Hence $(C, P_1, P_2, \mu)$ is in fact in $\ov{\H}_{g, 2}^k$. 

If $C$ is one-nodal with a rational tail containing both marked points, then necessarily $\mu$ vanishes at the preimage of the node in the genus $g$ component. Hence $(C, P_1, P_2, \mu)$ is in $E_{\{1\}}$. 

In the remaining cases of one-nodal $C$ with a disconnecting node, assume first that each marked point is in a component of $C$ where $\mu$ is not identically zero. Then
 $\mu$ vanishes at both marked points, hence $(C, P_1, P_2, \mu)$ is in $\ov{\H}_{g, 2}^k$. 

We are left with the case when $C$ is one-nodal, and one or both marked points are on a component of $C$ of positive genus where $\mu$ is identically zero. 
If $k=1$ and $\mu$ is identically zero on an elliptic component  containing both marked points, then $(C, P_1, P_2, \mu)$ is in the third locus listed in \eqref{eq:listZ1Z1}. 
Otherwise, $(C, P_1, P_2, \mu)$ is the general element of a locus of codimension higher than two inside 
 $\ov{\H}_{g, 2}^k$. Indeed, there exists a twisted $k$-differential of type $\mu$ whose aspect on the component of $C$ where $\mu$ is identically zero has a zero at the one or two marked points in there and has a pole of order at least $k+1$ at the node; in case one marked point is on the component where $\mu$ is not identically zero, then such point must be a zero  of $\mu$.

Finally, we analyze the two-nodal case, 
and show that necessarily $k=1$ and  $(C, P_1, P_2, \mu)$ is in the  loci of the fourth type listed in \eqref{eq:listZ1Z1}. 

If $C$ has only disconnecting nodes and nonrational components, then both marked points  are generically away from the zeros of $\mu$, hence such an element  is not in \eqref{eq:pi1pi2Hkg1}. Likewise, a general element of a locus of two-nodal curves with a twice marked rational tail is not in  \eqref{eq:pi1pi2Hkg1}, since the support of $\mu$ does not generically contain the preimage of the node in the nonrational component adjacent to the rational tail. 
If $C$ contains a rational bridge with one marked point, then the other marked point is generically away from the support of~$\mu$.
If instead $C$ has a twice marked rational bridge, then $(C, P_1, P_2, \mu)$ is in one of the loci of the fourth type listed in \eqref{eq:listZ1Z1}, and as discussed at the beginning of the proof, such loci are in \eqref{eq:pi1pi2Hkg1} but not in $\ov{\H}_{g, 2}^k$ only for $k=1$.

Similar arguments cover the case when $C$ is irreducible with two non-disconnecting nodes, as well as the case when $C$ has both a non-disconnecting node and a disconnecting node. We are left with the case when $C$ has two components meeting in two points.  If both  components of $C$ are nonrational,  the restriction of $\mu$ to each component has generically  poles of order $k$ at both nodes, and the marked points are generically not in the support  of $\mu$. If $C$ has a (marked) rational component, the restriction of $\mu$ to the genus $g-1$ component  has generically  poles of order $k$ at the two nodes, forcing there to be no zeros of $\mu$ on the rational component. 
\end{proof}

\subsection{Computation of multiplicities}
The intersection \eqref{eq:pi1pi2Hkg1} is generically transverse over $\M_{g,2}$, hence the component $\ov{\H}_{g, 2}^k$ has multiplicity one in the intersection for all $k\geq 1$.
Then Proposition \ref{prop:pi1Z1pi2Z1loci} implies that there exist $a,b,c_i\in \mathbb{Q}$, with $i=1,\dots, \lfloor  g/2 \rfloor$, such that for $k=1$, one has
\begin{multline}
\label{eq:aAbBciCi}
\pi_1^*\left[ \ov{\H}_{g, 1}^1\right] \cdot \pi_2^*\left[ \ov{\H}_{g, 1}^1 \right] = \left[\ov{\H}_{g, 2}^1 \right]
+ a \left[ E_{\{1\}} \right]
 + b \mbox{
\begin{tikzpicture}[baseline={([yshift=-.5ex]current bounding box.center)}]
      \path(0,0) ellipse (2 and 2);
      \tikzstyle{level 1}=[counterclockwise from=-60,level distance=15mm,sibling angle=120]
      \node [draw,circle,inner sep=2.5] (A0) at (180:1.5) {$\scriptstyle{g-1}$};
      \tikzstyle{level 1}=[counterclockwise from=-60,level distance=15mm,sibling angle=120]
      \node [draw,circle,inner sep=2.5] (A1) at (0:1.5) {$\scriptstyle{1}$}
            child {node [label=-60: {$\scriptstyle{2}$}]{}}
	    child {node [label=60: {$\scriptstyle{1}$}]{}};
      \path (A0) edge [] node[auto,below=0.4, label={-90:$\scriptstyle{\mu_1=0}$}]{} (A1);
    \end{tikzpicture}
} 
+ \sum_{i=1}^{\lfloor  g/2 \rfloor} c_i \mbox{
\begin{tikzpicture}[baseline={([yshift=-.5ex]current bounding box.center)}]
      \path(0,0) ellipse (2 and 2);
      \tikzstyle{level 1}=[counterclockwise from=-60,level distance=15mm,sibling angle=120]
      \node [draw,circle,inner sep=2.5] (A0) at (120:1.5) {$\scriptstyle{i}$};
      \tikzstyle{level 1}=[counterclockwise from=-45,level distance=12mm,sibling angle=90]
      \node [draw,circle,fill] (A1) at (0:1.5) {$\scriptstyle{}$}
      child {node [label=-45: {$\scriptstyle{2}$}]{}}
      child {node [label=45: {$\scriptstyle{1}$}]{}};
      \tikzstyle{level 1}=[counterclockwise from=-60,level distance=15mm,sibling angle=60]
      \node [draw,circle,inner sep=2.5] (A2) at (240:1.5) {$\scriptstyle{g-i}$};
      \path (A0) edge [] (A1);
      \path (A1) edge [] (A2);
    \end{tikzpicture}
}
\end{multline}
in $A^2\left( \P\E^1_{g,2}\right)$, and for $k\geq 2$, one has
\begin{equation}
\label{eq:aAbBciCik2}
\pi_1^*\left[ \ov{\H}_{g, 1}^k\right] \cdot \pi_2^*\left[ \ov{\H}_{g, 1}^k \right] = \left[\ov{\H}_{g, 2}^k \right]
+ a \left[ E_{\{1\}} \right]
\in A^2\left(\P\E^k_{g,2}\right).
\end{equation}

\begin{proposition}
\label{prop:pi1Z1pi2Z1}
One has
\begin{align*}
a&=1 & \mbox{for  $k\geq 1$,} &&\mbox{and}&&
b=c_i &=1 &\mbox{for all $i$ and $k=1$.}
\end{align*}
\end{proposition}

\begin{proof}
To find the coefficient $a$, let $\pi\colon \P\E^k_{g,2} \ra \P\E^k_{g}$ be the forgetful map and consider $\pi_*$ of \eqref{eq:aAbBciCi} and \eqref{eq:aAbBciCik2}.
One has 
\begin{align*}
&\pi_*\left(\pi_1^*\left[ \ov{\H}_{g, 1}^k\right] \cdot \pi_2^*\left[ \ov{\H}_{g, 1}^k \right]\right) &&= k^2(2g-2)^2\left[\P\E^k_{g}\right], \\ 
&\pi_*\left(\ov{\H}_{g, 2}^k\right) &&= k(2g-2)(k(2g-2)-1)\left[\P\E^k_{g}\right], \\
&\pi_*\left(
E_{\{1\}}
\right) &&= k(2g-2)\left[\P\E^k_{g}\right], 
\end{align*}
for all $k\geq 1$, and the push-forward of the other classes in \eqref{eq:aAbBciCi} vanishes. Hence $a=1$. 
This  concludes the proof of the statement for $k\geq 2$.

\smallskip

For $k=1$, to find the coefficient $b$, we restrict to a test surface $S$ in $\P\E^1_{g,2}$ defined as follows: consider a pencil of plane cubics of degree $12$ along with a section of the Hodge bundle over $\ov{\M}_{1,1}$,  identify one of its basepoints with a general point on a fixed general genus $g-1$ curve with a general abelian differential, vary the first marked point along the genus $g-1$ component, and let the second marked point be one of the other basepoints of the pencil of plane cubics. 
We compute the following intersections below:
\begin{align}
\label{eq:Sint1}
&S \cdot \pi_1^*\left[ \ov{\H}_{g, 1}^1\right] \cdot \pi_2^*\left[ \ov{\H}_{g, 1}^1 \right] &&= 2(g-1) -1, \\
\label{eq:Sint2}
&S \cdot \left[\ov{\H}_{g, 2}^1 \right] &&= 2(g-1) -2, \\
\label{eq:Sint3}
&S \cdot \mbox{
\begin{tikzpicture}[baseline={([yshift=-.5ex]current bounding box.center)}]
      \path(0,0) ellipse (2 and 2);
      \tikzstyle{level 1}=[counterclockwise from=-60,level distance=15mm,sibling angle=120]
      \node [draw,circle,inner sep=2.5] (A0) at (180:1.5) {$\scriptstyle{g-1}$};
      \tikzstyle{level 1}=[counterclockwise from=-60,level distance=15mm,sibling angle=120]
      \node [draw,circle,inner sep=2.5] (A1) at (0:1.5) {$\scriptstyle{1}$}
            child {node [label=-60: {$\scriptstyle{2}$}]{}}
	    child {node [label=60: {$\scriptstyle{1}$}]{}};
      \path (A0) edge [] node[auto,below=0.4, label={-90:$\scriptstyle{\mu_1=0}$}]{} (A1);
    \end{tikzpicture}
} &&= 1.
\end{align}
Moreover, $S$ has empty intersection with the remaining classes in \eqref{eq:aAbBciCi}. These intersections imply that $b=1$.

Since $\lambda_1=1$ on the elliptic pencil, where $\lambda_1$ is the first Chern class of the Hodge bundle, the section of the Hodge bundle over $\ov{\M}_{1,1}$ assigns the zero abelian differential to precisely one elliptic curve in the elliptic pencil. 

The intersection in \eqref{eq:Sint3} is contributed by the element of $S$ where the elliptic tail has the zero abelian differential, and the marked point in the genus $g-1$ component collides with the node, creating a rational bridge.

The remaining intersections can be computed using the description of the boundaries of these loci provided by the incidence variety compactification. 

The intersection in \eqref{eq:Sint2} is contributed by the elements of $S$ where the elliptic tail has the zero abelian differential, and  the marked point in the genus $g-1$ component coincides with one of the zeros of the  differential (such zeros are generically away from the node). For each such element of $S$, there exists a twisted abelian differential whose aspect on the elliptic tail has a zero at the marked point and a pole of order $2$ at the node. The intersection is transverse along all such contributions.

The intersection in \eqref{eq:Sint1} consists of the sum of the contributions to \eqref{eq:Sint2} and \eqref{eq:Sint3}. Indeed, the contribution to \eqref{eq:Sint3} is clearly in $\pi_1^{-1}(\ov{\H}_{g, 1}^1)$, because the second marked point is in the elliptic tail with a zero differential, so that a suitable twisted differential can be found as above, and it is also seen to be in $\pi_2^{-1}(\ov{\H}_{g, 1}^1)$ by taking a twisted differential whose aspect on the rational bridge has poles of order $2$ at both nodes  and simple zeros at the first marked point and at some other smooth point. Such a twisted differential satisfies the global residue condition of the incidence variety compactification \cite[Def.~1.2(4)]{bcggm1}. The intersection is transverse along all such contributions.

We remark that the contribution to \eqref{eq:Sint3} is not in $\ov{\H}_{g, 2}^1$:  a twisted differential vanishing at the second marked point has order $-2$ at the preimage of the node in the elliptic component, and generically order $0$ at the preimage of the node in the genus $g-1$ component. It follows that the aspect on the rational bridge has orders $0$ and $-2$ at the nodes, hence does not vanishes at the first marked point.

\smallskip

To find the coefficient $c_i$, for $1 \leq i \leq \lfloor g/2 \rfloor$, consider the test surface $S_i$ in $\P\E^1_{g,2}$ obtained as follows: select a general element of the boundary divisor $\Delta_{i: \{ 1\}}$ in $\P\E^1_{g,2}$ consisting of curves with a genus $i$ component containing precisely the first marked point, and vary the two marked points in their corresponding components. In particular, elements of $S_i$ have a fixed general stable differential $\mu$. When a marked point collides with the node, it gives rise to a rational bridge, and $\mu$ is extended by zero on such rational bridges.
We show the following intersections below:
\begin{align}
\label{eq:Siint1}
&S_i \cdot \pi_1^*\left[ \ov{\H}_{g, 1}^1\right] \cdot \pi_2^*\left[ \ov{\H}_{g, 1}^1 \right] &&= (2i-1)(2(g-i) -1), \\
\label{eq:Siint2}
&S_i \cdot \ov{\H}_{g, 2}^1&&= (2i-1)(2(g-i) -1)-1, \\
\label{eq:Siint3}
&S_i \cdot \mbox{
\begin{tikzpicture}[baseline={([yshift=-.5ex]current bounding box.center)}]
      \path(0,0) ellipse (2 and 2);
      \tikzstyle{level 1}=[counterclockwise from=-60,level distance=15mm,sibling angle=120]
      \node [draw,circle,inner sep=2.5] (A0) at (120:1.5) {$\scriptstyle{i}$};
      \tikzstyle{level 1}=[counterclockwise from=-45,level distance=12mm,sibling angle=90]
      \node [draw,circle,fill] (A1) at (0:1.5) {$\scriptstyle{}$}
      child {node [label=-45: {$\scriptstyle{2}$}]{}}
      child {node [label=45: {$\scriptstyle{1}$}]{}};
      \tikzstyle{level 1}=[counterclockwise from=-60,level distance=15mm,sibling angle=60]
      \node [draw,circle,inner sep=2.5] (A2) at (240:1.5) {$\scriptstyle{g-i}$};
      \path (A0) edge [] (A1);
      \path (A1) edge [] (A2);
    \end{tikzpicture}
}  &&= 1.
\end{align}
Moreover, $S_i$ has empty intersection with the remaining classes in \eqref{eq:aAbBciCi}. It follows that $c_i =1$, for all $1 \leq i \leq \lfloor g/2 \rfloor$.

The intersection \eqref{eq:Siint3} is contributed by the element of $S_i$ where both marked points collide with the node, creating two rational bridges.

The intersection \eqref{eq:Siint2} is contributed by those elements of $S_i$ where either  both marked points  coincide with zeros of the differential, or where exactly one of the marked points collides with the node and the other marked point is a zero of the differential. In the cases when exactly one of the marked points collides with the node, creating a rational bridge, there exists a twisted differential of type $\mu$ whose aspect on the rational bridge  has zeros at the marked point and at some other smooth point, poles of order $2$ at the nodes, and satisfies the global residue condition \cite[Def.~1.2(4)]{bcggm1}. The intersection is transverse along all such contributions.

The intersection \eqref{eq:Siint1} is the sum of the contribution to \eqref{eq:Siint2}  and \eqref{eq:Siint3}. 
In fact, the contribution to \eqref{eq:Siint3} is in \eqref{eq:Siint1} as well, since each marked point is the zero of some  twisted differential of type $\mu$. 

However, the contribution to \eqref{eq:Siint3} is not in \eqref{eq:Siint2}, since there exist no twisted differential of type $\mu$ vanishing  at \textit{both} marked points. For this, 
assume there exists a twisted differential $\tau$ of type $\mu$ whose aspects $\tau_1$ and $\tau_2$ on the two rational bridges have zeros at the two marked points. 
Since $\mu$ is holomorphic and non-zero at the preimages of the nodes in the components of genus $i$ and $g-i$, it follows that $\tau_1$ and $\tau_2$ have a pole of order $2$ at the nodes $Q_1$ and $Q_2$ where the rational bridges meet the genus $i$ and $g-i$ components, respectively.
Then, since $\tau_1$ and $\tau_2$ vanish at the marked points, for degree reasons one has that $\tau_1$ and $\tau_2$ must each have a pole of order  at least $1$ at the node $Q_0$ where the two rational bridges meet. As the sum of the order of the poles at $Q_0$ has to be equal to $2$, one  has that $\tau_1$ and $\tau_2$ must each have a simple pole  at $Q_0$. The global residue condition \cite[Def.~1.2(4)]{bcggm1} requires that $\mathrm{Res}_{Q_1} (\tau_1) = \mathrm{Res}_{Q_2}(\tau_2) = 0$. By the residue theorem, this   forces the residues at $Q_0$ to be zero, a contradiction. 
 \end{proof}
 
 Theorem \ref{thm:classZ2}  follows from \eqref{eq:aAbBciCi}, \eqref{eq:aAbBciCik2},  Proposition \ref{prop:pi1Z1pi2Z1}, and Lemma \ref{lem:Z1class}.


\section{A recursive identity for incidence loci} 
\label{sec:recidinc}

Theorem \ref{thm:pinZnrhonZ1plus}  provides a recursive identity for the class of the incidence locus $\ov{\H}_{g, n}^k$ over curves with rational tails. 
For $k=1$, Theorem \ref{thm:pinZnrhonZ1plus} and Remark \ref{rmk:1} are specializations of  results by Sauvaget \cite{sauvagetcohomology} valid over $\ov{\M}_{g,n}$. The techniques of \cite{sauvaget2020volumes} could  be used to extend the argument for \mbox{$k\geq 1$.} 
Here we provide an alternative proof   by applying the incidence variety compactification from \cite{bcggm1, bcggm}.

We start by proving the next Theorem \ref{thm:pinZnrhonZ1}. Recall the loci $E_{\bm{I}}$ and the natural maps
 $\pi_n\colon \P\E^k_{g,n} \rightarrow \P\E^k_{g,n-1}$  and  $\rho_n\colon \P\E^k_{g,n} \rightarrow\P\E^k_{g,1}$ from \S\ref{sec:recidincinit}.
 
\begin{theorem}[{\cite{sauvagetcohomology, sauvaget2020volumes}}]
\label{thm:pinZnrhonZ1}
For $g\geq 2$ and $n\geq 2$, the identity
\[
\pi_n^* \left[\ov{\H}_{g, n-1}^k \right] \cdot \rho_n^* \left[\ov{\H}_{g, 1}^k \right] = \left[\ov{\H}_{g, n}^k \right] + \sum_{\bm{I}} |\bm{I}|[E_{\bm{I}}]
\] 
where the sum is over all non-empty $\bm{I}\subseteq \{1,\dots, n-1\}$, holds in:
\begin{enumerate}[(i)]
\item $A^{n}\left( \P\E^k_{g,n}\big|_{\M_{g,n}^\mathsf{rt}}  \right)$ when
\begin{equation}
\label{eq:hypkgn}
\left\{
\begin{array}{l}
\mbox{$k=1$; or}\\[.1cm]
\mbox{$k=2$ and $n \leq 2(2g-2)-2$; or}\\[.1cm]
\mbox{$k\geq 3$ and $n \leq k(2g-2)-1$;}
\end{array}
\right.
\end{equation}
\item $A^{n}\left( \P\E^k_{g,n} \right)$ when $n\leq k$.
\end{enumerate}
\end{theorem}

\subsection{On the hypothesis \eqref{eq:hypkgn}}
\smallskip

We start with the following Lemma, which explains how we use the hypothesis \eqref{eq:hypkgn}:

\begin{lemma}
\label{lemma:hypkgn}
Let $k\geq 2$,  let $g$ and $n$ be such that \eqref{eq:hypkgn} holds, and $m\leq n-1$.
The locus of $k$-differentials that are $k$-th powers of abelian differentials has 
positive codimension inside the closure of ${\H}^k_{g,\bm{m}} \subset \mathbb{PE}^k_{g,n-m}$ from \eqref{eq:H1heavypt}. 
\end{lemma}

Here are some examples showing how the lemma fails when \eqref{eq:hypkgn} fails. 
When $k=2$ and $n=4g-5$,  consider the locus 
\begin{equation}
\label{eq:counterexk2}
\ov{\H}^2_{g,\bm{m}} \subset \P\E^2_{g,1} \qquad \mbox{where} \qquad \bm{m}=\left(4g-6\right). 
\end{equation}
This locus has two components:
a component parametrizing quadratic differentials generically with  a zero of order $4g-6$ at the marked point and simple zeros at two distinct unmarked points; and a component parametrizing  \textit{squares of abelian differentials}  generically with a zero of order $4g-6$ at the marked point and a double zero at an unmarked point.
 These two components have equal dimension by \cite[Theorem 1.1]{bcggm} or \cite{schmitttwisted}. 

Similarly, when $k \geq 3$ and $n=k(2g-2)$, consider the locus 
\begin{equation}
\label{eq:counterexk3}
\ov{\H}^k_{g,\bm{m}} \subset \P\E^k_{g,1} \qquad \mbox{where} \qquad\bm{m}= \left(k(2g-2)-1 \right).
\end{equation}
This locus has two equidimensional  components: a component pa\-ram\-e\-trizing $k$-differentials with  a zero of order $k(2g-2)-1$ at the marked point and a simple zero elsewhere; and a component parametrizing \textit{$k$-th powers of abelian differentials} vanishing with order $k(2g-2)$ at the marked point.

\begin{proof}[Proof of Lemma \ref{lemma:hypkgn}]
For $k\geq 2$ and $m\leq n-1$, consider a locus $\ov{\H}^k_{g,\bm{m}}$ in $\mathbb{PE}^k_{g,n-m}$ as in \eqref{eq:H1heavypt}, and write
$\ov{\H}^k_{g,\bm{m}}=\mathbb{A} \cup \mathbb{B}$,
where $\mathbb{A}$ (respectively, $\mathbb{B}$) is the closure of the locus of differentials in $\ov{\H}^k_{g,\bm{m}}$ that \textit{are not} $k$-th powers of abelian differentials,
(resp., \textit{are} $k$-th powers of abelian differentials).

Differentials in $\mathbb{A}$ have generically a zero of order $m$ and $k(2g-2)-m$ simple zeros. 
Differentials in $\mathbb{B}$ have generically a zero of order $k\lceil\frac{m}{k}\rceil$ (i.e., the smallest multiple of $k$ which is greater than or equal to $m$) and $(2g-2)-\lceil\frac{m}{k}\rceil$ zeros of order $k$. 
From \cite[Theorem 1.1]{bcggm}, one has 
\begin{align*}
\dim \mathbb{A} & =2(k+1)(g-1)-m, & \dim \mathbb{B} & = 4 (g-1) - \left\lceil\frac{m}{k}\right\rceil + 1.
\end{align*}
It follows that for values of $m$ small enough, $\mathbb{B}$ has positive codimension in $\ov{\H}^k_{g,\bm{m}}$, and  the loci $\mathbb{A}$ and $\mathbb{B}$ have equal dimension for values of $m$ such that
\begin{equation}
\label{eq:ABequaldim}
(k-1)2(g-1) = m - \left\lceil\frac{m}{k}\right\rceil + 1.
\end{equation}
A case study shows that \eqref{eq:ABequaldim} fails when \eqref{eq:hypkgn} holds. Indeed, if $k\mid m$, then \eqref{eq:ABequaldim} only holds for $k=2$ and $m=4g-6$. This implies $n\geq 4g-5$, e.g., the case given by \eqref{eq:counterexk2}. Otherwise, if $k \nmid m$, then $m\equiv k-1$ mod $k$, hence $m=k(2g-2)-1$. This is the case given by \eqref{eq:counterexk3}. The statement follows.
\end{proof}

\subsection{A set-theoretic study}
To prepare for the proof of Theorem \ref{thm:pinZnrhonZ1}, we study the components of the intersection 
\begin{equation}
\label{eq:pinrhon}
\pi_n^{-1} \left(\ov{\H}_{g, n-1}^k \right) \cap \rho_n^{-1} \left(\ov{\H}_{g, 1}^k \right) \quad\subset \P\E^k_{g,n}.
\end{equation}
By definition, ${\H}_{g, n}^k$ is the only component of \eqref{eq:pinrhon} over $\M_{g,n}$.

\begin{lemma} 
\label{lemma:EIinsupportEn} 
Let $k$, $g$, and $n$ be such that \eqref{eq:hypkgn} holds.
The loci $E_{\bm{I}}\subset \P\E^k_{g,n}$ for all non-empty $\bm{I}\subseteq \{1,\dots, n-1\}$ are contained in the intersection 
\eqref{eq:pinrhon}, but not in $\ov{\H}_{g, n}^k$.
\end{lemma}

\begin{proof}
The argument is similar to the one used for Proposition \ref{prop:pi1Z1pi2Z1loci}.
For a non-empty $\bm{I}\subseteq \{1,\dots, n-1\}$, consider a general element $(C, P_1, \dots, P_n,\mu)$ in $E_{\bm{I}}$. By definition, the curve $C$ has a rational tail containing the marked points with labels in $\bm{I}\sqcup \{n\}$, and  the stable $k$-differential $\mu$ has a zero of order $|\bm{I}|$ at the preimage of the node in the genus $g$ component of $C$ and zeros at the marked points in the genus $g$ component.

To see that $E_{\bm{I}}$ is in $\pi_n^{-1}\left(\ov{\H}_{g, n-1}^k \right)$, we need to construct a twisted $k$-differential of type $\mu$ satisfying the conditions of the incidence variety compactification and vanishing at the first $n-1$ marked points. For this, consider the twisted $k$-differential whose aspect on the genus $g$ component will be the restriction of $\mu$ on that component, while the aspect on the rational component will have a pole of order $2k + |\bm{I}|$ at the node and zeros at the marked points with labels in $\bm{I}$. When $k=1$, the global residue condition \cite[Def.~1.2(4)]{bcggm1} is satisfied since the residue of the twisted abelian differential at the preimage of the node in the rational component is zero. When $k \geq 2$, using Lemma \ref{lemma:hypkgn}, the hypothesis \eqref{eq:hypkgn} implies that the restriction of $\mu$ to the genus $g$ component is generically not the $k$-th power of an abelian differential.
It follows that the global $k$-residue condition is automatically satisfied \cite[Def.~1.4(4)(ii)]{bcggm}. 

In order for $E_{\bm{I}}$ to be in $\rho_n^{-1} \left(\ov{\H}_{g, 1}^k \right)$, we need to construct a twisted $k$-differential of type $\mu$ satisfying the conditions of the incidence variety compactification and vanishing at the last marked point. This is as in the previous paragraph, just simpler, and is made possible by the fact that $\mu$ vanishes with order $|\bm{I}| \geq 1$ at the preimage of the node in the genus $g$ component.  

There is no single twisted $k$-differential of type $\mu$ vanishing at \textit{all} marked points, hence $E_{\bm{I}}$ is not in $\ov{\H}_{g, n}^k$. The statement follows.
\end{proof}

Next, we prove the converse of Lemma \ref{lemma:EIinsupportEn} over the locus of curves with rational tails. We use an inductive argument by means of the following maps:
For $1\leq i\leq n-1$, the map $\pi_i\colon \P\E^k_{g,n} \rightarrow \P\E^k_{g,n-1}$ is obtained by forgetting the marked point $P_i$, and relabeling the marked points $P_j$ for $j>i$ as $P_{j-1}$.

\begin{proposition} \label{prop:EnonlyEI} 
Let $k$, $g$, and $n$ be such that \eqref{eq:hypkgn} holds.
The locus $\ov{\H}_{g, n}^k$ and the extra loci $E_{\bm{I}}\subset \P\E^k_{g,n}$ for all non-empty  $\bm{I}\subseteq \{1,\dots, n-1\}$ are the \textit{only} components of \eqref{eq:pinrhon} over $\M_{g,n}^\mathsf{rt}$. 
\end{proposition}

\begin{proof}
We proceed by induction on $n$. The base case $n=2$ is treated by Proposition \ref{prop:pi1Z1pi2Z1loci}.
For $n\geq 3$, let $B$ be an irreducible component of \eqref{eq:pinrhon} over $\M_{g,n}^\mathsf{rt}$.
For $1\leq i\leq n-1$, the image of \eqref{eq:pinrhon} under $\pi_i$ is 
\begin{align}
\begin{split}
\label{eq:piiint}
\pi_{i} \left(\pi_n^{-1} \left(\ov{\H}_{g, n-1}^k \right) \cap \rho_n^{-1} \left(\ov{\H}_{g, 1}^k \right) \right) 
&= \pi_{i} \left(\pi_n^{-1}\left(\ov{\H}_{g, n-1}^k\right) \right) \cap \rho_{n-1}^{-1} \left(\ov{\H}_{g, 1}^k\right) \\
&=  \pi_{n-1}^{-1} \left(\ov{\H}_{g, n-2}^k \right) \cap \rho_{n-1}^{-1} \left(\ov{\H}_{g, 1}^k \right).
\end{split}
\end{align}
By the induction hypothesis, it follows that $\pi_i\,(B)$ is either in $\ov{\H}_{g, n-1}^k$ or in one of the loci
\begin{equation}
\label{eq:Lnminus1}
E_{\ov{\bm{I}}} \,\subset \P\E^k_{g,n-1} \quad \mbox{for } \varnothing\neq \ov{\bm{I}} \subseteq \{1, \dots, n-2\}.
\end{equation}
The rational tail in a general element of one of such $E_{\ov{\bm{I}}}$ contains the marked points with labels in $\ov{\bm{I}}\sqcup\{n-1\}$. 

Consider a general element $(C, P_1, \dots, P_n,\mu)$ in $B$, and assume  $B\not\subseteq \ov{\H}_{g, n}^k$. This implies that there is no twisted $k$-differential of type $\mu$ vanishing at all marked points. We show below that $B\subseteq E_{\bm{I}}$ for some $\bm{I}$.

If $\pi_i\,(B)\subseteq \ov{\H}_{g, n-1}^k$, then $C$ has a twisted $k$-differential of type $\mu$ vanishing at all points $P_j$ with $j\neq i$. This implies that $\mu$ vanishes at those marked points and preimages of nodes that are in the genus $g$ component of $C$; and moreover, for each preimage of a node in the genus $g$ component, the order of vanishing  is at least equal to the number of marked points $P_j$ with $j\neq i$ contained in the maximal rational subcurve of $C$ intersecting the genus $g$ component at that node. 

For $B$ to be in \eqref{eq:pinrhon} while $B\not\subseteq \ov{\H}_{g, n}^k$, the point $P_i$ needs to replace $P_n$ as a zero of the twisted $k$-differential in the following sense. Assume that $P_n$ is in the genus $g$ component of $\pi_i(C)$. Then the only possibility for $P_i$ in $C$ is to be in a rational tail containing only the marked points $P_i$ and~$P_n$. In this case, $B\subseteq E_{\bm{I}}$ with $\bm{I}=\{i\}$. Otherwise, assume that $P_n$ is in a rational component of $\pi_i(C)$, and let $R$ be the maximal rational subcurve of $C$ containing $P_n$. Then $P_i$ is necessarily in $R$. In this case, $B\subseteq E_{\bm{I}}$ with $\bm{I}\sqcup \{n\}$ equal to the set of markings in $R$.

Next, assume  $\pi_i\,(B)\subseteq E_{\ov{\bm{I}}}$ for some $\ov{\bm{I}}$ as in \eqref{eq:Lnminus1}. Then, the point $P_n$ is in a rational component of $C$. 
Let $R$ be the maximal rational subcurve of $C$ containing $P_n$. In particular, $\ov{\bm{I}}\sqcup \{n-1\}$ is the set of markings in $\pi_i(R)$.

We argue that the point $P_i$ is not in $R$. Let $Q$ be the node of $C$ where the genus $g$ component intersects $R$.
Since $\pi_i\,(B)\subseteq E_{\ov{\bm{I}}}$ while $B\not\subseteq \ov{\H}_{g, n}^k$, the stable $k$-differential $\mu$ vanishes at the preimage of $Q$ in the genus $g$ component with order \textit{precisely} equal to $\left|\ov{\bm{I}}\right|$. It follows that if $P_i$ is in $R$, then there is no twisted $k$-differential of type $\mu$ vanishing at all the first $n-1$ marked points, a contradiction to $B$ being in \eqref{eq:pinrhon}. 

Hence the point $P_i$ is in the complement of $R$ in $C$. We deduce that $B\subseteq E_{\bm{I}}$ where $\bm{I}$ is obtained from $\ov{\bm{I}}$ after increasing by one all  $j\in \ov{\bm{I}}$ such that $j\geq i$ (this shift is due to the relabeling of points given by the map $\pi_i$).

The statement follows.
\end{proof}

Next, we show that when $n \leq k$, Proposition \ref{prop:EnonlyEI} extends over $\ov{\M}_{g,n}$:

\begin{proposition} \label{prop:EnonlyRT} 
When $n \leq k$, the intersection \eqref{eq:pinrhon}  is  contained \textit{only} over the locus of curves with rational tails. 
\end{proposition}

\begin{proof}
We again proceed by induction on $n$. The base case $n=2$ is treated by Proposition \ref{prop:pi1Z1pi2Z1loci}. For $n \geq 3$, let $B$ be an irreducible component of \eqref{eq:pinrhon} which does not lie over $\M_{g,n}^\mathsf{rt}$. We will show that each such $B$ has positive codimension in either $\ov{\H}_{g, n}^k$ or one of the extra loci $E_{\bm{I}}\subset \P\E^k_{g,n}$. 

Let $1\leq i \leq n-1$. Using  \eqref{eq:piiint} as in the proof of Proposition \ref{prop:EnonlyEI}, the induction hypothesis implies that $\pi_i\,(B)$ is in either $\ov{\H}_{g, n-1}^k$ or  one of the loci $E_{\ov{\bm{I}}}\subset \P\E^k_{g,n-1}$ for some $\ov{\bm{I}}$ as in \eqref{eq:Lnminus1}.
Additionally, since $B$ does not lie over $\M_{g,n}^\mathsf{rt}$, one has that $\pi_i(B)$ is strictly contained in either $\ov{\H}_{g, n-1}^k$ or one of the $E_{\ov{\bm{I}}}$.
Let $(C, P_1, \dots, P_n,\mu)$ be a general element in $B$. 

Consider the case when $\pi_i(B) \subset \ov{\H}_{g, n-1}^k$. 
We can assume that $C$ has no rational tails, otherwise we can replace $B$ with the higher-dimensional locus whose general element is obtained by smoothing all rational tails in $C$.

First, suppose that $C$ is a curve with (at least) one disconnecting node, a genus $a$ subcurve, and a genus $g-a$ subcurve, where $1\leq a \leq g-1$. If $\mu$ is generically nonzero on each component and vanishes at the $n$ marked points, then $B$ has codimension at least $n + 1$ and is  contained in~$\ov{\H}_{g, n}^k$.  
Consider instead the case where $\mu$ is identically zero on the genus $a$ subcurve, and vanishes at all marked points on the genus $g-a$ subcurve. The codimension of this locus is at least
$N := h^0\left(C_a, \omega^{\otimes k}_{C_a}(kQ)\right) + 1 = k(2a-1) - a + 2,$
where $C_a$ is the genus $a$ subcurve of $C$, and $Q$ is the node connecting $C_a$ to the rest of $C$. The assumption $n \leq k$ implies 
$n < n(2i-1) - i + 2 \leq N.$
Thus, $B$ has codimension higher than $n$. Furthermore, since $\pi_i(B) \subset \ov{\H}_{g, n-1}^k$,  there exists a twisted $k$-differential of type $\mu$ vanishing at $P_j$ with $j\neq i$. If $P_i$ is on the genus $g-a$ subcurve, it must also be a zero of $\mu$ since $B \subset \pi_n^{-1} \left(\ov{\H}_{g, n-1}^k \right)$. If $P_i$ is on the genus $a$ subcurve, then since $n < N$, we may find a potentially different twisted $k$-differential of type $\mu$ vanishing at all marked points on $C_a$ including $P_i$. 
In all cases, one concludes that $B$ is in $\ov{\H}_{g, n}^k$.
The case where $C$ has (at least) one non-disconnecting node and $\mu$ is generically nonzero and vanishes at all marked points is similar. 

Now suppose that $C$ has a rational bridge, and $\mu$ is nonzero on the nonrational components. 
In order for this locus to be in \eqref{eq:pinrhon}, $\mu$ must be identically zero on the rational bridge, and we must have poles of order at most $k-1$ at the nodes on the nonrational components meeting the bridge. These conditions allow the marked points on the rational bridge to be zeros of a twisted $k$-differential of type $\mu$. Moreover, a pole of order $k$ on a nonrational component implies that the vertices corresponding to that nonrational component and to the rational bridge are on the same level in the ordered dual graph of $C$, a contradiction to the assumption that $\mu$ is nonzero on nonrational components. 
In this case, $B$ has codimension at least $n+1$ and is contained in $\ov{\H}_{g, n}^k$.
The case where $C$ has a rational component meeting the rest of the curve in two non-disconnecting nodes is similar.

Finally, assume that $\pi_i(B) \subset E_{\ov{\bm{I}}}=\pi_i(E_{\bm{I}})$ for some $\bm{I}$. This requires that $C$ have a rational tail containing $P_n$. 
Let $m+1$ be the number of marked points on the maximal rational subcurve $R$ of $C$ containing $P_n$.
In order for $B$ to be in \eqref{eq:pinrhon}, $\mu$ needs to vanish with order at least $m$ at the nodal point in $C\setminus R$ meeting $R$. This implies that $B$  is contained in $E_{\bm{I}}$.
\end{proof}

\subsection{Auxiliary computations}
Theorem \ref{thm:pinZnrhonZ1} will follow from 
Propositions \ref{prop:EnonlyEI} and \ref{prop:EnonlyRT} after determining the multiplicity of each component in  \eqref{eq:pinrhon}.
For this, we start with some preliminary computations.
Let 
\begin{equation}
\label{eq:Ln}
\mathsf{L}_n := \pi_n^* \left[\ov{\H}_{g, n-1}^k \right] \cdot \rho_n^* \left[\ov{\H}_{g, 1}^k \right] \quad \in A^{n}\left(\P\E^k_{g,n} \right).
\end{equation}

\begin{lemma} \label{lemma:pushforwardEn} 
Let $k$, $g$, and $n$ be such that \eqref{eq:hypkgn} holds.
For $1\leq i \leq n-1$, one has
\begin{align} 
\label{eq:piiH}
&(\pi_{i})_*\left[\ov{\H}_{g, n}^k \right] &&= \left( k(2g-2)-(n-1) \right) \left[\ov{\H}_{g, n-1}^k \right],\\[0.2cm]
\label{eq:piiL}
&(\pi_{i})_*\left[\mathsf{L}_n\right] &&= \left( k(2g-2)-(n-2) \right) \mathsf{L}_{n-1}, \\[0.2cm]
\label{eq:piiE}
&(\pi_{i})_*\left[ E_{\bm{I}}\right] &&= \left\{
\begin{array}{ll}
\left( k(2g-2)-(n-2) \right) E_{\ov{\bm{I}}} & \mbox{if $i\not\in \bm{I}$,}\\[0.2cm]
\left[\ov{\H}_{g, n-1}^k \right] & \mbox{if $\bm{I}=\{i\}$,}\\[0.2cm]
0 & \mbox{if $i\in \bm{I}$ and $|\bm{I}|>1$}
\end{array}
\right.
\end{align} 
in $A^{n-1}\left( \P\E^k_{g,n-1}\right)$. In \eqref{eq:piiE}, when $i\not\in \bm{I}$, the set $\ov{\bm{I}}$ is obtained from $\bm{I}$ after decreasing by one all $j\in \bm{I}$ with $j> i$. 
\end{lemma} 

\begin{proof}
For \eqref{eq:piiH}, the restriction of $\pi_i$ to ${\H}_{g, n}^k$ has degree $k(2g-2)-(n-1)$ over ${\H}_{g, n-1}^k$, equal to the number of unmarked zeros of the $k$-differential $\mu$ for a general $(C, P_1, \dots, P_{n-1}, \mu)$ in ${\H}_{g, n-1}^k$.

For \eqref{eq:piiL}, since one may forget the marked points in any order, one has  $\rho_n = \rho_{n-1} \circ \pi_i$. Using the projection formula, one computes
\begin{align*}
(\pi_{i})_*\left[\mathsf{L}_n\right] &= (\pi_{i})_* \left(\pi_n^* \left[\ov{\H}_{g, n-1}^k \right] \cdot \rho_n^* \left[\ov{\H}_{g, 1}^k \right] \right) \\
&= (\pi_{i})_* \left(\pi_n^*\left[\ov{\H}_{g, n-1}^k\right] \right) \cdot \rho_{n-1}^* \left[\ov{\H}_{g, 1}^k\right] \\
&= c\, \pi_{n-1}^* \left[\ov{\H}_{g, n-2}^k \right] \cdot \rho_{n-1}^* \left[\ov{\H}_{g, 1}^k \right] \\
&= c\, \mathsf{L}_{n-1}
\end{align*}
where $c := k(2g-2) - (n-2)$.
The third equality follows from $(\pi_{i})_* \, \pi_n^* = \pi_{n-1}^* \, (\pi_{i})_*$ 
and the first part of the statement.

Finally, for \eqref{eq:piiE}, when $i\not\in \bm{I}$, the restriction of $\pi_i$ to $E_{\bm{I}}$ has generically degree $k(2g-2)-(n-2)$ over $E_{\ov{\bm{I}}}$, equal to the number of  zeros at unmarked smooth points of the $k$-differential in a general element of $E_{\ov{\bm{I}}}$. Indeed, from Lemma \ref{lemma:hypkgn}, a $k$-differential in a general element of $E_{\ov{\bm{I}}}$  is not the $k$-th power of an abelian differential when $k\geq 2$, hence has a zero of order $|\ov{\bm{I}}|$ at the preimage of the node in the genus $g$ component of $C$,  simple zeros at the $n-2-|\ov{\bm{I}}|$ marked points in the genus $g$ component of $C$, and simple zeros elsewhere for all $k\geq 1$.

When $\bm{I}=\{i\}$, the restriction of $\pi_i$ to $E_{\bm{I}}$ has generically degree one over ${\H}_{g, n-1}^k$. In the remaining case when $i\in \bm{I}$ and $|\bm{I}|>1$, the restriction of $\pi_i$ to $E_{\bm{I}}$ has generically one-dimensional fibers. The statement follows.
\end{proof}

Also, we will use intersections with the following test space.
For $n\geq 3$, let $C$ be a general smooth genus $g$ curve, and identify a general point $Q$ in $C$ with a general point in a rational curve $R$ containing the marked points $P_1,\dots, P_{n-1}$.
Consider the $n$-dimensional test space $T\subset \P\E^k_{g,n}$ consisting of elements $(C\cup_Q R, P_1, \dots, P_n, \mu)$ obtained by 
varying the  $k$-differential $\mu$  in a general $\P^{n-1}\subseteq \P H^0(C, \omega^{k}_C)$, and the point $P_n$  along $C$. One has $T\cong  \P^{n-1}\times C$.

\begin{lemma}
\label{lemma:claimtestndimspace}
Let $k$, $g$, and $n$ be such that \eqref{eq:hypkgn} holds.
One has
\begin{align}
\label{eq:Tint1}
&T \cdot \left[\ov{\H}_{g, n}^k \right] && = k(2g-2)-(n-1), \\[0.2cm]
\label{eq:Tint2}
&T \cdot \mathsf{L}_n &&= k(2g-2),\\[0.2cm]
\label{eq:Tint3}
&T \cdot \left[ E_{\bm{I}} \right] &&= \left\{
\begin{array}{ll}
1 & \mbox{if $\bm{I}=\{1, \dots, n-1\}$,}\\[0.2cm]
0 & \mbox{otherwise.}
\end{array}
\right.
\end{align}
\end{lemma}

\begin{proof}
The intersection \eqref{eq:Tint1} is contributed by those elements $(\mu, P_n)$ of $T\cong  \P^{n-1}\times C$ for which there exists a twisted $k$-differential of type $\mu$ vanishing at all marked points. The only possibility is that $\mu$ is the $k$-differential $\mu_0$ which vanishes with order $n-1$ at $Q$ (such $\mu_0$ is generically unique, and after Lemma \ref{lemma:hypkgn}, $\mu_0$ is generically not the $k$-th power of an abelian differential when $k\geq 2$), and $P_n$ coincides with one of the remaining $k(2g-2)-(n-1)$ zeros of $\mu_0$ in $C$. The intersection is transverse along all such elements.

The intersection \eqref{eq:Tint2} is contributed by those elements $(\mu, P_n)$ of $T$ for which there exists a twisted $k$-differential of type $\mu$ vanishing at the first $n-1$ marked points, and a possibly distinct twisted $k$-differential of type $\mu$ vanishing at $P_n$. Hence, in addition to the contributions to \eqref{eq:Tint1}, also the element $(\mu=\mu_0, P_n=Q)$ of $T$ contributes to \eqref{eq:Tint2}. 
After Lemma \ref{lemma:hypkgn}, $\mu_0$ is generically not the $k$-th power of an abelian differential when $k\geq 2$.
It follows that $\mu_0$ vanishes generically with order $n-1$ at $Q$ for all $k\geq 1$, hence the element  $(\mu_0, Q)$ contributes to \eqref{eq:Tint2} with multiplicity $n-1$.
 
Finally, for \eqref{eq:Tint3}, $T$  intersects a locus $E_{\bm{I}}$ only if $P_n$ is on a rational component. This only happens for $P_n = Q$, in which case $T$ intersects  $E_{\bm{I}}$ with $\bm{I}=\{1, \dots, n-1\}$ transversally at $(\mu=\mu_0, P_n=Q)$. 
\end{proof}

In the final step of the proof of Theorem \ref{thm:pinZnrhonZ1}, we will use the following:

\begin{lemma} \label{lemma:ZnandEIindependent}
Let $k$, $g$, and $n$ be such that \eqref{eq:hypkgn} holds.
The class of $\ov{\H}_{g, n}^k$ and the classes of the loci $E_{\bm{I}}$ for all $\bm{I}$ are independent in $A^{n}\left(\P\E^k_{g,n} \right)$. 
\end{lemma}

\begin{proof}
We  proceed by induction on $n$.
When $n=2$, there is only one locus $E_{\bm{I}}$, namely $E_{\bm{I}}$ with $\bm{I}=\{1\}$. 
To show independence, we restrict the classes of $\ov{\H}_{g, 2}^k$ and $E_{\{1\}}$ to a test surface. 
For this, consider the test surface $S_i\subset \P\E^k_{g,n}$ with $i=1$ appeared for $k=1$ in the proof of Proposition \ref{prop:pi1Z1pi2Z1}. 
Namely,  select a general element of the boundary divisor $\Delta_{1: \{ 1\}}$ in $\P\E^k_{g,2}$ consisting of curves with an elliptic tail containing only the first marked point.  The elements of the surface $S_1 \subset \P\E^k_{g,n}$ are obtained by varying the two marked points in their corresponding components. In particular, elements of $S_1$ have a fixed general stable differential $\mu$. When a marked point collides with the node, it gives rise to a rational bridge. For $k=1$, $\mu$~is extended by zero on such rational bridges, while for $k\geq 2$, the restriction of $\mu$ to such rational bridges has poles of order $k$ at the two nodes on each rational bridge.

The surface $S_1$ has empty intersection with $E_{\{1\}}$. For $k=1$, one has $S_1 \cdot \left[\ov{\H}_{g, 2}^1\right] = 2g-4$ from \eqref{eq:Siint2}.
For $k\geq 2$, one has
\begin{align*}
S_i \cdot \left[\ov{\H}_{g, 2}^k\right] &= \left(k(2i-2)+k\right)\left(k(2(g-i)-2)+k\right). 
\end{align*}
This intersection  is contributed by those elements of $S_i$ where  both marked points  coincide with zeros of the differential $\mu$, and the intersection is transverse along all such elements.
Contrary to the case $k=1$, there is no contribution from the elements where one marked point  collides with the node, creating a rational bridge. Indeed, when $k\geq 2$, the restriction of $\mu$ to a rational bridge has poles of order $k$ at both nodes and does not have any zeros for degree reasons.
 
It follows that $S_1$ has nonzero intersection with $\ov{\H}_{g, 2}^k$, unless $k=1$ and $g=2$.
We conclude that the classes of $\ov{\H}_{g, 2}^k$ and $E_{\{1\}}$ are independent, unless $k=1$ and $g=2$.

In case $k=1$ and $g=2$, we arrive at a similar statement by considering the  push-forward of $\ov{\H}_{2, 2}^1$ and $E_{\{1\}}$ via the forgetful map $\varphi\colon \P\E^1_{2,2} \rightarrow \ov{\M}_{2,2}$. One computes
\begin{align}
\label{eq:phipfwdH122}
\varphi_* \left[ \ov{\H}_{2, 2}^1 \right] &= \omega_1 + \omega_2 - \lambda_1 - \delta_{0:2} - \delta_{1:0},\\
\label{eq:phipfwdE1}
\varphi_* \left[ E_{\{1\}} \right] &= \delta_{0:2}
\end{align}
in $\mathrm{Pic}\left( \ov{\M}_{2,2}\right)$. 
Here $\lambda_1$ is the first Chern class of the Hodge bundle; $\delta_{0:2}$ is the class of the divisor $\Delta_{0:2}$ of curves with a rational tail; and $\delta_{1:0}$ is the class of the divisor $\Delta_{1:0}$ of curves with an unmarked elliptic tail.
The identity \eqref{eq:phipfwdH122} can be shown either by using Theorem \ref{thm:classZ2} for the class of $\ov{\H}_{2, 2}^1$, or by observing that the restriction of $\varphi$ to $\ov{\H}_{2, 2}^1$ is generically finite over the divisor of curves with marked hyperelliptic conjugate points, and the class of such divisor is indeed given by \eqref{eq:phipfwdH122}  \cite{MR1953519} (see also \S\ref{eq:Logandiv}). The identity \eqref{eq:phipfwdE1} follows since the restriction of $\varphi$ to $E_{\{1\}}$ is generically finite over $\Delta_{0:2}$.
The classes in \eqref{eq:phipfwdH122} and \eqref{eq:phipfwdE1} are independent in $\mathrm{Pic}\left( \ov{\M}_{2,2}\right)$.

Hence, the classes of $\ov{\H}_{g, 2}^k$ and $E_{\{1\}}$ are independent for all $k\geq 1$ and $g\geq 2$. The statement for $n=2$ follows.

\smallskip

For $n \geq 3$, assume that
\begin{equation}
\label{eq:assumrel}
\alpha \left[\ov{\H}_{g, n}^k\right] + \sum_{\bm{I}}\alpha_{\bm{I}}\left[E_{\bm{I}}\right] = 0 \quad \in A^{n}\left(\P\E^k_{g,n} \right)
\end{equation}
for some coefficients $\alpha, \alpha_{\bm{I}}\in \mathbb{Q}$. 
After applying $(\pi_i)_*$ to \eqref{eq:assumrel} for some $1\leq i \leq n-1$ and using Lemma \ref{lemma:pushforwardEn}, we get 
 \[
 \left((c-1)\,\alpha + \alpha_{\{ i\}}\right)\left[\ov{\H}_{g, n-1}^k\right] +  c\sum_{\bm{I} \, : \,i \notin \bm{I}} \alpha_{\bm{I}}\left[E_{\ov{\bm{I}}}\right]=0
 \]
where $c := k(2g-2) - (n-2)$.
The loci $E_{\ov{\bm{I}}}\subset \P\E^k_{g,n-1}$ thus obtained are precisely the extra loci at the $(n-1)$-th step as in \eqref{eq:Lnminus1}.
By the inductive assumption, we  deduce  
\begin{align*}
(c-1)\,\alpha + \alpha_{\{ i\}} &= 0 & \mbox{and} && \alpha_{\bm{I}} &=0 &\mbox{for all } \bm{I} \, : \, i \notin \bm{I}.
\end{align*}
Applying $(\pi_i)_*$ to \eqref{eq:assumrel} for all $1 \leq i \leq n-1$ shows that $\alpha_{\bm{I}} = 0$ in all cases except possibly for $\bm{I} = \{1, \dots, n-1\}$, and thus $\alpha = 0$ as well. 

Hence, \eqref{eq:assumrel} reduces to $\alpha_{\bm{I}} \left[E_{\bm{I}}\right] = 0$ where $\bm{I} = \{1, \dots, n-1\}$. 
Restricting  to the test space from Lemma \ref{lemma:claimtestndimspace}, we deduce $\alpha_{\bm{I}} = 0$ as well. 
\end{proof}

\subsection{Proof of Theorems \ref{thm:pinZnrhonZ1}, \ref{thm:pinZnrhonZ1plus}, and Remark \ref{rmk:1}}

\begin{proof}[Proof of Theorem \ref{thm:pinZnrhonZ1}] 
We proceed by induction on $n$. For the base case $n=2$, the statement holds by Theorem \ref{thm:classZ2}. For the inductive step, 
assume \eqref{eq:hypkgn}, and consider the intersection \eqref{eq:pinrhon}.
Since this is generically transverse over $\M_{g,n}$,  the component $\ov{\H}_{g, n}^k$ contributes with multiplicity one to the intersection for all $k\geq 1$.
From Proposition \ref{prop:EnonlyEI}, we know that 
\[
\mathsf{L}_n = \left[\ov{\H}_{g, n}^k \right] + \sum_{\bm{I}} a_{n,\bm{I}}  \left[E_{\bm{I}}\right] 
\quad \in A^{n}\left( \P\E^k_{g,n}\big|_{\M_{g,n}^\mathsf{rt}}  \right)
\]
for some coefficients $a_{n,\bm{I}}\in \mathbb{Q}$. Using the exact sequence 
\[
A^{n}\left( \P\E^k_{g,n}\big|_{\ov{\M}_{g,n}\setminus \M_{g,n}^\mathsf{rt}}  \right) \rightarrow A^{n}\left( \P\E^k_{g,n} \right) \rightarrow A^{n}\left( \P\E^k_{g,n}\big|_{\M_{g,n}^\mathsf{rt}}  \right) \rightarrow 0
\]
it follows that there exists a cycle $\mathsf{B}_n$ not supported over $\M_{g,n}^\mathsf{rt}$ such that
\begin{equation}
\label{eq:LnisHplusaE}
\mathsf{L}_n = \left[\ov{\H}_{g, n}^k \right] + \sum_{\bm{I}} a_{n,\bm{I}}  \left[E_{\bm{I}}\right] + \mathsf{B}_n
\quad \in A^{n}\left( \P\E^k_{g,n} \right).
\end{equation}
Moreover, from Proposition \ref{prop:EnonlyRT}, one has $\mathsf{B}_n=0$ when $n\leq k$.

It remains to show that $a_{n,\bm{I}} = |\bm{I}|$ for all $n$ and $\bm{I}$.  
Applying $(\pi_i)_*$ to \eqref{eq:LnisHplusaE} for some $1\leq i \leq n-1$ and using Lemma \ref{lemma:pushforwardEn} gives 
\[ 
c\,\mathsf{L}_{n-1} = \left((c-1) + a_{n,\{i\}} \right)\left[ \ov{\H}_{g, n-1}^k\right] + c \sum_{\bm{I}\,:\, i \notin \bm{I}} a_{n,\bm{I}} \left[E_{\ov{\bm{I}}}\right]
+(\pi_i)_* \, \mathsf{B}_n
\]
where $c := k(2g-2) - (n-2)$.
By induction, this is equal to 
\[
c\,\mathsf{L}_{n-1} = c\left[\ov{\H}_{g, n-1}^k\right] + c\,\sum_{\ov{\bm{I}}}  \left|\ov{\bm{I}}\right| \left[E_{\ov{\bm{I}}}\right] +c\, \mathsf{B}_{n-1} \quad \in A^{n-1}\left( \P\E^k_{g,n-1}\right)
\] 
where the loci $E_{\ov{\bm{I}}}\subset \P\E^k_{g,n-1}$  are as in \eqref{eq:Lnminus1}. By Lemma \ref{lemma:ZnandEIindependent}, we may determine $a_{n,\bm{I}}$ by simply comparing coefficients. 
We deduce  
\[
(c-1) + a_{n,\{i\}} = c \qquad\mbox{and}\qquad a_{n,\bm{I}} = \left|\ov{\bm{I}}\right| \quad \mbox{for $\bm{I}\,:\, i \notin \bm{I}$.} 
\]
One has $\left|\ov{\bm{I}}\right| = \left|{\bm{I}}\right|$ for $\bm{I}$ with $i \notin \bm{I}$.
Applying $(\pi_i)_*$ to \eqref{eq:LnisHplusaE} for all $1 \leq i \leq n-1$ shows that $a_{n,\bm{I}} =\left|\bm{I}\right|$ in all cases except possibly for $\bm{I} = \{1, \dots, n-1\}$.

To find $a_{n,\bm{I}}$ with $\bm{I} = \{1, \dots, n-1\}$, consider the $n$-dimensional test space $T$ from Lemma \ref{lemma:claimtestndimspace}.
The restriction of \eqref{eq:LnisHplusaE} to $T$ gives
\[
k(2g-2) = T \cdot \mathsf{L}_n = (c-1) + a_{n,\{1, \dots, n-1\}}.
\] 
Thus, $a_{n,\{1, \dots, n-1 \}} = n-1$. 
 Theorem \ref{thm:pinZnrhonZ1} follows. 
\end{proof}

\begin{proof}[Proof of Theorem \ref{thm:pinZnrhonZ1plus}]
The only case left to be discussed is $k=2$ and $n=4g-5$, since all other cases are covered by Theorem \ref{thm:pinZnrhonZ1}. 
For this case, the hypothesis \eqref{eq:hypkgn} does not hold, and Lemma \ref{lemma:hypkgn} fails.
Indeed,  as discussed after Lemma \ref{lemma:hypkgn}, the locus $\ov{\H}^2_{g,\bm{m}} \subset \P\E^2_{g,1}$ with $\bm{m}=\left(4g-6\right)$ from \eqref{eq:counterexk2}
consists of two equidimensional components. Consequently, the extra locus $E_{\bm{I}}:=\gamma_* \,\ov{\H}^2_{g,\bm{m}}$ with $\bm{I}=\{1,\dots,n-1\}$ from \eqref{eq:EI}
decomposes as
\begin{equation}
\label{eq:EIdec2comp}
E_{\bm{I}} = E_{\bm{I}}' \cup E_{\bm{I}}^{\textrm{ab}} \quad \subset \P\E^2_{g,4g-5}
\end{equation}
with $E_{\bm{I}}'$ parametrizing quadratic differentials which  generically are \textit{not squa\-res} of abelian differentials, and $E_{\bm{I}}^{\textrm{ab}}$ parametrizing \textit{squares} of abelian differentials. For both components, the general element has a genus $g$ component and a rational tail containing all marked points, and the quadratic differential  vanishes with order exactly $4g-6$ at the preimage of the node on the genus $g$ component.

Consider the intersection \eqref{eq:pinrhon} for $k=2$ and $n=4g-5$.
The argument in Lemma \ref{lemma:EIinsupportEn} shows that the component $E_{\bm{I}}'$ is inside \eqref{eq:pinrhon} but not in $\ov{\H}^2_{g,4g-5}$. Next, we argue that this holds for the component $E_{\bm{I}}^{\textrm{ab}}$ as well.

Consider first the case $g=2$. To show that 
\begin{equation}
\label{eq:EIabinpi3H22}
E_{\bm{I}}^{\textrm{ab}}\subset\pi_{3}^{-1} \left(\ov{\H}_{2, 2}^2 \right), 
\end{equation}
it is enough to show that $\widetilde{E}_{\bm{I}}^{\textrm{ab}}:= \pi_{3}\left(E_{\bm{I}}^{\textrm{ab}}\right)\subset \ov{\H}_{2, 2}^2$.
For this, let $\widehat{E}_{\bm{I}}^{\textrm{ab}}\subset \P\E^2_{2,4}$ be the closure of the locus consisting of elements of $\P\E^2_{2,4}$ having a genus $2$ component with two rational tails, and the square of an abelian differential vanishing at both nodes on the genus $2$ component. In particular, the nodes on the genus $2$ component are hyperelliptic conjugates, and each rational tail has two marked points. We argue that one has an inclusion $\widehat{\iota}$ as in the following diagram:
\[
\begin{tikzcd}
\widehat{E}_{\bm{I}}^{\textrm{ab}} \arrow[rightarrow, hook]{r}{\widehat{\iota}} \arrow[rightarrow]{d} & \ov{\H}_{2, 4}^2 \arrow[rightarrow]{d} \arrow[rightarrow, hook]{r} & \P\E^2_{2,4} \arrow[rightarrow]{d}{\pi_{3,4}}\\
\widetilde{E}_{\bm{I}}^{\textrm{ab}} \arrow[rightarrow, dashed, hook]{r}{\widetilde{\iota}} & \ov{\H}_{2, 2}^2 \arrow[rightarrow, hook]{r} & \P\E^2_{2,2}.
\end{tikzcd}
\]
The vertical maps are induced by the map $\pi_{3,4}$ obtained by forgetting the last two marked points. The inclusion $\widehat{\iota}$ induces the desired inclusion, denoted $\widetilde{\iota}$ in the diagram.

To prove that $\widehat{\iota}$ is an inclusion, we need to show that on the general element $(C,P_1, \dots,P_4, \mu)$ of $\widehat{E}_{\bm{I}}^{\textrm{ab}}$, there exists a twisted quadratic differential of type $\mu$  satisfying the conditions of the incidence variety compactification. For this, we consider a twisted quadratic differential $\tau$ which agrees with $\mu$ on the genus $2$ component of $C$ (hence has zeros of order two at the preimage of each node on the genus $2$ component), and on each rational component, has a pole of order $6$ at the node and simple zeros at the two marked points. Let $\Gamma$ be the dual graph of $C$, and consider the order on $\Gamma$ assigning the maximum level to the genus $2$ vertex, and \textit{equal} lower level to the two rational vertices, as suggested by this figure:
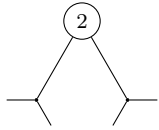
\[
\begin{tikzpicture}[baseline={([yshift=-.5ex]current bounding box.center)}]
      \path(0,0) ellipse (2 and 2);
      \tikzstyle{level 1}=[counterclockwise from=-60,level distance=15mm,sibling angle=120]
      \node [draw,circle,inner sep=2.5] (A0) at (90:2) {$\scriptstyle{2}$};
      \tikzstyle{level 1}=[clockwise from=0,level distance=12mm,sibling angle=120]
      \node [draw,circle,fill] (A1) at (-30:2) {$\scriptstyle{}$}
            child {node [label=-30: {$\scriptstyle{}$}]{}}
	    child {node [label=-150: {$\scriptstyle{}$}]{}};
      \tikzstyle{level 1}=[clockwise from=-60,level distance=12mm,sibling angle=120]
      \node [draw,circle,fill] (A2) at (210:2) {$\scriptstyle{}$}
            child {node [label=-30: {$\scriptstyle{}$}]{}}
	    child {node [label=-150: {$\scriptstyle{}$}]{}};
      \path (A0) edge [] node[]{} (A1);
      \path (A0) edge [] node[]{} (A2);
    \end{tikzpicture}
\]
Then $\tau$ is compatible with this order of $\Gamma$, and satisfies the conditions required by the incidence variety compactification. In particular, the global $2$-residue condition holds. Note that there is no quadratic differential on a rational curve with a pole of order $6$, with two simple zeros, and with zero $2$-residue at the pole (see \cite[\S 3.1]{chen2021towards}, and reference therein to \cite[Thm 1.10]{gendron2017diff}). However, the condition \cite[Def.~1.4(4)(v)]{bcggm} holds, since one can assume that $\tau$ has equal $2$-residues at the nodes on the two rational tails. Hence the global $2$-residue condition is satisfied. It follows that $\widehat{\iota}$  is an inclusion, hence $\widetilde{\iota}$ is an inclusion, and \eqref{eq:EIabinpi3H22} holds. 

For $g\geq 3$, the argument is quite similar. One considers a locus $\widehat{E}_{\bm{I}}^{\textrm{ab}}$ in $\P\E^2_{g,4g-4}$ as for $g=2$, defined as the closure of the locus consisting of elements of $\P\E^2_{g,4g-4}$ having a genus $g$ component with a rational tail with $2$ marked points and a rational tail with $4g-6$ marked points, and the square of an abelian differential vanishing at the nodes on the genus $2$ component with order $2$ and $4g-6$, respectively. To show that $\widehat{E}_{\bm{I}}^{\textrm{ab}}$ is in $\ov{\H}_{g, 4g-4}^2$, one argues that there exists a twisted quadratic differential $\tau$ on a general element $(C,P_1,\dots,P_{4g-4},\mu)$ of $\widehat{E}_{\bm{I}}^{\textrm{ab}}$ satisfying the conditions of the incidence variety compactification. The restriction of such a $\tau$ to the two rational tails has poles of order $6$ and $4g-2$ at the two nodes, respectively, and simple zeros at all marked points. One checks that the conditions of the incidence variety compactification are satisfied  with respect to the \textit{three-level} order on the dual graph of $C$ suggested by this figure:
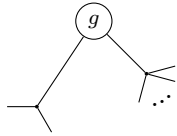
\[
\begin{tikzpicture}[baseline={([yshift=-.5ex]current bounding box.center)}]
      \path(0,0) ellipse (2 and 2);
      \tikzstyle{level 1}=[counterclockwise from=-60,level distance=15mm,sibling angle=120]
      \node [draw,circle,inner sep=2.5] (A0) at (90:2) {$\scriptstyle{g}$};
      \tikzstyle{level 1}=[clockwise from=15,level distance=12mm,sibling angle=120]
      \node [draw,circle,fill] (A1) at (0:2) {$\scriptstyle{}$}
            child {node [label=-30: {$\scriptstyle{}$}]{}}
	    child {node [label=90: {$\scriptstyle{}$}]{}}
	    child [grow=-15]{node{}}
	    child [grow=-45] {node[draw=none] {${.}$} edge from parent[draw=none]}
	    child [grow=-60] {node[draw=none] {${.}$} edge from parent[draw=none]}
	    child [grow=-75] {node[draw=none] {${.}$} edge from parent[draw=none]};
      \tikzstyle{level 1}=[clockwise from=-60,level distance=12mm,sibling angle=120]
      \node [draw,circle,fill] (A2) at (210:2.5) {$\scriptstyle{}$}
            child {node [label=-30: {$\scriptstyle{}$}]{}}
	    child {node [label=-150: {$\scriptstyle{}$}]{}};
      \path (A0) edge [] node[]{} (A1);
      \path (A0) edge [] node[]{} (A2);
    \end{tikzpicture}
\]
Contrary to the case $g=2$, the two rational tails are not at the same level.  Since the aspect of such a $\tau$ on the rational tail with more marked points has zero $2$-residue at the pole, the condition \cite[Def.~1.4(4)(v)]{bcggm} holds.

Showing that $E_{\bm{I}}^{\textrm{ab}}$ is in $\rho_{4g-5}^{-1} \left(\ov{\H}_{g, 1}^2 \right)$ and not in $\ov{\H}^2_{g,4g-5}$ for $g\geq 2$ is easier and similar to the arguments used for Lemma \ref{lemma:EIinsupportEn}. Then the proof continues in a similar way to the proof of Theorem \ref{thm:pinZnrhonZ1}. We sketch here the main points.
The arguments used for Proposition \ref{prop:EnonlyEI} and Lemma \ref{lemma:pushforwardEn} remain valid for this case. 
The test space $T$ from Lemma \ref{lemma:claimtestndimspace} can be replaced by two analogous $n$-dimensional test spaces $T_1,T_2\subset \P\E^2_{g,n}$ such that all quadratic differentials in $T_1$ are not squares of abelian differentials, and all quadratic differentials in $T_2$ are squares of abelian differentials.
Using these two test families, one shows that the class of $\ov{\H}_{g, 4g-5}^2$, the classes of the loci $E_{\bm{I}}$ for all $|\bm{I}|<n-1$, and the classes of each of the two components from \eqref{eq:EIdec2comp} are independent in $A^{4g-5}\left(\P\E^2_{g,4g-5} \right)$, as in Lemma \ref{lemma:ZnandEIindependent}. Finally, one shows that $E_{\bm{I}}$ is included in \eqref{eq:pinrhon} with multiplicity equal to $|\bm{I}|$ for all $\bm{I}$ as in the final steps of the proof of Theorem \ref{thm:pinZnrhonZ1}, whence the statement.
\end{proof}

\begin{remark}
\label{rmk:1_app}
When $k\geq 2$ and $n=k(2g-2)$,
the locus $\ov{\H}^k_{g,\bm{m}} \subset \P\E^k_{g,1}$ with $\bm{m}=\left(k(2g-2)-1\right)$ from \eqref{eq:counterexk3}
consists of two equidimensional components. Consequently, the extra locus $E := E_{\bm{I}}=\gamma_* \,\ov{\H}^k_{g,\bm{m}}$ with $\bm{I}=\{1,\dots,n-1\}$ from \eqref{eq:EI}
decomposes as
\begin{equation*}
E = E' \cup E^{\textrm{ab}} \quad \subset \P\E^k_{g,k(2g-2)}
\end{equation*}
such that the general element of $E'$ has a genus $g$ component, a rational tail containing all marked points, and a $k$-differential vanishing with order exactly $k(2g-2)-1$ at the preimage of the node on the genus $g$ component; and the general element of $E^{\textrm{ab}}$ has a genus $g$ component, a rational tail containing all marked points, and the $k$-th power of an abelian differential vanishing with order $k(2g-2)$ at the preimage of the node on the genus $g$ component.
As in the proof of Theorem \ref{thm:pinZnrhonZ1plus},  both components of $E$ are contained in the intersection \eqref{eq:pinrhon}, although $E'$ appears with multiplicity $|\bm{I}|=k(2g-2)-1$, while $E^{\textrm{ab}}$  with multiplicity $k(2g-2)$. Hence \eqref{eq:k2plusnmax} holds.
\end{remark}

\bibliographystyle{alphanumN}
\bibliography{Biblio}

\end{document}